\documentclass[runningheads]{llncs}

\newcommand{\R}{\mathbb{R}}

\usepackage{latexsym, amssymb, amsfonts,amsmath}
\usepackage{graphicx}
\usepackage[all,knot,arc,import,poly]{xy}
\begin{document}
\title{A Universal Algebraic Survey of $\mathcal{C}^{\infty}-$Rings}
\author{Jean Cerqueira Berni\inst{1} \and
Hugo Luiz Mariano\inst{2}}
\authorrunning{J. C. Berni and H. L. Mariano}
%
\institute{Institute of Mathematics and Statistics, University of S\~{a}o Paulo,
Rua do Mat\~{a}o, 1010, S\~{a}o Paulo - SP, Brazil. \email{jeancb@ime.usp.br} \and
Institute of Mathematics and Statistics, University of S\~{a}o Paulo,
Rua do Mat\~{a}o, 1010, S\~{a}o Paulo - SP, Brazil, \email{hugomar@ime.usp.br}}
\maketitle 

\begin{abstract}
	In this paper we present some basic results of the Universal Algebra of $\mathcal{C}^{\infty}-$rings which were nowhere to be found in the current literature. The outstanding book of I. Moerdijk and G. Reyes, \cite{MSIA}, presents the basic (and advanced) facts about $\mathcal{C}^{\infty}-$rings, however such a presentation has no universal algebraic ``flavour''. We have been inspired to describe $\mathcal{C}^{\infty}-$rings through this viewpoint by D. Joyce in \cite{Joyce}. Our main goal here is to provide a comprehensive  material with detailed proofs of many known ``taken for granted''  results and constructions used in the literature about $\mathcal{C}^{\infty}-$rings and their applications - such proofs either could not be found or were merely sketched. We present, in detail, the main constructions one can perform within this category, such as limits, products, homomorphic images, quotients, directed colimits, free objects and others, providing a ``propaedeutic exposition'' for the reader's benefit. 	
\end{abstract}

\section*{Introduction}

\hspace{0.6cm} As observed by E. Dubuc in \cite{Dubucs}, a $\mathcal{C}^{\infty}-$ring is a model of the algebraic theory which has as $n-$ary operations all the smooth functions from $\R^n$ into $\R$, and whose axioms are all the equations that hold between these functions. Since every polynomial is smooth, all $\mathcal{C}^{\infty}-$rings are, in particular, $\mathbb{R}-$algebras. This point of view allows us to regard this theory as an extension of the concept of $\R-$algebra.\\

The theory of $\mathcal{C}^{\infty}-$rings has been originally  studied in view of its applications to Singularity Theory and in order to construct topos-models for Synthetic Differential Geometry (the Dubuc Topos, for instance), which ``grew out of ideas of Lawvere in the 1960s'' (cf. \cite{Joyce}). Recently, however, this theory has been explored by some eminent mathematicians like David I. Spivak and Dominic Joyce in order to extend Jacob Lurie's program of Derived Algebraic Geometry to Derived Differential Geometry (cf. \cite{Joyce} and \cite{Joyce2} ).\\

Just as any Lawvere theory, $\mathcal{C}^{\infty}-$rings can be interpreted within any topos. In this specific case, a $\mathcal{C}^{\infty}-$ring in a topos $\mathcal{E}$ is a finite product preserving functor from the category whose objects are the Cartesian products of $\mathbb{R}$ and whose morphisms are the smooth functions between them, into $\mathcal{E}$  (see, for example, \cite{rings1}). In this work, however, we focus on set-theoretic $\mathcal{C}^{\infty}-$rings (i.e., $\mathcal{C}^{\infty}-$rings in the topos of sets, ${\rm \bf Set}$).\\

This is the first paper of a series that presents a detailed account of some algebraic aspects of $\mathcal{C}^{\infty}-$rings. Some categorial and logical aspects were given in \cite{SAJL}. In the next papers we are going to present some topics on Smooth Commutative Algebra and von Neumann regular $\mathcal{C}^{\infty}-$rings. Here we present and analyze a $\mathcal{C}^{\infty}-$ring as a universal algebra whose functional symbols are the symbols for all smooth functions from Cartesian powers of $\mathbb{R}$ to $\mathbb{R}$. Such an approach emphasizes the power of a $\mathcal{C}^{\infty}-$ring in interpreting a broader language than the algebraic one, which is expressed in terms of the $\mathbb{R}-$algebra structure. It also has the advantage of giving us explicitly many constructions, such as products, coproducts, directed colimits, among others, as well as simpler proofs of the main results.\\

We make a detailed exposition of the description of free $\mathcal{C}^{\infty}-$rings in terms of a colimit, and we use it to account the often used description of an arbitrary $\mathcal{C}^{\infty}-$rings in terms of generators and relations.\\

Our idea of describing $\mathcal{C}^{\infty}-$rings from a universal-algebraic point of view  was mainly inpired by the clear and elegant presentation made by Dominic Joyce in \cite{Joyce} - which we found very enlightening.\\

\textbf{Overview of the paper:} We begin by presenting the equational theory of $\mathcal{C}^{\infty}-$rings in terms of a first order language with a denumerable set of variables. We define the class of $\mathcal{C}^{\infty}-$structures and the (equationally defined) subclass of $\mathcal{C}^{\infty}-$rings.\\

In the \textbf{Section \ref{main}} we present a detailed description of the main constructions involving $\mathcal{C}^{\infty}-$rings: $\mathcal{C}^{\infty}-$subrings (\textbf{Definition \ref{sub}}), intersections (\textbf{Proposition \ref{Zimbabue}}), the $\mathcal{C}^{\infty}-$subring generated by a set (\textbf{Definition \ref{gen}}), the directed union of $\mathcal{C}^{\infty}-$rings (\textbf{Proposition \ref{xiquexique}}), products (\textbf{Definition \ref{produto}}), $\mathcal{C}^{\infty}-$con\-gruences (\textbf{Definition \ref{conga}}) and quotients (\textbf{Definition \ref{quot}}), homomorphic images (\textbf{Propo\-sition \ref{H}}), directed colimits (\textbf{Theorem \ref{colimite}}) and small projective limits (\textbf{Theorem \ref{projlim}}). We present and prove the ``Fundamental Theorem of $\mathcal{C}^{\infty}-$Homomorphism'' (\textbf{Theorem \ref{FTH}}). We prove also that the category of $\mathcal{C}^{\infty}-$rings is a reflective subcategory of the category of all $\mathcal{C}^{\infty}-$structures (\textbf{Theorem \ref{adjQV-le}}).\\

We dedicate \textbf{Section \ref{Petrus}} to  describe the free $\mathcal{C}^{\infty}-$rings, first with a finite set of generators and then with an arbitrary set of generators. We use this construction in order to describe an adjunction between the category of all $\mathcal{C}^{\infty}-$rings and $\mathcal{C}^{\infty}-$homomorphisms, $\mathcal{C}^{\infty}{\rm \bf Rng}$ and the category of sets, ${\rm \bf Set}$ (\textbf{Proposition \ref{Julia}}).\\

In \textbf{Section \ref{oc}} we describe other constructions, and in \textbf{Subsection \ref{coquette}} we use Hadamard's lemma (\textbf{Theorem \ref{hadamard}}) to prove that the ring-theoretic ideals of any finitely generated $\mathcal{C}^{\infty}-$ring classify their congruences (\textbf{Proposition \ref{Cemil}}). We extend this result by presenting a proof for the general case (\textbf{Proposition \ref{Meirelles}}).\\

We prove that any $\mathcal{C}^{\infty}-$ring can be expressed as a directed colimit of finitely generated $\mathcal{C}^{\infty}-$rings (\textbf{Theorem \ref{LimMono}}) and in \textbf{Subsection \ref{kop}} we present an explicit description for the $\mathcal{C}^{\infty}-$coproduct of $\mathcal{C}^{\infty}-$rings (\textbf{Definition \ref{cop}}). We end up this work presenting an ubiquitous construction in Algebra in the \textbf{Subsection \ref{pol}}, namely the $\mathcal{C}^{\infty}-$ring of $\mathcal{C}^{\infty}-$polynomials, constructed in terms of the $\mathcal{C}^{\infty}-$product. Such a construction will play an important role in the future papers on Smooth Commutative Algebra and von Neumann Regular $\mathcal{C}^{\infty}-$rings.\\


\section{Preliminaries: The equational theory of $C^\infty$-rings}

\hspace{0.5cm}The theory of $\mathcal{C}^{\infty}-$rings can be described within a first order language $\mathcal{L}$ with a denumerable set of variables (${\rm \bf Var}(\mathcal{L}) = \{ x_1, x_2, \cdots, x_n, \cdots\}$) whose nonlogical symbols are the symbols of $\mathcal{C}^{\infty}-$functions from $\mathbb{R}^m$ to $\mathbb{R}^n$, with $m,n \in \mathbb{N}$, \textit{i.e.}, the non-logical symbols consist only of function symbols, described as follows:\\

For each $n \in \mathbb{N}$, the $n-$ary \textbf{function symbols} of the set $\mathcal{C}^{\infty}(\mathbb{R}^n, \mathbb{R})$, \textit{i.e.}, $\mathcal{F}_{(n)} = \{ f^{(n)} | f \in \mathcal{C}^{\infty}(\mathbb{R}^n, \mathbb{R})\}$. So the set of function symbols of our language is given by:
      $$\mathcal{F} = \bigcup_{n \in \mathbb{N}} \mathcal{F}_{(n)} = \bigcup_{n \in \mathbb{N}} \mathcal{C}^{\infty}(\mathbb{R}^n)$$
      Note that our set of constants is $\mathbb{R}$, since it can be identified with the set of all $0-$ary function symbols, \textit{i.e.}, ${\rm \bf Const}(\mathcal{L}) = \mathcal{F}_{(0)} = \mathcal{C}^{\infty}(\mathbb{R}^0) \cong \mathcal{C}^{\infty}(\{ *\}) \cong \mathbb{R}$.\\

The terms of this language are defined, in the usual way, as the smallest set which comprises the individual variables, constant symbols and $n-$ary function symbols followed by $n$ terms ($n \in \mathbb{N}$).\\

\begin{definition} \label{cabala} A \textbf{$\mathcal{C}^{\infty}-$structure} on a set $A$ is a pair $ \mathfrak{A} =(A,\Phi)$, where:

 $$\begin{array}{cccc}
 \Phi: & \bigcup_{n \in \mathbb{N}} \mathcal{C}^{\infty}(\mathbb{R}^n, \mathbb{R})& \rightarrow & \bigcup_{n \in \mathbb{N}} {\rm Func}\,(A^n; A)\\
      & (f: \mathbb{R}^n \stackrel{\mathcal{C}^{\infty}}{\to} \mathbb{R}) & \mapsto & \Phi(f) := (f^{A}: A^n \to A)
 \end{array},$$
that is, $\Phi$ interprets the \textbf{symbols}\footnote{here considered simply as syntactic symbols rather than functions.} of all smooth real functions of $n$ variables as $n-$ary function symbols on $A$.
\end{definition}

We call a $\mathcal{C}^{\infty}-$struture $\mathfrak{A} = (A, \Phi)$ a \textbf{$\mathcal{C}^{\infty}-$ring} if it preserves  projections and all equations between smooth functions. We have the following:

\begin{definition}\label{CravoeCanela}Let $\mathfrak{A}=(A,\Phi)$ be a $\mathcal{C}^{\infty}-$structure. We say that $\mathfrak{A}$ (or, when there is no danger of confusion, $A$) is a \textbf{$\mathcal{C}^{\infty}-$ring} if the following is true:\\

$\bullet$ Given any $n,k \in \mathbb{N}$ and any projection $p_k: \mathbb{R}^n \to \mathbb{R}$, we have:

$$\mathfrak{A} \models (\forall x_1)\cdots (\forall x_n)(p_k(x_1, \cdots, x_n)=x_k)$$

$\bullet$ For every $f, g_1, \cdots g_n \in \mathcal{C}^{\infty}(\mathbb{R}^m, \mathbb{R})$ with $m,n \in \mathbb{N}$, and every $h \in \mathcal{C}^{\infty}(\mathbb{R}^n, \mathbb{R})$ such that $f = h \circ (g_1, \cdots, g_n)$, one has:
$$\mathfrak{A} \models (\forall x_1)\cdots (\forall x_m)(f(x_1, \cdots, x_m)=h(g(x_1, \cdots, x_m), \cdots, g_n(x_1, \cdots, x_m)))$$
\end{definition}

\begin{definition}Let $(A, \Phi)$ and $(B,\Psi)$ be two $\mathcal{C}^{\infty}-$rings. A function $\varphi: A \to B$ is called a \textbf{morphism of $\mathcal{C}^{\infty}-$rings} or \textbf{$\mathcal{C}^{\infty}$-homomorphism} if for any $n \in \mathbb{N}$ and any $f: \mathbb{R}^n \stackrel{\mathcal{C}^{\infty}}{\to} \mathbb{R}$ the following diagram commutes:
$$\xymatrixcolsep{5pc}\xymatrix{
A^n \ar[d]_{\Phi(f)}\ar[r]^{\varphi^{(n)}} & B^n \ar[d]^{\Psi(f)}\\
A \ar[r]^{\varphi^{}} & B
}$$
 \textit{i.e.}, $\Psi(f) \circ \varphi^{(n)} = \varphi^{} \circ \Phi(f)$.
\end{definition}

Since $\mathcal{L}$ does not contain any relational symbol, the set of the \textbf{atomic formulas}, \textbf{AF}, is given simply by the equality between terms, that is
$${\rm \bf AF} = \{ t_1 = t_2 | t_1,t_2 \in \mathbb{T}\}$$

Finally, the \textbf{well formed formulas}, \textbf{WFF} are constructed as one usually does in any first order theory.\\

One possible set of axioms for the theory of the $\mathcal{C}^{\infty}$-rings can be given by the following two sets of equations:

\begin{itemize}
  \item[\textbf{(E1)}]{For each $n \in \mathbb{N}$ and for every $k\leq n$, denoting the $k$-th projection by $p_k: \mathbb{R}^n \to \mathbb{R}$, the equations:
      $${\rm \bf Eq}_{(1)}^{n,k} = \{ (\forall x_1)\cdots(\forall x_n)(p_k(x_1, \cdots, x_n) = x_k)\}$$}
  \item[\textbf{(E2)}]{for every $k,n \in \mathbb{N}$ and for every $(n+2)-$tuple of function symbols, $(f,g_1, \cdots,g_n,h)$ such that $f \in \mathcal{F}_{(n)}$, $g_1, \cdots, g_n, h \in \mathcal{F}_{(k)}$ and $h = f\circ (g_1, \cdots, g_n)$, the equations:
     \begin{multline*}{\rm \bf Eq}_{(2)}^{n,k} = \\
     =\{(\forall x_1)\cdots(\forall x_k)(h(x_1, \cdots,x_k) = f(g_1(x_1, \cdots, x_k), \cdots, g_n(x_1, \cdots, x_k)))\}\end{multline*}}
\end{itemize}

The class of $\mathcal{C}^{\infty}-$ring is, thus, equational, and as such it a

Now we present an alternative description of a $\mathcal{C}^{\infty}-$rings making use of universal algebraic concepts,inspired by \cite{Joyce}.\\

We first define the notion of a \index{$\mathcal{C}^{\infty}-$structure}$\mathcal{C}^{\infty}-$structure:\\

\begin{definition}\label{cabala} A \textbf{$\mathcal{C}^{\infty}-$structure} on a set $A$ is a pair $(A,\Phi)$, where:

$$\begin{array}{cccc}
\Phi: & \bigcup_{n \in \mathbb{N}} \mathcal{C}^{\infty}(\mathbb{R}^n, \mathbb{R})& \rightarrow & \bigcup_{n \in \mathbb{N}} {\rm Func}\,(A^n; A)\\
      & (f: \mathbb{R}^n \stackrel{\mathcal{C}^{\infty}}{\to} \mathbb{R}) & \mapsto & \Phi(f) := (f^{A}: A^n \to A)
\end{array},$$

\hspace{-0.51cm}that is, $\Phi$ is a function that interprets the \textbf{symbols}\footnote{here considered simply as syntactic symbols rather than functions.} of all smooth real functions of $n$ variables as $n-$ary function symbols on $A$.
\end{definition}


\begin{definition}Let $(A, \Phi)$ and $(B,\Psi)$ be two $\mathcal{C}^{\infty}-$structures. A function $\varphi: A \to B$ is called a \index{$\mathcal{C}^{\infty}-$structure!morphism of $\mathcal{C}^{\infty}-$structures}\textbf{morphism of $\mathcal{C}^{\infty}-$structures} (or, simply, a $\mathcal{C}^{\infty}-$morphism) if for any $n \in \mathbb{N}$ and any $f \in \mathcal{C}^{\infty}\,(\mathbb{R}^n, \mathbb{R})$ the following diagram commutes:
$$\xymatrixcolsep{5pc}\xymatrix{
A^n \ar[d]_{\Phi(f)}\ar[r]^{\varphi^{(n)}} & B^n \ar[d]^{\Psi(f)}\\
A \ar[r]^{\varphi} & B
}$$
 \textit{i.e.}, $\Psi(f) \circ \varphi^{(n)} = \varphi \circ \Phi(f)$.
\end{definition}

\begin{theorem}Let $(A,\Phi), (B,\Psi), (C,\Omega)$ be any $\mathcal{C}^{\infty}-$structures, and let $\varphi: (A,\Phi) \to (B, \Psi)$ and $\psi: (B,\Psi) \to (C,\Omega)$ be two morphisms of $\mathcal{C}^{\infty}-$structures. We have:
\begin{itemize}
  \item[(1)]{${\rm id}_A: (A,\Phi) \to (A,\Phi)$ is a morphism of $\mathcal{C}^{\infty}-$structures;}
  \item[(2)]{$\psi \circ \varphi: (A,\Phi) \to (C,\Omega)$ is a morphism of $\mathcal{C}^{\infty}-$structures.}
\end{itemize}
\end{theorem}
\begin{proof}
Ad (1): Given any $n \in \mathbb{N}$ and any $f \in \mathcal{C}^{\infty}\,(\mathbb{R}^n,\mathbb{R})$, since ${\rm id_A}^{(n)} = {\rm id}_{A^n}$, we have:

$$\xymatrixcolsep{5pc}\xymatrix{
A^n \ar[d]^{\Phi(f)}\ar[r]^{{\rm id}_A^{(n)}} & A^n \ar[d]^{\Phi(f)}\\
A \ar[r]^{{{\rm id}_A}} & A
}$$

so $\Phi(f) \circ {{\rm id}_A}^{(n)} = \Phi(f) \circ {\rm id}_{A^n} = \Phi(f) = {{\rm id}_{A}} \circ \Phi(f)$, thus ${\rm id}_A: (A, \Phi) \to (A, \Phi)$ is a morphism of $\mathcal{C}^{\infty}-$structures.\\

Ad (2): Suppose that $\varphi: (A,\Phi) \to (B, \Psi)$ and $\psi: (B,\Psi) \to (C,\Omega)$ are two morphisms of $\mathcal{C}^{\infty}-$structures, so given any $m \in \mathbb{N}$ and any $f \in \mathcal{C}^{\infty}(\mathbb{R}^m, \mathbb{R})$ the following diagrams commute:

$$\begin{array}{cc}\xymatrixcolsep{5pc}\xymatrix{
A^m \ar[d]^{\Phi(f)}\ar[r]^{\varphi^{(m)}} & B^m \ar[d]^{\Psi(f)}\\
A \ar[r]^{\varphi} & B
} & \xymatrixcolsep{5pc}\xymatrix{
B^m \ar[d]^{\Psi(f)}\ar[r]^{\psi^{(m)}} & C^m \ar[d]^{\Omega(f)}\\
B \ar[r]^{\psi} & C
}\end{array}$$

Since $(\psi \circ \varphi)^{(m)} = \psi^{(m)} \circ \varphi^{(m)}$, the following diagram commutes:

$$\xymatrixcolsep{5pc}\xymatrix{
A^m \ar[d]^{\Phi(f)}\ar[r]^{(\psi \circ \varphi)^{(m)}} & C^m \ar[d]^{\Omega(f)}\\
A \ar[r]_{(\psi \circ \varphi)} & C
}$$

so $\psi \circ \varphi: (A,\Phi) \to (C, \Omega)$ is a morphism of $\mathcal{C}^{\infty}-$structures.
\end{proof}

\begin{theorem}Let $(A,\Phi), (B,\Psi), (C,\Omega), (D, \Gamma)$ be any $\mathcal{C}^{\infty}-$structures, and let $\varphi: (A,\Phi) \to (B, \Psi)$, $\psi: (B,\Psi) \to (C,\Omega)$ and $\nu: (C, \Omega) \to (D,\Gamma)$ be morphisms of $\mathcal{C}^{\infty}-$structures. We have the following equations between pairs of morphisms of $\mathcal{C}^{\infty}-$structures:

$$\nu \circ (\psi \circ \varphi) = (\nu \circ \psi) \circ \varphi;$$

$$\varphi \circ {\rm id}_A = {\rm id}_B \circ \varphi.$$

\end{theorem}
\begin{proof}
Since, as functions, we have:
$$\nu \circ (\psi \circ \varphi) = (\nu \circ \psi) \circ \varphi$$
and
$$\varphi \circ {\rm id}_A = {\rm id}_B \circ \varphi$$

and the composition of morphisms of $\mathcal{C}^{\infty}-$structures is again a morphism of $\mathcal{C}^{\infty}-$structures, it follows that the equality between morphisms of $\mathcal{C}^{\infty}-$structures holds.
\end{proof}

\begin{definition}We are going to denote by $\mathcal{C}^{\infty}{\rm \bf Str}$ the category whose objects are the $\mathcal{C}^{\infty}-$structures and whose morphisms are the morphisms of $\mathcal{C}^{\infty}-$structures.
\end{definition}

As a full subcategory of $\mathcal{C}^{\infty}{\rm \bf Str}$ we have the category of $\mathcal{C}^{\infty}-$rings, defined as follows:

\begin{definition}Let $(A,\Phi)$ be a $\mathcal{C}^{\infty}-$structure. We say that $(A,\Phi)$ is a \index{$\mathcal{C}^{\infty}-$ring}\textbf{$\mathcal{C}^{\infty}-$ring} if the following two conditions are fulfilled:
\begin{itemize}
  \item[(1)]{For each $n \in \mathbb{N}$ and for every $k\leq n$, denoting the $k$-th projection by $p_k: \mathbb{R}^n \to \mathbb{R}$, we have:
      $$\Phi(p_k) = \pi_k: A^n \to A,$$
      that is, each projection is interpreted as the canonical projection on the $k-$th coordinate, $\pi_k: A^n \to A$}
  \item[(2)]{for every $k,n \in \mathbb{N}$ and for every $(n+2)-$tuple of smooth functions: $$(f,g_1, \cdots,g_n,h)$$
  such that $f \in \mathcal{C}^{\infty}(\mathbb{R}^n, \mathbb{R})$, $g_1, \cdots, g_n, h \in \mathcal{C}^{\infty}(\mathbb{R}^k, \mathbb{R})$, we have:
      $$h = f\circ (g_1, \cdots, g_n) \Rightarrow \Phi(h) = \Phi(f)\circ (\Phi(g_1), \cdots, \Phi(g_n))$$}
\end{itemize}
Given two $\mathcal{C}^{\infty}-$rings $(A,\Phi)$ and $(B,\Psi)$, a \index{$\mathcal{C}^{\infty}-$ring!morphism of $\mathcal{C}^{\infty}-$rings}\textbf{morphism of $\mathcal{C}^{\infty}-$rings} from $(A,\Phi)$ to $(B,\Psi)$, or a \\
\textbf{$\mathcal{C}^{\infty}-$homo\-morphism} from $(A,\Phi)$ to $(B,\Psi)$, is simply a morphism of $\mathcal{C}^{\infty}-$struc\-tures from $(A,\Phi)$ to $(B,\Psi)$. We will denote the category of $\mathcal{C}^{\infty}-$rings and $\mathcal{C}^{\infty}-$homo\-morphisms by $\mathcal{C}^{\infty}{\rm \bf Rng}$.
\end{definition}

Thus, in the context of first order theories, a model of a  \textbf{$\mathcal{C}^{\infty}-$ring} is a $\mathcal{C}^{\infty}-$structure, $\mathfrak{A} = (A,\Phi)$ such that:

$\bullet$ For every $n \in \mathbb{N}$, $k \leq n$, denoting the projection on the $k-$th coordinate by $p_k: \mathbb{R}^n \to \mathbb{R}$:\\

$$\mathfrak{A} \models (\forall x_1)\cdots (\forall x_n)(p_k(x_1, \cdots, x_n)=x_k)$$ \\
that is, $\Phi(p_k)=\pi_k: A^n \to A$;\\

$\bullet$ For every $n,k \in \mathbb{N}$, $f \in \mathcal{C}^{\infty}(\mathbb{R}^n,\mathbb{R})$, $h,g_1, \cdots,g_n \in \mathcal{C}^{\infty}(\mathbb{R}^k,\mathbb{R})$ such that $h = f(g_1, \cdots, g_n)$:

$$\mathfrak{A} \models (\forall x_1)\cdots (\forall x_k) (h(x_1, \cdots,x_k) = f(g_1(x_1, \cdots, x_k), \cdots, g_n(x_1, \cdots, x_k))$$

that is, $\Phi(h) = \Phi(f)(\Phi(g_1), \cdots, \Phi(g_n))$.\\

As we are going to see later on, the category of $\mathcal{C}^{\infty}-$rings and its morphisms has many constructions, such as arbitrary products, coproducts, directed colimits, quotients and many others. It also ``extends'' the category of commutative unital rings, ${\rm \bf CRing}$, in the following sense:\\

\begin{remark}\label{Rock}Since the sum $+: \mathbb{R}^{2} \to \mathbb{R}$, the opposite, $-: \mathbb{R} \rightarrow \mathbb{R}$, $\cdot: \mathbb{R}^2 \to \mathbb{R}$ and the constant functions $0: \mathbb{R} \rightarrow \mathbb{R}$ and $1: \mathbb{R} \rightarrow \mathbb{R}$ are particular cases of $\mathcal{C}^{\infty}-$functions, any $\mathcal{C}^{\infty}-$ring $(A,\Phi)$ may be regarded as a commutative ring $(A,\Phi(+), \Phi(\cdot), \Phi(-), \Phi(0), \Phi(1))$, where:
$$\begin{array}{cccc}
    \Phi(+): & A \times A & \to & A \\
     & (a_1,a_2) & \mapsto & \Phi(+)(a_1,a_2)= a_1 + a_2
\end{array}$$

$$\begin{array}{cccc}
    \Phi(-): & A  & \to & A \\
     & a & \mapsto & \Phi(-)(a)= -a
\end{array}$$

$$\begin{array}{cccc}
    \Phi(0): & A^0  & \to & A \\
     & * & \mapsto & \Phi(0)
\end{array}$$

$$\begin{array}{cccc}
    \Phi(1): & A^0  & \to & A \\
     & * & \mapsto & \Phi(1)
\end{array}$$

where $A^0=\{ *\}$, and:

$$\begin{array}{cccc}
    \Phi(\cdot): & A \times A & \to & A \\
     & (a_1,a_2) & \mapsto & \Phi(\cdot)(a_1,a_2)= a_1 \cdot a_2
\end{array}$$

Thus, we have a forgetful functor:

$$\begin{array}{cccc}
\widetilde{U}: & \mathcal{C}^{\infty}{\rm \bf Rng} & \rightarrow & {\rm \bf CRing}\\
   & \xymatrix{(A\,\Phi) \ar[d]^{\varphi} \\ (B,\Psi)} & \mapsto & \xymatrix{(A, \Phi(+), \Phi(\cdot), \Phi(-), \Phi(0), \Phi(1)) \ar[d]^{\varphi \upharpoonright}\\
   (B, \Psi(+), \Psi(\cdot), \Psi(-), \Psi(0), \Psi(1))}
\end{array}$$

$$\widetilde{U}: \mathcal{C}^{\infty}{\rm \bf Rng} \to {\rm \bf CRing}$$
\end{remark}

Analogously, we can define a forgetful functor from the category of $\mathcal{C}^{\infty}-$rings and $\mathcal{C}^{\infty}-$homomorphisms into the category of commutative $\mathbb{R}-$algebras with unity,

$$\hat{U}: \mathcal{C}^{\infty}{\rm \bf Rng} \rightarrow \mathbb{R}-{\rm \bf Alg}$$

These functors will be analyzed with detail in the next paper, \cite{TSCA}.

\section{The main constructions in the category of $\mathcal{C}^{\infty}-$rings}\label{main}

\hspace{0.5cm}Since the theory of $\mathcal{C}^{\infty}-$rings is equational, the class $\mathcal{C}^{\infty}{\rm \bf Rng}$ is closed in $\mathcal{C}^{\infty}{\rm \bf Str}$ under many algebraic constructions, such as substructures, products, quotients, directed colimits and others. In this section we give explicit descriptions for some of these constructions.\\

\subsection{$\mathcal{C}^{\infty}-$Subrings}

We begin defining what we mean by a $\mathcal{C}^{\infty}-$subring.\\

\begin{definition}\label{sub} Let $(A,\Phi)$ be a $\mathcal{C}^{\infty}-$ring and let $B \subseteq A$. Under these circumstances, we say that $(B,\Phi')$ is a \index{$\mathcal{C}^{\infty}-$subring}$\mathcal{C}^{\infty}-$subring of $(A,\Phi)$ if, and only if, for any $n \in \mathbb{N}$, $f \in \mathcal{C}^{\infty}(\mathbb{R}^n,\mathbb{R})$ and any $(b_1, \cdots, b_n) \in B^n$ we have:
$$\Phi(f)(b_1, \cdots, b_n) \in B$$
That is to say that $B$ is closed under any $\mathcal{C}^{\infty}-$function $n$-ary symbol. Note that the $\mathcal{C}^{\infty}-$structure of $B$ is virtually the same as the $\mathcal{C}^{\infty}-$structure of $(A,\Phi)$, since they interpret every smooth function in the same way. However $\Phi'$ has a different codomain, as:
$$\begin{array}{cccc}
\Phi': & \bigcup_{n \in \mathbb{N}}\mathcal{C}^{\infty}\,(\mathbb{R}^n, \mathbb{R}) & \rightarrow & \bigcup_{n \in \mathbb{N}} {\rm Func}\,(B^n,B)\\
   & (\mathbb{R}^n \stackrel{f}{\rightarrow} \mathbb{R}) & \mapsto & (\Phi(f)\upharpoonright_{B^n}: B^n \rightarrow B)
\end{array}$$
\end{definition}

We observe that $\Phi'$ is the unique $\mathcal{C}^{\infty}-$structure such that the inclusion map:

$$\iota^{A}_{B}: B \hookrightarrow A$$

is a $\mathcal{C}^{\infty}-$homomorphism.\\

We are going to denote the class of all $\mathcal{C}^{\infty}-$subrings of a given $\mathcal{C}^{\infty}-$ring $(A,\Phi)$ by ${\rm Sub}\,(A,\Phi)$.\\

Next we prove that the intersection of any family of $\mathcal{C}^{\infty}-$subrings of a given $\mathcal{C}^{\infty}-$ring is, again, a $\mathcal{C}^{\infty}-$subring.\\

\begin{proposition}\label{Zimbabue}Let $\{ (A_{\alpha}, \Phi_{\alpha}) | \alpha \in \Lambda\}$ be a family of $\mathcal{C}^{\infty}-$subrings of $(A,\Phi)$, so

$$(\forall \alpha \in \Lambda)(\forall n \in \mathbb{N})(\forall f \in \mathcal{C}^{\infty}(\mathbb{R}^n, \mathbb{R}))(\Phi_{\alpha}(f) = \Phi(f)\upharpoonright_{{A_{\alpha}}^n}: {A_{\alpha}}^n \rightarrow A_{\alpha})$$

We have that:
$$\left( \bigcap_{\alpha \in \Lambda} A_{\alpha}, \Phi'\right)$$
where:
$$\begin{array}{cccc}
\Phi': & \bigcup_{n \in \mathbb{N}}\mathcal{C}^{\infty}(\mathbb{R}^n, \mathbb{R}) & \rightarrow & \bigcup_{n \in \mathbb{N}}{\rm Func}\, \left( \left( \bigcap_{\alpha \in \Lambda} A_{\alpha}\right)^n, \bigcap_{\alpha \in \Lambda} A_{\alpha}\right)\\
  & (\mathbb{R}^n \stackrel{f}{\rightarrow} \mathbb{R}) & \mapsto & \Phi(f)\upharpoonright_{\left(\bigcap_{\alpha \in \Lambda} A_{\alpha}\right)^n}: \left( \bigcap_{\alpha \in \Lambda} A_{\alpha}\right)^n \rightarrow \bigcap_{\alpha \in \Lambda} A_{\alpha}
\end{array}$$

is a $\mathcal{C}^{\infty}-$subring of $(A,\Phi)$.
\end{proposition}
\begin{proof} Let $n \in \mathbb{N}$, $f \in \mathcal{C}^{\infty}(\mathbb{R}^n, \mathbb{R})$ and $(a_1, \cdots, a_n) \in \left( \bigcap_{\alpha \in \Lambda}A_{\alpha}\right)^n$.\\

Since for every $i \in \{1, \cdots, n\}$ we have $a_i \in \bigcap_{\alpha \in \Lambda} A_{\alpha}$, we have, for every $\alpha \in \Lambda$, $a_i \in A_{\alpha}$ for every $i \in \{ 1, \cdots, n\}$.\\

Since every $(A_{\alpha}, \Phi_{\alpha})$ is a $\mathcal{C}^{\infty}-$subring of $(A,\Phi)$, it follows that for every $\alpha \in \Lambda$ we have:

$$\Phi_{\alpha}(f)(a_1, \cdots, a_n):= \Phi(f)\upharpoonright_{A_{\alpha}^n}(a_1, \cdots, a_n) \in A_{\alpha}$$

thus

$$\Phi'(f)(a_1, \cdots, a_n) = \Phi_{\alpha}(f)\upharpoonright_{\left(\bigcap_{\alpha \in \Lambda} A_{\alpha}\right)^n}(a_1, \cdots, a_n) \in \bigcap_{\alpha \in \Lambda} A_{\alpha}$$

so the result follows.\\
\end{proof}

As an application of the previous result, we can define the \index{$\mathcal{C}^{\infty}-$subring! $\mathcal{C}^{\infty}-$subring generated by a subset}$\mathcal{C}^{\infty}-$subring generated by a subset of the carrier of a $\mathcal{C}^{\infty}-$ring:\\

\begin{definition}\label{gen}Let $(A,\Phi)$ be a $\mathcal{C}^{\infty}-$ring and $X \subseteq A$. The $\mathcal{C}^{\infty}-$subring of $(A,\Phi)$  generated by $X$ is given by:

$$\langle X\rangle = \bigcap_{\substack{X \subseteq A_i\\
(A_i, \Phi_i) \preceq (A,\Phi)}}(A_i,\Phi_i),$$

where $(A_i, \Phi_i) \preceq (A,\Phi)$ means that $(A_i, \Phi_i)$ is a $\mathcal{C}^{\infty}-$subring of $(A,\Phi)$

together with the $\mathcal{C}^{\infty}-$structure given in \textbf{Proposition \ref{Zimbabue}}.
\end{definition}

We note that, given any $\mathcal{C}^{\infty}-$ring $(A,\Phi)$, the map of partially ordered sets given by:

$$\begin{array}{cccc}
    \sigma: & (\wp(A), \subseteq) & \rightarrow & ({\rm Sub}\,(A), \subseteq) \\
     & X & \mapsto & \langle X \rangle
  \end{array}$$

satisfies the axioms of a \index{closure operation}closure operation.\\

\subsection{Directed union of $\mathcal{C}^{\infty}-$rings}

In general, given an arbitrary family of $\mathcal{C}^{\infty}-$subrings, $(A_{\alpha},\Phi_{\alpha})_{\alpha \in \Lambda}$ of a given $\mathcal{C}^{\infty}-$ring $(A,\Phi)$, its union, $\bigcup_{\alpha \in \Lambda} A_{\alpha}$, together with $\Phi\upharpoonright_{\cup_{\alpha \in \Lambda} A_{\alpha}}$, needs not to be a $\mathcal{C}^{\infty}-$subring of $(A,\Phi)$. However, there is an important case in which the union of a family of $\mathcal{C}^{\infty}-$subrings of a $\mathcal{C}^{\infty}-$ring $(A,\Phi)$ is again a $\mathcal{C}^{\infty}-$ring. This case is discussed in the following:\\

\begin{proposition}\label{xiquexique}
Let $(A,\Phi)$ be a $\mathcal{C}^{\infty}-$ring and let $\{ (A_{\alpha}, \Phi_{\alpha}) | \alpha \in \Lambda\}$, $\Lambda \neq \varnothing$, be a \index{directed family}directed family of $\mathcal{C}^{\infty}-$subrings of $(A,\Phi)$, that is, a family such that for every pair $(\alpha, \beta) \in \Lambda \times \Lambda$ there is some $\gamma \in \Lambda$ such that:
$$A_{\alpha} \subseteq A_{\gamma}$$
and
$$A_{\beta} \subseteq A_{\gamma}$$
We have that:
$$\left( \bigcup_{\alpha \in \Lambda} A_{\alpha}, \Phi'\right)$$
where:
$$\begin{array}{cccc}
\Phi': & \bigcup_{n \in \mathbb{N}}\mathcal{C}^{\infty}(\mathbb{R}^n, \mathbb{R}) & \rightarrow & \bigcup_{n \in \mathbb{N}} {\rm Func}\, \left( \left( \bigcup_{\alpha \in \Lambda} A_{\alpha}\right)^n, \bigcup_{\alpha \in \Lambda} A_{\alpha}\right)\\
 & (\mathbb{R}^n \stackrel{f}{\rightarrow} \mathbb{R}) & \mapsto & \Phi(f)\upharpoonright_{\left(\bigcup_{\alpha \in \Lambda} A_{\alpha}\right)^n}: \left( \bigcup_{\alpha \in \Lambda} A_{\alpha}\right)^n \to \bigcup_{\alpha \in \Lambda} A_{\alpha}
\end{array}$$

is a $\mathcal{C}^{\infty}-$subring of $(A,\Phi)$.
\end{proposition}
\begin{proof} First we note that since $\Lambda$ is directed, we have:

$$\bigcup_{\alpha \in \Lambda}(A_{\alpha})^n = \left( \bigcup_{\alpha \in \Lambda}A_{\alpha}\right)^n$$

Let $n \in \mathbb{N}$, $f \in \mathcal{C}^{\infty}(\mathbb{R}^n, \mathbb{R})$ and $(a_1, \cdots, a_n) \in \left( \bigcup_{\alpha \in \Lambda}A_{\alpha}\right)^n$. By definition of the union $\bigcup_{\alpha \in \Lambda} A_{\alpha}$, for every $i \in \{ 1, \cdots, n\}$ there is some $\alpha_i \in \Lambda$ such that:
$$a_i \in A_{\alpha_i}.$$

Since $\{ (A_{\alpha}, \Phi_{\alpha}) | \alpha \in \Lambda\}$ is directed, given $(A_{\alpha_1}, \Phi_{\alpha_1}), \cdots, (A_{\alpha_n}, \Phi_{\alpha_n})$, there is some $\gamma \in \Lambda$ such that for every $i \in \{1, \cdots, n \}$ we have:

$$A_{\alpha_i} \subseteq A_{\gamma}$$

Thus, $(a_1, \cdots, a_n) \in A_{\gamma}^n$, and since $(A_{\gamma}, \Phi_{\gamma})$ is a $\mathcal{C}^{\infty}-$subring of $(A,\Phi)$, it follows that:

$$\Phi_{\gamma}(f)(a_1, \cdots, a_n) := \Phi(f)\upharpoonright_{{A_{\gamma}}^n}(a_1, \cdots, a_n) \in A_{\gamma} \subseteq \bigcup_{\alpha \in \Lambda} A_{\alpha}$$

Since $\Lambda$ is directed and  $(A_{\gamma})^n \subseteq \left(\bigcup_{\alpha \in \lambda} A_{\alpha}\right)^n$, we have:

$$\Phi(f)\upharpoonright_{\left( \bigcup_{\alpha \in \Lambda}A_{\alpha}\right)^n}(a_1, \cdots, a_n) = \Phi(f)\upharpoonright_{(A_{\gamma})^n}(a_1, \cdots, a_n)$$

so:

\begin{multline*}\Phi'(f)(a_1, \cdots, a_n):= \Phi(f)\upharpoonright_{\left( \bigcup_{\alpha \in \Lambda}A_{\alpha}\right)^n}(a_1, \cdots, a_n) = \\
=\Phi(f)\upharpoonright_{(A_{\gamma})^n}(a_1, \cdots, a_n) \in A_{\gamma} \subseteq \bigcup_{\alpha \in \Lambda}A_{\alpha}\end{multline*}

Thus, for every $(a_1, \cdots, a_n) \in \left( \bigcup_{\alpha \in \Lambda}A_{\alpha} \right)^n$, we have $\Phi'(f)(a_1, \cdots, a_n) \in \bigcup_{\alpha \in \Lambda}A_{\alpha}$.\\

This proves that $\left( \bigcup_{\alpha \in \Lambda} A_{\alpha}, \Phi'\right)$ is a $\mathcal{C}^{\infty}-$subring of $(A,\Phi)$.
\end{proof}


\subsection{Products, $\mathcal{C}^{\infty}-$Congruences and Quotients}

Next we describe the products in the category $\mathcal{C}^{\infty}{\rm \bf Rng}$, that is, products of arbitrary families of $\mathcal{C}^{\infty}-$rings.\\

\begin{definition}\label{produto}Let $\{ (A_{\alpha}, \Phi_{\alpha}) | \alpha \in \Lambda \}$ be a family of $\mathcal{C}^{\infty}-$rings. The \index{product}\textbf{product} of this family is the pair:

$$\left( \prod_{\alpha \in \Lambda} A_{\alpha}, \Phi^{(\Lambda)} \right)$$

where $\Phi$ is given by:

$$\begin{array}{cccc}
  \Phi^{(\Lambda)}:   & \bigcup_{n \in \mathbb{N}}\mathcal{C}^{\infty}\,(\mathbb{R}^n, \mathbb{R}) & \rightarrow & \bigcup_{n \in \mathbb{N}} {\rm Func}\, \left( \left( \prod_{\alpha \in \Lambda} A_{\alpha}\right)^n, \prod_{\alpha \in \Lambda} A_{\alpha} \right) \\
     & (f: \mathbb{R}^n \stackrel{\mathcal{C}^{\infty}}{\rightarrow} \mathbb{R})& \mapsto & \begin{array}{cccc}
     \Phi^{(\Lambda)}(f): & \left( \prod_{\alpha \in \Lambda} A_{\alpha}\right)^n & \rightarrow & \prod_{\alpha \in \Lambda} A_{\alpha}\\
     & \!\!\!((x_{\alpha}^1)_{\alpha \in \Lambda}, \cdots, (x_{\alpha}^n)_{\alpha \in \Lambda}) & \mapsto & (\Phi_{\alpha}(f)(x_{\alpha}^1, \cdots, x_{\alpha}^{n}))_{\alpha \in \Lambda}
     \end{array}
\end{array}$$
\end{definition}

\begin{remark}\label{Karola}In particular, given a $\mathcal{C}^{\infty}-$ring $(A,\Phi)$, we have the product $\mathcal{C}^{\infty}-$ring:

$$(A\times A, \Phi^{(2)})$$

where:

$$\begin{array}{cccc}
    \Phi^{(2)} : & \bigcup_{n \in \mathbb{N}}\mathcal{C}^{\infty}(\mathbb{R}^n, \mathbb{R}) & \rightarrow & \bigcup_{n \in \mathbb{N}} {\rm Func}\,((A\times A)^n, A \times A)\\
     & (f: \mathbb{R}^n \stackrel{\mathcal{C}^{\infty}}{\rightarrow} \mathbb{R}) & \mapsto & (\Phi \times \Phi)(f): (A \times A)^n \rightarrow A \times A
\end{array}$$

and:

$$\begin{array}{cccc}
     \Phi^{(2)}(f): & (A\times A)^n & \rightarrow & A \times A\\
        & ((x_1,y_1), \cdots, (x_n, y_n)) & \mapsto & (\Phi(f)(x_1, \cdots,x_n), \Phi(f)(y_1, \cdots, y_n))
     \end{array}$$
\end{remark}

We turn now to the definition of congruence relations in $\mathcal{C}^{\infty}-$rings. As  we shall see later on, the congruences of a $\mathcal{C}^{\infty}-$rings will be classified by their ring-theoretic ideals.\\

\begin{definition}\label{conga}Let $(A, \Phi)$ be a $\mathcal{C}^{\infty}-$ring. A \index{$\mathcal{C}^{\infty}-$congruence}\textbf{$\mathcal{C}^{\infty}-$congruence}  is an equivalence relation $R \subseteq A \times A$  such that for every $n \in \mathbb{N}$ and $f \in \mathcal{C}^{\infty}(\mathbb{R}^n, \mathbb{R})$ we have:

$$(x_1, y_1), \cdots, (x_n,y_n) \in R \Rightarrow \Phi^{(2)}(f)((x_1,y_1), \cdots, (x_n,y_n)) \in R$$

In other words,a $\mathcal{C}^{\infty}-$congruence is an equivalence relations that preserves $\mathcal{C}^{\infty}-$function symbols.
\end{definition}

A characterization of a $\mathcal{C}^{\infty}-$congruence can be given using the product $\mathcal{C}^{\infty}-$structure, as we see in the following:\\

\begin{proposition}\label{P}Let $(A,\Phi)$ be a $\mathcal{C}^{\infty}-$ring and let $R \subseteq A \times A$ be an equivalence relation. Under these circumstances, $R$ is a $\mathcal{C}^{\infty}-$congruence on $(A,\Phi)$ if, and only if, $(R \times R, {\Phi^{(2)}}')$, where:

$$\begin{array}{cccc}
{\Phi^{(2)}}': & \bigcup_{n \in \mathbb{N}} \mathcal{C}^{\infty}\,(\mathbb{R}^n, \mathbb{R}) & \rightarrow & \bigcup_{n \in \mathbb{N}}{\rm Func}\,(R^n, R)\\
   & (\mathbb{R}^n \stackrel{f}{\rightarrow} \mathbb{R}) & \mapsto & \Phi^{(2)}(f)\upharpoonright_{R^n}: R^{n} \rightarrow R
\end{array}$$

is a $\mathcal{C}^{\infty}-$subring of $(A\times A, \Phi^{(2)})$, with the structure described in the \textbf{Remark \ref{Karola}}.
\end{proposition}
\begin{proof}
Given any $n \in \mathbb{N}$, $f \in \mathcal{C}^{\infty}(\mathbb{R}^n, \mathbb{R})$, $((x_1, y_1), \cdots, (x_n, y_n)) \in R^{n}$  we have,
by definition:
$$\Phi^{(2)}(f)((x_1,y_1) \cdots, (x_n,y_n)) :=  (\Phi(f)(x_1, \cdots, x_n), \Phi(f)(y_1, \cdots, y_n))$$

and since $R$ is a $\mathcal{C}^{\infty}-$congruence, we have:

$$((x_1,y_1), \cdots, (x_n,y_n)) \in R^n  \Rightarrow {\Phi}^{(2)}(f)((x_1,y_1), \cdots, (x_n,y_n)) \in R$$

Since:

$${\Phi}^{(2)}(f)((x_1,y_1), \cdots, (x_n,y_n)) = (\Phi(f)(x_1, \cdots, x_n), \Phi(f)(y_1, \cdots, y_n))$$

it follows that

$$((x_1,y_1), \cdots, (x_n,y_n)) \in R^{n} \Rightarrow (\Phi(f)(x_1, \cdots, x_n), \Phi(f)(y_1, \cdots, y_n)) \in R$$

Also, since $R^n \subseteq (A \times A)^n$,

$$\Phi^{(2)}(f)\upharpoonright_{R^n}((x_1,y_1), \cdots, (x_n,y_n)) = \Phi^{(2)}(f)((x_1,y_1), \cdots, (x_n,y_n)) $$
so

\begin{multline*}
{\Phi^{(2)}}'(f)((x_1,y_1), \cdots, (x_n,y_n)) = \Phi^{(2)}(f)\upharpoonright_{R^n}((x_1,y_1), \cdots, (x_n,y_n)) = \\
= \Phi^{(2)}(f)((x_1,y_1), \cdots, (x_n,y_n)) = (\Phi(f)(x_1, \cdots, x_n), \Phi(f)(y_1, \cdots, y_n))\in R
\end{multline*}

Thus $(R, {\Phi^{(2)}}')$ is a $\mathcal{C}^{\infty}-$subring of $(A \times A, \Phi^{(2)})$.\\

Conversely, let $R$ be an equivalence relation on $A$ such that  $(R, {\Phi^{(2)}}')$ is a $\mathcal{C}^{\infty}-$subring of $(A \times A, \Phi^{(2)})$.\\

Given any $n \in \mathbb{N}$, $f \in \mathcal{C}^{\infty}(\mathbb{R}^n, \mathbb{R})$ and $(x_1, y_1), \cdots, (x_n, y_n) \in R$, since $(R, {\Phi^{(2)}}')$ is a $\mathcal{C}^{\infty}-$subring of $(A\times A, \Phi^{(2)})$, we have ${\Phi^{(2)}}'(f)((x_1,y_1), \cdots, (x_n,y_n)) \in R$.\\

By definition,

$${\Phi^{(2)}}'(f)((x_1,y_1), \cdots, (x_n,y_n)) = {\Phi^{(2)}}(f)\upharpoonright_{R^n}((x_1,y_1), \cdots, (x_n,y_n))$$

and since $R^n \subseteq (A \times A)^n$, we have:

$${\Phi^{(2)}}(f)\upharpoonright_{R^n}((x_1,y_1), \cdots, (x_n,y_n)) = {\Phi^{(2)}}(f)((x_1,y_1), \cdots, (x_n,y_n))$$

and

$$ {\Phi^{(2)}}(f)((x_1,y_1), \cdots, (x_n,y_n)) = (\Phi(f)(x_1, \cdots, x_n), \Phi(f)(y_1, \cdots, y_n))$$
Thus,

$$(x_1,y_1), \cdots, (x_n,y_n) \in R \Rightarrow (\Phi(f)(x_1, \cdots, x_n), \Phi(f)(y_1, \cdots, y_n)) \in R$$

and $R$ is a $\mathcal{C}^{\infty}-$congruence.
\end{proof}

\begin{remark}\label{pele}
Given a $\mathcal{C}^{\infty}-$ring $(A,\Phi)$ and a $\mathcal{C}^{\infty}-$congruence $R \subseteq A \times A$, let:

$$\dfrac{A}{R}= \{ \overline{a} | a \in A \}$$

be the quotient set. Given any $n \in \mathbb{N}$, $f \in \mathcal{C}^{\infty}(\mathbb{R}^n, \mathbb{R})$ and $(\overline{a_1}, \cdots, \overline{a_n}) \in \left( \dfrac{A}{R}\right)^n$ we define:

$$\begin{array}{cccc}
\overline{\Phi}: & \bigcup_{n \in \mathbb{N}}\mathcal{C}^{\infty}\,(\mathbb{R}^n, \mathbb{R}) & \rightarrow & \bigcup_{n \in \mathbb{N}} {\rm Func}\left( \left( \dfrac{A}{R}\right)^n, \dfrac{A}{R}\right)\\
  & (f: \mathbb{R}^n \stackrel{\mathcal{C}^{\infty}}{\rightarrow} \mathbb{R}) & \mapsto & \left(\overline{\Phi}(f): \left( \dfrac{A}{R}\right)^n \rightarrow \dfrac{A}{R}\right)
\end{array}$$

where:

$$\begin{array}{cccc}
\overline{\Phi}(f): & \left( \dfrac{A}{R}\right)^n & \rightarrow & \dfrac{A}{R}\\
 & (\overline{a_1}, \cdots, \overline{a_n}) & \mapsto & \overline{\Phi(f)(a_1, \cdots, a_n)}
\end{array}$$

Note that the interpretation above is indeed a function, that is, its value does not depend on any particular choice of the representing element. This means that given $(\overline{a_1}, \cdots, \overline{a_n}), (\overline{a_1'}, \cdots, \overline{a_n'}) \in \left( \dfrac{A}{R} \right)^n$ such that $(a_1, a_1'), \cdots, (a_n, a_n') \in R$, we have:

$$\overline{\Phi}(f)(\overline{a_1}, \cdots, \overline{a_n}) = \overline{\Phi(f)(a_1, \cdots, a_n)}$$

and since $R$ is a $\mathcal{C}^{\infty}-$congruence,

$$(a_1, a_1'), \cdots, (a_n, a_n') \in R \Rightarrow (\Phi(f)(a_1, \cdots, a_n), \Phi(f)(a_1', \cdots, a_n')) \in R$$

so:

$$\overline{\Phi}(\overline{a_1}, \cdots, \overline{a_n}) = \overline{\Phi(f)(a_1, \cdots, a_n)} = \overline{\Phi(f)(a_1', \cdots, a_n')} = \overline{\Phi}(f)(\overline{a_1'}, \cdots, \overline{a_n'})$$
\end{remark}

The above construction leads directly to the following:

\begin{definition}\label{quot}Let $(A,\Phi)$ be a $\mathcal{C}^{\infty}-$ring and let $R \subseteq A \times A$ be a $\mathcal{C}^{\infty}-$congruence. The \index{quotient $\mathcal{C}^{\infty}-$ring}\textbf{quotient $\mathcal{C}^{\infty}-$ring} of $A$ by $R$ is the ordered pair:

$$\left( \dfrac{A}{R}, \overline{\Phi}\right)$$
where:
$$\dfrac{A}{R}= \{ \overline{a} | a \in A \}$$

and

$$\begin{array}{cccc}
\overline{\Phi}: & \bigcup_{n \in \mathbb{N}}\mathcal{C}^{\infty}\,(\mathbb{R}^n, \mathbb{R}) & \rightarrow & \bigcup_{n \in \mathbb{N}} {\rm Func}\left( \left( \dfrac{A}{R}\right)^n, \dfrac{A}{R}\right)\\
  & (f: \mathbb{R}^n \stackrel{\mathcal{C}^{\infty}}{\rightarrow} \mathbb{R}) & \mapsto & \left(\overline{\Phi}(f): \left( \dfrac{A}{R}\right)^n \rightarrow \dfrac{A}{R}\right)
\end{array}$$

where $\overline{\Phi}(f)$ is described in \textbf{Remark \ref{pele}}.
\end{definition}

The following result shows that the canonical quotient map is, again, a $\mathcal{C}^{\infty}-$homomorphism.\\

\begin{proposition}Let $(A,\Phi)$ be a $\mathcal{C}^{\infty}-$ring and let $R \subseteq A \times A$ be a $\mathcal{C}^{\infty}-$congruence. The function:

$$\begin{array}{cccc}
q: & (A, \Phi) & \rightarrow & \left( \dfrac{A}{R}, \overline{\Phi}\right)\\
 & a & \mapsto & \overline{a}
\end{array}$$

is a $\mathcal{C}^{\infty}-$homomorphism.
\end{proposition}
\begin{proof}Let $n \in \mathbb{N}$ and $f \in \mathcal{C}^{\infty}(\mathbb{R}^n, \mathbb{R})$. We are going to show that the following diagram commutes:

$$\xymatrixcolsep{5pc}\xymatrix{
A^n \ar[r]^{q^{n}} \ar[d]_{\Phi(f)} &  \left( \dfrac{A}{R}\right)^n \ar[d]^{\overline{\Phi}(f)}\\
A \ar[r]_{q} & \left( \dfrac{A}{R}\right)
}$$

Given $(a_1, \cdots, a_n) \in A^n$, we have on the one hand:

$$\overline{\Phi}(f) \circ q^n(a_1, \cdots, a_n) = \overline{\Phi}(f)(q(a_1), \cdots, q(a_n)) = \overline{\Phi}(f)(\overline{a_1}, \cdots, \overline{a_n}):= \overline{\Phi(f)(a_1, \cdots, a_n)}$$

On the other hand,

$$q \circ \Phi(f)(a_1, \cdots, a_n) = \overline{\Phi(f)(a_1, \cdots, a_n)}$$

so $\overline{\Phi} \circ q^n = q \circ \Phi(f)$, and $q$ is a $\mathcal{C}^{\infty}-$homomorphism.
\end{proof}

We remark that the structure given above is the unique $\mathcal{C}^{\infty}-$structure such that the quotient map is a $\mathcal{C}^{\infty}-$homomorphism.\\

\begin{proposition}Let $(A,\Phi)$ and $(B,\Psi)$ be two $\mathcal{C}^{\infty}-$rings and let $\varphi: (A,\Phi) \rightarrow (B, \Psi)$ be a $\mathcal{C}^{\infty}-$homomorphism. The set:

$$\ker (\varphi) = \{ (a,a') \in A \times A | \varphi(a) = \varphi(a')\}$$

is a $\mathcal{C}^{\infty}-$congruence on $(A,\Phi)$.
\end{proposition}
\begin{proof}It is easy to check that $\ker (\varphi)$ is an equivalence relation on $A$.\\

Let $n \in \mathbb{N}$, $f \in \mathcal{C}^{\infty}(\mathbb{R}^n, \mathbb{R})$ and $(a_1,a_1'), \cdots, (a_n,a_n') \in \ker (\varphi)$, that is:
$$(\forall i \in \{1, \cdots, n \})(\varphi(a_i)=\varphi(a_i'))$$

We are going to show that $(\Phi(f)(a_1, \cdots, a_n), \Phi(f)(a_1', \cdots, a_n')) \in \ker (\varphi)$.\\

Since $\varphi$ is a $\mathcal{C}^{\infty}-$homomorphism, we have the following commutative diagram:

$$\xymatrixcolsep{5pc}\xymatrix{
A^n \ar[r]^{\varphi^{n}} \ar[d]_{\Phi(f)} & B^n \ar[d]^{\Psi(f)}\\
A \ar[r]_{\varphi} & B
}$$

thus, we have:

\begin{multline*}\varphi(\Phi(f)(a_1, \cdots, a_n)) = \Psi(f)(\varphi(a_1), \cdots, \varphi(a_n)) = \\
= \Psi(f)(\varphi(a_1'), \cdots, \varphi(a_n')) = \varphi(\Phi(f)(a_1', \cdots, a_n'))\end{multline*}

and $(\Phi(f)(a_1, \cdots, a_n), \Phi(f)(a_1', \cdots, a_n')) \in \ker (\varphi)$.\\
\end{proof}

\begin{corollary}Let $(A,\Phi)$ and $(B,\Psi)$ be two $\mathcal{C}^{\infty}-$rings and let $\varphi: (A,\Phi) \rightarrow (B, \Psi)$ be a $\mathcal{C}^{\infty}-$homo\-morphism. Then $(\ker(\varphi), {\Phi^{(2)}}')$ is a $\mathcal{C}^{\infty}-$subring of $(A\times A, \Phi^{(2)})$.
\end{corollary}

\begin{proposition}For every $\mathcal{C}^{\infty}-$congruence $R \subseteq A \times A$ in $(A,\Phi)$, there are some $\mathcal{C}^{\infty}-$ring $(B,\Psi)$ and some $\mathcal{C}^{\infty}-$homomorphism $\varphi: (A,\Phi) \rightarrow (B,\Psi)$ such that $R = \ker(\varphi)$.
\end{proposition}
\begin{proof}
It suffices to take $(B,\Psi) = \left( \dfrac{A}{R}, \overline{\Phi}\right)$ and $\varphi = q_R: (A,\Phi) \rightarrow \left( \dfrac{A}{R}, \overline{\Phi}\right)$
\end{proof}

\begin{theorem}{\index{{Fundamental Theorem of the $\mathcal{C}^{\infty}-$Homomorphism}}\emph{\textbf{(Fundamental Theorem of the $\mathcal{C}^{\infty}-$Homomorphism)}}}\label{FTH} Let $(A,\Phi)$ be a $\mathcal{C}^{\infty}-$ring and $R \subseteq A \times A$ be a $\mathcal{C}^{\infty}-$congruence. For every $\mathcal{C}^{\infty}-$ring $(B,\Psi)$ and for every  $\mathcal{C}^{\infty}-$homomorphism $\varphi: (A,\Phi) \to (B,\Psi)$ such that $R \subseteq \ker(\varphi)$, that is, such that:

$$(a,a') \in R \Rightarrow \varphi(a) = \varphi(a'),$$

there is a unique $\mathcal{C}^{\infty}-$homomorphism:

$$\widetilde{\varphi}: \left( \dfrac{A}{R}, \overline{\Phi}\right) \rightarrow (B,\Psi)$$

such that the following diagram commutes:

$$\xymatrixcolsep{5pc}\xymatrix{
(A,\Phi) \ar[d]_{q} \ar[r]^{\varphi} & (B,\Psi)\\
\left( \dfrac{A}{R}, \overline{\Phi}\right) \ar@{-->}[ur]_{\widetilde{\varphi}} &
}$$
that is, such that $\widetilde{\varphi} \circ q = \varphi$, where $\overline{\Phi}$ is the canonical $\mathcal{C}^{\infty}-$structure induced on the quotient $\frac{A}{R}$
\end{theorem}
\begin{proof}
Define:
$$\begin{array}{cccc}
\widetilde{\varphi}: & \dfrac{A}{R} & \rightarrow & B\\
 & \overline{a} & \mapsto & \varphi(a)
\end{array}$$

and note that $\widetilde{\varphi}$ defines a function, since given $(a,a')\in R$, i.e., such that $\overline{a}=\overline{a'}$, we have $\varphi(a) = \varphi(a')$, so $\widetilde{\varphi}(\overline{a}) = \varphi(a) = \varphi(a') = \widetilde{\varphi}(\overline{a'})$.\\

As functions, we have $\widetilde{\varphi} \circ q = \varphi$.\\

Given $n \in \mathbb{N}$ and $f \in \mathcal{C}^{\infty}(\mathbb{R}^n, \mathbb{R})$, the following diagram commutes,

$$\xymatrixcolsep{5pc}\xymatrix{
\left( \dfrac{A}{R}\right)^n \ar[r]^{\widetilde{\varphi}^n} \ar[d]_{\overline{\Phi}(f)} & B^n \ar[d]^{\Psi(f)}\\
\dfrac{A}{R} \ar[r]_{\widetilde{\varphi}} & B
}$$

since for any $(\overline{a_1}, \cdots, \overline{a_n}) \in \left( \dfrac{A}{R}\right)^n$ we have:\\

$$\widetilde{\varphi} \circ \overline{\Phi}(f)(\overline{a_1}, \cdots, \overline{a_n}) = \widetilde{\varphi}(\overline{\Phi(f)(a_1, \cdots, a_n)}) := \varphi(\Phi(f)(a_1, \cdots, a_n))$$

and since $\varphi$ is a $\mathcal{C}^{\infty}-$homomorphism, we have:

$$\varphi(\Phi(f)(a_1, \cdots, a_n)) = \Psi(f)(\varphi(a_1), \cdots, \varphi(a_n))$$

thus:

$$\widetilde{\varphi} \circ \overline{\Phi}(f)(\overline{a_1}, \cdots, \overline{a_n}) = \Psi(f)(\varphi(a_1), \cdots, \varphi(a_n)).$$

On the other hand, we have:

$$\Psi(f)\circ \widetilde{\varphi}^n(\overline{a_1}, \cdots, \overline{a_n}) = \Psi(f)(\widetilde{\varphi}(\overline{a_1}), \cdots, \widetilde{\varphi}(\overline{a_n})) = \Psi(f)(\varphi(a_1), \cdots, \varphi(a_n))$$

so

$$(\forall (\overline{a_1}, \cdots, \overline{a_n}) \in \left( \dfrac{A}{R}\right)^n)(\widetilde{\varphi} \circ \overline{\Phi}(f)(\overline{a_1}, \cdots, \overline{a_n}) = \Psi(f)\circ \widetilde{\varphi}^n(\overline{a_1}, \cdots, \overline{a_n})),$$

and $\widetilde{\varphi}$ is a $\mathcal{C}^{\infty}-$homomorphism.\\

Thus we have the following equation of $\mathcal{C}^{\infty}-$homomorphisms:

$$\widetilde{\varphi}\circ q = \varphi.$$

Since $q$ is surjective, it follows that $\widetilde{\varphi}$ is unique.
\end{proof}

The following result is straightforward:

\begin{proposition}\label{H}Let $(A,\Phi)$ and $(B,\Psi)$ be two $\mathcal{C}^{\infty}-$rings and let $\varphi: (A,\Phi) \rightarrow (B,\Psi)$ be a $\mathcal{C}^{\infty}-$homomorphism. The  ordered pair:
$$(\varphi[A], \Psi')$$
where:

$$\begin{array}{cccc}
\Psi': & \bigcup_{n \in \mathbb{N}} \mathcal{C}^{\infty}(\mathbb{R}^n, \mathbb{R}) & \rightarrow & \bigcup_{n \in \mathbb{N}} {\rm Func}\,(\varphi[A]^n, \varphi[A])\\
 & (\mathbb{R}^n \stackrel{f}{\rightarrow} \mathbb{R}) & \mapsto & \Psi(f)\upharpoonright_{\varphi[A]^n}: \varphi[A]^n \rightarrow \varphi[A]
\end{array}$$

is a $\mathcal{C}^{\infty}-$subring of $(B,\Psi)$, called \textbf{the homomorphic image of $A$ by $\varphi$}.
\end{proposition}

\begin{corollary}Let $(A,\Phi)$ and $(B,\Psi)$ be two $\mathcal{C}^{\infty}-$rings and let $\varphi: (A,\Phi) \to (B,\Psi)$ be a $\mathcal{C}^{\infty}-$homomorphism. As we have noticed in \textbf{Proposition \ref{H}}, $(\varphi[A],\Psi')$ is a $\mathcal{C}^{\infty}-$subring of $(B,\Psi)$.\\

Under these circumstances, there is a unique  $\mathcal{C}^{\infty}-$isomorphism:

$$\widetilde{\varphi}: \left( \dfrac{A}{\ker(\varphi)}, \overline{\Phi}\right) \rightarrow (\varphi[A],\Psi')$$

such that the following diagram commutes:

$$\xymatrixcolsep{5pc}\xymatrix{
(A,\Phi) \ar[d]_{q} \ar[r]^{\varphi} & (\varphi[A],\Psi')\\
\left( \dfrac{A}{\ker(\varphi)}, \overline{\Phi}\right) \ar@{-->}[ur]_{\widetilde{\varphi}} &
}$$
that is, such that $\widetilde{\varphi} \circ q = \varphi$, where $\overline{\Phi}$ is the canonical $\mathcal{C}^{\infty}-$structure induced on the quotient $\frac{A}{\ker(\varphi)}$
\end{corollary}
\begin{proof}
Applying the previous result to $R=\ker(\varphi)$ yields the existence of a unique $\mathcal{C}^{\infty}-$homomorphism such that the diagram commutes. We need only to prove that $\widetilde{\varphi}$ is bijective.\\

Given $\overline{a}, \overline{a'}$ such that $\widetilde{\varphi}(\overline{a}) = \widetilde{\varphi}(\overline{a'})$, by definition we have $\varphi(a)=\varphi(a')$, so $(a,a') \in \ker(\varphi)$ and $\overline{a}=\overline{a'}$.\\

Also, given any $y \in \varphi[A]$, there is some $a\in A$ such that $\varphi(a)=y$. Thus, $\widetilde{\varphi}(\overline{a}) = \varphi(a) = y$, so $\varphi$ is surjective.\\
\end{proof}

\subsection{Directed Colimits of $\mathcal{C}^{\infty}-$Rings}\label{DColim}

The following result is going to be used to construct directed colimits of $\mathcal{C}^{\infty}-$rings.\\

\begin{lemma}\label{Gloria}Let $(A,\Phi)$ be a $\mathcal{C}^{\infty}-$ring. The ordered pair:
$$(A \times \{ \alpha\}, \Phi \times {\rm id}_{\alpha})$$

where:

$$\begin{array}{cccc}
\Phi \times {\rm id}_{\alpha} : & \bigcup_{n \in \mathbb{N}} \mathcal{C}^{\infty}(\mathbb{R}^n, \mathbb{R}) & \rightarrow & \bigcup_{n \in \mathbb{N}} {\rm Func}\,((A \times \{ \alpha\})^n, A \times \{ \alpha\})\\
   & (\mathbb{R}^n \stackrel{f}{\rightarrow} \mathbb{R}) & \mapsto & \begin{array}{cccc}\Phi(f) \times {\rm id}_{\alpha}:& (A \times \{ \alpha\})^n & \rightarrow & A \times \{ \alpha\}\\
    & ((a_1, \alpha), \cdots, (a_n, \alpha)) & \mapsto & (\Phi(f)(a_1, \cdots, a_n), \alpha)
   \end{array}
\end{array}$$

is a $\mathcal{C}^{\infty}-$ring and:

$$\begin{array}{cccc}
\pi_1 : & A \times \{ \alpha\} & \rightarrow & A\\
        & (a, \alpha) & \mapsto & a
\end{array}$$
is a $\mathcal{C}^{\infty}-$isomorphism, that is:
$$(A, \Phi) \cong_{\pi_1} (A\times \{ \alpha\}, \Phi \times {\rm id}_{\alpha})$$
\end{lemma}
\begin{proof}
It is clear that $\pi_1$ is a bijection, so it suffices to prove it is a $\mathcal{C}^{\infty}-$homomorphism such that its inverse:

$$\begin{array}{cccc}
\iota_1: & A & \rightarrow & A \times \{ \alpha\}\\
         & a & \mapsto & (a, \alpha)
\end{array}$$

is a $\mathcal{C}^{\infty}-$homomorphism.\\

Let $n \in \mathbb{N}$ and $f \in \mathcal{C}^{\infty}(\mathbb{R}^n, \mathbb{R})$. Given $((a_1,\alpha), \cdots, (a_n, \alpha)) \in (A \times \{ \alpha\})^n$, we have:
$$\pi_1 \circ ((\Phi \times {\rm id}_{\alpha})(f))((a_1, \alpha), \cdots, (a_n, \alpha)) = \pi_1(\Phi(f)(a_1, \cdots, a_n), \alpha) = \Phi(f)(a_1, \cdots, a_n)$$
and
$$\Phi(f) \circ \pi_1^n ((a_1, \alpha), \cdots, (a_n, \alpha)) = \Phi(f)(\pi_1(a_1, \alpha), \cdots, \pi_1(a_n, \alpha)) = \Phi(f)(a_1, \cdots, a_n).$$

Thus, the following diagram commutes:

$$\xymatrixcolsep{5pc}\xymatrix{
(A \times \{ \alpha\})^n \ar[r]^{\pi_1^n} \ar[d]_{(\Phi \times {\rm id}_{\alpha})(f)} & A^n \ar[d]^{\Phi(f)}\\
A \times \{ \alpha\} \ar[r]_{\pi_1} & A
}$$

and $\pi_1$ is a $\mathcal{C}^{\infty}-$homomorphism.\\

Also, for any $(a_1, \cdots, a_n) \in A^n$ we have:

$$(\Phi \times {\rm id}_{\alpha})\circ \iota^n(a_1, \cdots, a_n) = (\Phi \times {\rm id}_{\alpha})((a_1, \alpha), \cdots, (a_n, \alpha)) = (\Phi(f)(a_1, \cdots, a_n), \alpha)$$

and

$$\iota \circ \Phi(f)(a_1, \cdots, a_n) = (\Phi(f)(a_1, \cdots, a_n), \alpha)$$

so the following diagram commutes:

$$\xymatrixcolsep{5pc}\xymatrix{
A^n \ar[r]^{\iota^n} \ar[d]_{\Phi(f)} & (A \times \{ \alpha\})^n \ar[d]^{(\Phi \times {\rm id}_{\alpha})}\\
A \ar[r]_{\iota} & A \times \{ \alpha\}
}$$

and $\iota$ is a $\mathcal{C}^{\infty}-$homomorphism, inverse to $\pi_1$.
\end{proof}

Now we are going to describe the construction of  directed colimits of directed families of $\mathcal{C}^{\infty}-$rings.\\

\begin{theorem}\label{colimite}Let $(I, \leq)$ be a directed set and let $((A_{\alpha}, \Phi_{\alpha}), \mu_{\alpha \beta})_{\alpha, \beta \in I}$ be a directed system. There is an object $(A, \Phi)$ in $\mathcal{C}^{\infty}{\rm \bf Rng}$ such that:

$$(A, \Phi) \cong \varinjlim_{\alpha \in I} (A_{\alpha}, \Phi_{\alpha})$$

\end{theorem}
\begin{proof}
Let:
$$A' := \bigcup_{\alpha \in I} A_{\alpha}\times \{ \alpha\}$$
and consider the following equivalence relation:
$$R = \{ ((a, \alpha), (b, \beta)) \in A' \times A' | (\exists \gamma \in I)(\gamma \geq \alpha, \beta)(\mu_{\alpha \gamma}(a) = \mu_{\beta \gamma}(b))\}$$

Take $A = \dfrac{A'}{R}$ and define the following $\mathcal{C}^{\infty}-$structure on $A$:

$$\begin{array}{cccc}
\Phi: & \bigcup_{n \in \mathbb{N}} \mathcal{C}^{\infty}(\mathbb{R}^n, \mathbb{R}) & \rightarrow & \bigcup_{n \in \mathbb{N}} {\rm Func}\,(A^n, A)\\
  & (\mathbb{R}^n \stackrel{f}{\rightarrow} \mathbb{R}) & \mapsto & (A^n \stackrel{\Phi(f)}{\rightarrow} A)
\end{array}$$

where $\Phi(f): A^n \to A$ is defined as follows: given $(\overline{(a_1, \alpha_1)}, \cdots, \overline{(a_n, \alpha_n)}) \in A^n$, we have $\{ \alpha_1, \cdots, \alpha_n \} \subseteq I$, and since $(I, \leq)$ is directed, there is some $\alpha \in I$ such that:
$$(\forall i \in \{ 1, \cdots, n\})(\alpha_i \leq \alpha )$$

We define:

$$\Phi(f)(\overline{(a_1,\alpha_1)}, \cdots, \overline{(a_n, \alpha_n)}):= \overline{(\Phi_{\alpha}(f)(\mu_{\alpha_1 \alpha}(a_1), \cdots, \mu_{\alpha_n \alpha}(a_n)), \alpha)}
$$

We are going to show that this definition is independent of the choice of $\alpha$.\\

Let $\beta \geq \alpha_1, \cdots, \alpha_n$ so we have:

$$\Phi(f)(\overline{(a_1, \alpha_1)}, \cdots, \overline{(a_n, \alpha_n)}) = \overline{(\Phi_{\beta}(f)(\mu_{\alpha_1 \beta}(a_1), \cdots, \mu_{\alpha_n \beta}(a_n)), \beta)}
$$

We need to show that there is some $\gamma \in I$ such that $\gamma \geq \alpha, \beta$ and:

$$\mu_{\alpha \gamma}(\Phi_{\alpha}(f)(\mu_{\alpha_1 \alpha}(a_i)), \cdots, \mu_{\alpha_n \alpha}(a_n))) = \mu_{\beta \gamma}(\Phi_{\beta}(f)(\mu_{\alpha_1 \beta}(a_i)), \cdots, \mu_{\alpha_n \beta}(a_n)))$$

Choose $\gamma \geq \alpha, \beta$. Then:
$$(\forall i \in \{ 1, \cdots, n\})(\mu_{\beta \gamma} \circ \mu_{\alpha_i \beta} = \mu_{\alpha \beta } \circ \mu_{\alpha_i \alpha})$$

Since $\mu_{\alpha \gamma}: (A_{\alpha}, \Phi_{\alpha}) \rightarrow (A_{\gamma}, \Phi_{\gamma})$ is a $\mathcal{C}^{\infty}-$homomorphism, we have the following commutative diagram:

$$\xymatrixcolsep{5pc}\xymatrix{
A_{\alpha}^n \ar[r]^{\mu_{\alpha \gamma}^n} \ar[d]_{\Phi_{\alpha}(f)} & A_{\gamma}^n \ar[d]^{\Phi_{\gamma}(f)}\\
A_{\alpha} \ar[r]_{\mu_{\alpha \gamma}} & A_{\gamma}
}$$
so

\begin{multline*}\mu_{\alpha \gamma}(\Phi_{\alpha}(f)(\mu_{\alpha_1 \alpha}(a_1)), \cdots, \Phi_{\alpha}(f)(\mu_{\alpha_n \alpha}(a_n))) =\\
=\Phi_{\gamma}(f)(\mu_{\alpha \gamma}(\mu_{\alpha_1 \alpha}(a_1)), \cdots, \mu_{\alpha \gamma}(\mu_{\alpha_n \alpha}(a_n)))
\end{multline*}

and since:

$$(\forall i \in \{1, \cdots, n \})(\mu_{\alpha \gamma}\circ \mu_{\alpha_i \alpha} = \mu_{\alpha_i \gamma})$$
we have:

\begin{equation}\label{Baf}\mu_{\alpha \gamma}(\Phi_{\alpha}(f)(\mu_{\alpha_1 \alpha}(a_1)), \cdots, \Phi_{\alpha}(f)(\mu_{\alpha_n \alpha}(a_n))) = \Phi_{\gamma}(f)(\mu_{\alpha_1 \gamma}(a_1), \cdots, \mu_{\alpha_n \gamma}(a_n))
\end{equation}

Also, since $\mu_{\beta \gamma}: (A_{\beta}, \Phi_{\beta}) \rightarrow (A_{\gamma}, \Phi_{\gamma})$ is a $\mathcal{C}^{\infty}-$homomorphism, we have the following commutative diagram:

$$\xymatrixcolsep{5pc}\xymatrix{
A_{\beta}^n \ar[r]^{\mu_{\beta \gamma}^n} \ar[d]_{\Phi_{\beta}(f)} & A_{\gamma}^n \ar[d]^{\Phi_{\gamma}(f)}\\
A_{\beta} \ar[r]_{\mu_{\beta \gamma}} & A_{\gamma}
}$$
so

\begin{multline*}\mu_{\beta \gamma}(\Phi_{\beta}(f)(\mu_{\alpha_1 \beta}(a_1)), \cdots, \Phi_{\beta}(f)(\mu_{\alpha_n \beta}(a_n))) = \\
=\Phi_{\gamma}(f)(\mu_{\beta \gamma}(\mu_{\alpha_1 \beta}(a_1)), \cdots, \mu_{\beta \gamma}(\mu_{\alpha_n \beta}(a_n)))
\end{multline*}

and since:

$$(\forall i \in \{1, \cdots, n \})(\mu_{\beta \gamma}\circ \mu_{\alpha_i \beta} = \mu_{\alpha_i \gamma})$$
we have:

\begin{equation}\label{Fab}\mu_{\beta \gamma}(\Phi_{\beta}(f)(\mu_{\alpha_1 \beta}(a_1)), \cdots, \Phi_{\beta}(f)(\mu_{\alpha_n \beta}(a_n))) = \Phi_{\gamma}(f)(\mu_{\alpha_1 \gamma}(a_1), \cdots, \mu_{\alpha_n \gamma}(a_n))
\end{equation}

Comparing \eqref{Baf} and \eqref{Fab}, we get:

\begin{multline*}\mu_{\alpha \gamma}(\Phi_{\alpha}(f)(\mu_{\alpha_1 \alpha}(a_1)), \cdots, \Phi_{\alpha}(f)(\mu_{\alpha_n \alpha}(a_n))) =\\
=\mu_{\beta \gamma}(\Phi_{\beta}(f)(\mu_{\alpha_1 \beta}(a_1)), \cdots, \Phi_{\beta}(f)(\mu_{\alpha_n \beta}(a_n)))
\end{multline*}

so

$$(\overline{\Phi_{\alpha}(\mu_{\alpha_1\alpha}(a_1), \cdots, \mu_{\alpha_n \alpha}(a_n))}, \alpha) = \overline{(\Phi_{\beta}(\mu_{\alpha_1 \beta}(a_1), \cdots, \mu_{\alpha_n \beta}(a_n)), \beta)}$$

and $\Phi(f)$ does not depend on the choice of the index. The definition of $\Phi(f)$ does not depend on the choice of the representing elements, $(a_i, \alpha_i)+R$, either\footnote{the proof of this fact is analogous to the one we just made.}.\\

Hence $\Phi$ is a $\mathcal{C}^{\infty}-$structure on $A$ and this shows that $(A,\Phi)$ is a $\mathcal{C}^{\infty}-$ring.\\

For each $\alpha \in I$ we define:

$$\begin{array}{cccc}
\eta_{\alpha}: & A_{\alpha} \times \{ \alpha\} & \rightarrow & A \\
  & (a,\alpha) & \mapsto & \overline{(a, \alpha)}
\end{array}$$

\textbf{Claim:} $\eta_{\alpha}: A_{\alpha} \times \{ \alpha\} \rightarrow A$ is a $\mathcal{C}^{\infty}-$homomorphism between $(A_{\alpha}\times \{ \alpha\}, \Phi_{\alpha} \times {\rm id}_{\alpha})$ and $(A,\Phi)$.\\

Let $n \in \mathbb{N}$ and $f \in \mathcal{C}^{\infty}(\mathbb{R}^n, \mathbb{R})$. Given $((a_1, \alpha), \cdots, (a_n, \alpha)) \in (A_{\alpha} \times \{ \alpha \})^n$. We have:

\begin{multline*}\Phi(f)\circ \eta_{\alpha}^n((a_1, \alpha), \cdots, (a_n, \alpha)) := (\Phi_{\alpha} \times {\rm id}_{\alpha})(f)(\eta_{\alpha}(a_1, \alpha), \cdots, \eta_{\alpha}(a_n, \alpha)) =\\
=\overline{(\Phi_{\alpha}(f)(a_1, \cdots, a_n), \alpha)}
\end{multline*}

and

\begin{multline*}\eta_{\alpha}\circ (\Phi_{\alpha} \times {\rm id}_{\alpha})(f)((a_1, \alpha), \cdots, (a_n, \alpha)) = \eta_{\alpha}(\Phi_{\alpha}(f)(a_1, \cdots, a_n), \alpha) =\\
=\overline{(\Phi_{\alpha}(f)(a_1, \cdots, a_n), \alpha)}
\end{multline*}

so the following diagram commutes:

$$\xymatrixcolsep{5pc}\xymatrix{
(A_{\alpha} \times \{ \alpha\})^n \ar[r]^{\eta_{\alpha}^n} \ar[d]_{(\Phi_{\alpha} \times {\rm id}_{\alpha})(f)} & A^n \ar[d]^{\Phi(f)}\\
A_{\alpha} \times \{ \alpha \} \ar[r]_{\eta_{\alpha}} & A
}$$

Now we take, for each $\alpha \in I$ the following $\mathcal{C}^{\infty}-$homomorphism:

$$\lambda_{\alpha}:= \eta_{\alpha} \circ \iota_{\alpha}: A_{\alpha} \to A$$
where $\iota_{\alpha}: A_{\alpha} \rightarrow A_{\alpha} \times \{ \alpha\}$ is the $\mathcal{C}^{\infty}-$isomorphism described in \textbf{Lemma \ref{Gloria}}.\\

\textbf{Claim:} $(\forall \alpha \in I)(\forall \beta \in I)(\alpha \leq \beta \to \lambda_{\beta} \circ \mu_{\alpha \beta} = \lambda_{\alpha})$, that is, the following diagram commutes:

$$\xymatrixcolsep{5pc}\xymatrix{
 & A & \\
 A_{\alpha} \times \{ \alpha\} \ar[ur]^{\eta_{\alpha}} & & A_{\beta} \times \{ \beta\} \ar[ul]_{\eta_{\beta}}\\
 A_{\alpha} \ar[u]_{\iota_{\alpha}} \ar[rr]_{\mu_{\alpha \beta}} & & A_{\beta} \ar[u]^{\iota_{\beta}}
}$$

Indeed, given $a \in A_{\alpha}$, we have:

$$\lambda_{\beta} \circ \mu_{\alpha \beta}(a) = \eta_{\beta} \circ \iota_{\beta}(\mu_{\alpha \beta}(a)) = \eta_{\beta}(\mu_{\alpha \beta}(a), \beta) = \overline{(\mu_{\alpha \beta}(a), \beta)}$$
and

$$\lambda_{\alpha}(a) = \eta_{\alpha}(\iota_{\alpha}(a)) = \eta_{\alpha}(a, \alpha) = \overline{(a, \alpha)}$$

Given any $\gamma \geq \alpha, \beta$, we have $\mu_{\alpha \gamma} = \mu_{\beta \gamma} \circ \mu_{\alpha \beta}$, so:
$$\mu_{\alpha \gamma}(a) = \mu_{\beta \gamma}(\mu_{\alpha \beta}(a))$$
and $((a, \alpha), (\mu_{\alpha \beta}(a), \beta)) \in R$, hence $\lambda_{\beta} \circ \mu_{\alpha \beta}(a) = \lambda_{\alpha}(a)$, and the diagram commutes.\\

Now we need only to show that $((A,\Phi), \lambda_{\alpha})$ has the universal property of the colimit.\\

Let $(B,\Psi)$ be a $\mathcal{C}^{\infty}-$ring and let $\zeta_{\alpha}: A_{\alpha} \to B$ be a family of $\mathcal{C}^{\infty}-$homo\-morphisms such that for every $\alpha, \beta \in I$ with $\alpha \leq \beta$ we have $\zeta_{\beta} \circ \mu_{\alpha \beta} = \zeta_{\alpha}$, that is, the diagram:

$$\xymatrixcolsep{5pc}\xymatrix{
 & B & \\
A_{\alpha} \ar[ur]^{\zeta_{\alpha}} \ar[rr]_{\mu_{\alpha \beta}} & & A_{\beta} \ar[ul]_{\zeta_{\beta}}
}$$

commutes. We are going to show that there is a unique $\mathcal{C}^{\infty}-$homomorphism $\theta: A \to B$ such that the following diagram commutes:

$$\xymatrixcolsep{5pc}
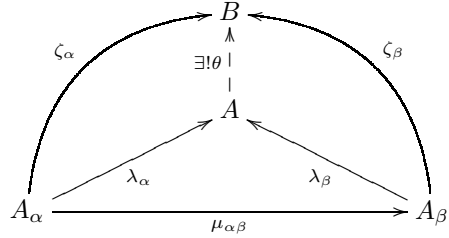
\xymatrix{
  & B & \\
  & A \ar@{-->}[u]^{\exists ! \theta} & \\
 A_{\alpha} \ar[ur]_{\lambda_{\alpha}} \ar@/^2pc/[uur]^{\zeta_{\alpha}} \ar[rr]_{\mu_{\alpha \beta}} & & A_{\beta} \ar[ul]^{\lambda_{\beta}} \ar@/_2pc/[uul]_{\zeta_{\beta}}
}$$

Indeed, given any $\overline{(a, \alpha)} \in A$, we have $(a, \alpha) + R = \lambda_{\alpha}(a) = \lambda_{\beta}(\mu_{\alpha \beta}(a))$, so we define:

$$\begin{array}{cccc}
\theta: & A & \rightarrow & B \\
        & \overline{(a, \alpha)} & \mapsto & \zeta_{\alpha}(a)
\end{array}$$

which is a function, since $(\zeta_{\alpha})_{\alpha \in I}$ is a commutative co-cone.\\

\textbf{Claim:} $\theta: (A,\Phi) \to (B, \Psi)$ is a $\mathcal{C}^{\infty}-$homomorphism.\\

Let $n \in \mathbb{N}$ and $f \in \mathcal{C}^{\infty}(\mathbb{R}^n, \mathbb{R})$. Given $(\overline{(a_1, \alpha_1)}, \cdots, \overline{(a_n, \alpha_n)}) \in A^n$, we have, by definition,

$$\Phi(f)(\overline{(a_1, \alpha_1)}, \cdots, \overline{(a_n, \alpha_n)}) = \overline{(\Phi_{\alpha}(f)(a_1, \cdots, a_n), \alpha)}$$
for every $\alpha \in I$ such that $\alpha \geq \alpha_i$, for every $i \in \{ 1, \cdots, n\}$. Thus:

\begin{multline*}\theta \circ \Phi(f)(\overline{(a_1, \alpha_1)}, \cdots, \overline{(a_n, \alpha_n)}) = \theta(\overline{(\Phi_{\alpha}(f)(\mu_{\alpha_1 \alpha}(a_1), \cdots, \mu_{\alpha_n \alpha}(a_n)), \alpha)}) =\\
=\zeta_{\alpha}(\Phi_{\alpha}(f)(\mu_{\alpha_1 \alpha}(a_1), \cdots, \mu_{\alpha_n \alpha}(a_n)))
\end{multline*}

On the other hand,

\begin{multline*}\Psi(f) \circ \theta^n((a_1, \alpha_1)+R, \cdots, (a_n, \alpha_n)+R) = \Psi(f)(\zeta_{\alpha_1}(a_1), \cdots, \zeta_{\alpha_n}(a_n)) =\\
=\Psi(f)(\zeta_{\alpha}(\mu_{\alpha_1 \alpha}(a_1)), \cdots, \zeta_{\alpha}(\mu_{\alpha_n \alpha}(a_n)))
\end{multline*}

Since for every $ \alpha \in I$, $\zeta_{\alpha}: A_{\alpha} \to B$ is a $\mathcal{C}^{\infty}-$homomorphism, it follows that:

$$\zeta_{\alpha}(\Phi_{\alpha}(f)(\mu_{\alpha_1 \alpha}(a_1), \cdots, \mu_{\alpha_n \alpha}(a_n))) = \Psi(f)(\zeta_{\alpha}(\mu_{\alpha_1 \alpha}(a_1)), \cdots, \zeta_{\alpha}(\mu_{\alpha_n \alpha}(a_n)))$$

so $\theta$ is a $\mathcal{C}^{\infty}-$homomorphism.\\

Now, given $a \in A_{\alpha}$, we have:

$$\theta(\lambda_{\alpha}(a)) = \theta(\overline{(a, \alpha)}) = \zeta_{\alpha}(a)$$

so the required diagram commutes.\\

Since $A = \bigcup_{\alpha \in I}\lambda_{\alpha}[A_{\alpha}]$, $\theta$ is uniquely determined, and the result follows.
\end{proof}

\begin{theorem}\label{projlim}Given any small category $\mathcal{J}$ and any diagram:

$$\begin{array}{cccc}
D: & \mathcal{J} & \rightarrow & \mathcal{C}^{\infty}{\rm \bf Rng}\\
  & (\alpha \stackrel{h}{\rightarrow} \beta) & \mapsto & (A_{\alpha}, \Phi_{\alpha}) \stackrel{D(h)}{\rightarrow} (A_{\beta},\Phi_{\beta})
\end{array}$$

there is a $\mathcal{C}^{\infty}-$ring $(A,\Phi)$ such that:
$$(A,\Phi) \cong \varprojlim_{\alpha \in I} D(\alpha)$$
\end{theorem}
\begin{proof}
Consider the product $\mathcal{C}^{\infty}-$ring:
$$\left(\displaystyle\prod_{i \in I} A_{i}, \Phi^{(I)}\right)$$

and take:

$$A = \{ (a_{\alpha})_{\alpha \in I} \in \displaystyle\prod_{\alpha \in I} A_{\alpha} | (\forall \alpha, \beta \in I)(\forall \alpha \stackrel{h}{\rightarrow} \beta)( D(h)(a_{\alpha}) =a_{\beta})\} \subseteq \displaystyle\prod_{\alpha \in I} A_{\alpha}$$

together with the $\mathcal{C}^{\infty}-$subring structure $\Phi'$ of $\Phi^{(I)}$.\\

We have $((A,\Phi'), \pi_{\alpha}\upharpoonright_{A}: A \to A_{\alpha}) \cong \varprojlim_{\alpha \in I} (A_{\alpha}, \Phi_{\alpha})$
\end{proof}

\begin{remark}
Let $\Sigma = \bigcup_{n \in \mathbb{N}}\mathcal{C}^{\infty}(\mathbb{R}^n, \mathbb{R})$ and let $X = \{ x_1, x_2, \cdots, x_n, \cdots\}$ be a denumerable set of variables, so $F(\Sigma, X)$ will denote the algebra of terms of this language $\Sigma$. A class of \textbf{ordered pairs} will be simply a subset $S \subseteq F(\Sigma, X) \times F(\Sigma, X)$. In our case, these pairs are given by the axioms, so $S$ consists of the following:\\

$\bullet$ For any $n \in \mathbb{R}$, $i \leq n$ and a (smooth) projection map $p_i : \mathbb{R}^n \to \mathbb{R}$ we have:\\

$$(p_i(x_1, \cdots, x_i, \cdots, x_n), x_i) \in S$$

$\bullet$ for every $f, g_1, \cdots, g_n \in \mathcal{C}^{\infty}(\mathbb{R}^m, \mathbb{R})$ and $h \in \mathcal{C}^{\infty}(\mathbb{R}^n, \mathbb{R})$ such that $f = h \circ (g_1, \cdots, g_n)$, we have

$$(h(g_1(x_1, \cdots, x_m), \cdots, g_n(x_1, \cdots, x_m)),f(x_1, \cdots, x_m)) \in S$$
\end{remark}

\begin{remark}\label{Catarina}The class of $\mathcal{C}^{\infty}-$rings is a model of an equational theory, thus it is a variety of algebras. However, if we were not given this information,  noting that the category of $\mathcal{C}^{\infty}-$rings is closed under products, subalgebras and homomorphic images,  the \textbf{HSP Birkhoff's Theorem} would lead us to the same conclusion, that is, that the class $\mathcal{C}^{\infty}{\rm \bf Rng}$ is a variety of algebras, and by the previous remark, $\mathcal{C}^{\infty}{\rm \bf Rng} = V(S)$, the variety of algebras defined by $S$.

In particular, we have some classical results. We list some of them:\\

$\bullet$ for every set $X$ there is a free $\mathcal{C}^{\infty}-$ring determined by $X$;\\

$\bullet$ any $\mathcal{C}^{\infty}-$ring is a homomorphic image of some free $\mathcal{C}^{\infty}-$ring;\\

$\bullet$ a $\mathcal{C}^{\infty}-$homomorphism is monic if, and only if, it is an injective map;\\

$\bullet$ any indexed set of $\mathcal{C}^{\infty}-$rings, $\{ (A_{\alpha}, \Phi_{\alpha}) | \alpha \in I\}$, has a coproduct in $\mathcal{C}^{\infty}{\rm \bf Rng}$.\\

\end{remark}

We end this section by proving that the (variety) of all $\mathcal{C}^{\infty}-$rings is a reflective subcategory of $\mathcal{C}^{\infty}{\rm \bf Str}$.\\

\begin{theorem} \label{adjQV-le}The inclusion functor $\iota: \mathcal{C}^{\infty}{\rm \bf Rng} \hookrightarrow \mathcal{C}^{\infty}{\rm \bf Str}$ has a left adjoint $L: \mathcal{C}^{\infty}{\rm \bf Str} \to \mathcal{C}^{\infty}{\rm \bf Rng}$: given by $M \mapsto M/\theta_M$ where $\theta_M$ is the least $\mathcal{C}^{\infty}$-congruence of $M$ such that $M/\theta_M \in {\rm Obj}(\mathcal{C}^{\infty}{\rm \bf Rng})$. Moreover, the unit of the adjunction $L \dashv \iota$ has components $(q_M)_{M \in {\rm Obj}(\mathcal{C}^{\infty}{\rm \bf Str})}$, where $q_M : M \twoheadrightarrow M/\theta_M$ is the quotient homomorphism.
\end{theorem}

\begin{proof}
Let $(M, \mu)$ be any $\mathcal{C}^{\infty}-$structure and let $|M|$ denote the underlying set of $(M,\mu)$.\\

Consider $\Gamma_{M}=\{\theta\subseteq|M|\times|M|;$ it is a congruence relation and $M/\theta \in {\rm Obj}(\mathcal{C}^{\infty}{\rm \bf Rng})\}$. $\Gamma$ is non-empty, since $\theta=|M|\times|M|$ is a congruence relation and $M/\theta = \{ *\} \cong \mathcal{C}^{\infty}(\varnothing) \in {\rm Obj}(\mathcal{C}^{\infty}{\rm \bf Rng})$. Let $\theta_{M}=\bigcap\Gamma_{M}$.  We will show first that $\theta_{M}\in \Gamma_{M}$. Since  $\theta_M$ is a $\mathcal{C}^{\infty}$-congruence in $M$, it remains to check that $M/\theta_M \in {\rm Obj}(\mathcal{C}^{\infty}{\rm \bf Rng})$.\\


Consider the {\em diagonal} $\mathcal{C}^{\infty}$-morphism:

$$\begin{array}{cccc}
  \delta_M :  & M & \rightarrow  & \prod_{\theta\in\Gamma_{M}}M/\theta\\
  & m & \mapsto & ([m]_{\theta})_{\theta\in\Gamma_{M}}
\end{array}$$


We have $(m,n)\in \ker(\delta_M)\Leftrightarrow ([m]_{\theta})_{\theta\in \Gamma_{M}}=([n]_{\theta})_{\theta\in \Gamma_{M}}\Leftrightarrow (\forall\ \theta\in\Gamma_{M})([m]_{\theta}=[n]_{\theta}) \Leftrightarrow (\forall\ \theta\in\Gamma_{M})(m\theta n)\ \Leftrightarrow m\theta_{M} n$.\\

Thus, by the {\em homomorphism theorem} for $\mathcal{C}^{\infty}{\rm \bf Str}$, there is a unique $\mathcal{C}^{\infty}$-\underline{mono}morphism $\bar{\delta}_M : M/\theta_M \rightarrowtail \prod_{\theta\in\Gamma_{M}}M/\theta$ such the diagram below commutes
\[\xymatrix{
M\ar[r]^(.3){\delta_M}\ar[d]_{q_M}&\prod_{\theta\in\Gamma_{M}}M/\theta\\
M/\theta_{M}\ar[ur]_{\bar{\delta}_M}
}\]

 Since $\mathcal{C}^{\infty}{\rm \bf Rng}$ is closed under products, we have that $\prod_{\theta\in\Gamma_{M}}M/\theta \in {\rm Obj}(\mathcal{C}^{\infty}{\rm \bf Rng})$. We also have that $\mathcal{C}^{\infty}{\rm \bf Rng}$ is closed under substructures and isomorphisms, then $M/\theta_{M}$ is ${\rm Obj}(\mathcal{C}^{\infty}{\rm \bf Rng})$.\\

Denote by $L(M) := M/\theta_M$. We show that $q_M : M \twoheadrightarrow \iota(L(M))$ satisfies the universal property relatively to $\mathcal{C}^{\infty}$-homomorphisms $f : M \rightarrow \iota(N)$, with $N \in {\rm Obj}(\mathcal{C}^{\infty}{\rm \bf Rng})$.\\

Thus, we obtain an injective $\mathcal{C}^{\infty}$-morphism $\bar{f} : M/\ker(f) \rightarrowtail \iota(N)$. Since $\mathcal{C}^{\infty}{\rm \bf Rng}$ is closed under substructures and isomorphisms, we have that $M/\ker(f) \in {\rm Obj}(\mathcal{C}^{\infty}{\rm \bf Rng})$. Hence $\ker(f)\in\Gamma_{M}$ and $\theta_M \subseteq \ker(f)$. Then, again by the homomorphism theorem, there is a unique homomorphism $\tilde{f} : M/\theta_M \rightarrow N$ such that the following diagram commutes

\[\xymatrix{
M \ar[r]^{q_M}\ar[dr]_{f}&\iota(L(M))\ar[d]^{\iota(\tilde{f})}\\
&\iota(N)
}\]
\end{proof}

\section{Free $\mathcal{C}^{\infty}-$Rings}\label{Petrus}

Let $E$ be any set. We are going to describe the \textbf{free} $\mathcal{C}^{\infty}-$ring determined by $E$.\\

For any set $E$ (finite or infinite), we set:

$$\mathbb{R}^{E}:= \{ f: E \rightarrow \mathbb{R} | f \,\, \mbox{is}\,\, \mbox{a}\,\,\, \mbox{function}\}$$

and

$${\rm Func}\,(\mathbb{R}^{E}, \mathbb{R}) = \{ g: \mathbb{R}^{E} \rightarrow \mathbb{R} | g \,\, \mbox{is}\,\, \mbox{a}\,\,\, \mbox{function}\}.$$

Given any two finite subsets of $E$, $E_i,E_j \subseteq E$, whenever $E_i \subseteq E_j$ we have the restriction map:

$$\begin{array}{cccc}
\mu_{i,j}: & \mathbb{R}^{E_j} & \twoheadrightarrow & \mathbb{R}^{E_i}\\
         & f & \mapsto & f \circ \imath_{i,j} : E_i \rightarrow \mathbb{R}
\end{array}$$

where

$$\begin{array}{cccc}
\imath_{i,j}: & E_i & \hookrightarrow & E_j\\
 & x & \mapsto & x
\end{array}$$

is the inclusion map of $E_i$ into $E_j$, and

$$\begin{array}{cccc}
\iota_{i}: & E_i & \hookrightarrow & E\\
 & x & \mapsto & x
\end{array}$$

is the inclusion map of $E_i$ into $E$.\\

We also have the injective pullback (defined by ``composition'') associated with $\mu_{i,j}$, namely:

$$\begin{array}{cccc}
\widehat{\mu}_{i,j}:& {\rm Func}\,(\mathbb{R}^{E_i}, \mathbb{R}) & \rightarrowtail & {\rm Func}\,(\mathbb{R}^{E_j}, \mathbb{R})\\
                    & (\xymatrix{\mathbb{R}^{E_i} \ar[r]^{g} & \mathbb{R}}) & \mapsto & (\xymatrix{\mathbb{R}^{E_j} \ar[r]^{g \circ \mu_{i,j}} & \mathbb{R}})
\end{array}$$




Analogously that  for each finite subset $E_i$ of $E$ we have the injective map  $\iota_{E_i,E}: {\rm Func}\,(\mathbb{R}^{E_i}, \mathbb{R}) \to {\rm Func}\,(\mathbb{R}^{E}, \mathbb{R})$ given by

$$\begin{array}{cccc}
    \iota_{E_i,E}: & {\rm Func}\,(\mathbb{R}^{E_i}, \mathbb{R}) & \rightarrow & {\rm Func}\,(\mathbb{R}^E, \mathbb{R}) \\
     & (\xymatrix{ \mathbb{R}^{E_i} \ar[r]^{f} & \mathbb{R}}) & \mapsto & \begin{array}{cccc}
                                                                            \widetilde{f}: & \mathbb{R}^E & \rightarrow & \mathbb{R} \\
                                                                             & g & \mapsto & f(g\circ \iota_i: \mathbb{R}^{E_i} \rightarrow \mathbb{R})                             \end{array}
  \end{array}$$

Given $f \in {\rm Func}\,(\mathbb{R}^{E_i}, \mathbb{R})$, on the one hand we have:

$$\iota_{E_j,E}\circ \widehat{\mu_{i,j}}(f) = \iota_{E_j,E}(f\circ \mu_{i,j}) = \widetilde{f \circ \mu_{i,j}},$$

where

$$\begin{array}{cccc}
   \widetilde{f \circ \mu_{i,j}}: & \mathbb{R}^{E} & \rightarrow & \mathbb{R}\\
   & g & \mapsto & f \circ \mu_{i,j}(g \circ \iota_j) = f \circ (g \circ \iota_j \circ \imath_{i,j}) = f(g \circ \iota_i)
\end{array}$$

and on the other hand,

$$\begin{array}{cccc}
   \widetilde{f}: & \mathbb{R}^{E} & \rightarrow & \mathbb{R}\\
   & g & \mapsto & f(g \circ \iota_i)
\end{array}$$

so $(\forall f \in {\rm Func}\,(\mathbb{R}^{E_i}, \mathbb{R}))(\iota_{E_j,E}\circ \widehat{\mu_{i,j}}(f) = \iota_{E_i,E}(f))$.\\

Thus, for every $E_i,E_j$ such that $E_i \subseteq E_j$, the following diagram commutes:

\begin{equation}\label{Icaro}\xymatrixcolsep{5pc}\xymatrix{
{\rm Func}\,(\mathbb{R}^{E_i}, \mathbb{R}) \ar@/_/[dr]_{\iota_{E_i,E}} \ar@{>->}[rr]^{\widehat{\mu}_{i,j}} & & {\rm Func}\,(\mathbb{R}^{E_j}, \mathbb{R}) \ar@/^/[dl]^{\iota_{E_j,E}}\\
  & {\rm Func}\,(\mathbb{R}^E, \mathbb{R})&
}
\end{equation}

that is,

$$\iota_{E_j,E} \circ \widehat{\mu}_{i,j} = \iota_{E_i,E}.$$

Consider the following commutative diagram:

$$\xymatrixcolsep{5pc}\xymatrix{
{\rm Func}\,(\mathbb{R}^{E_i}, \mathbb{R}) \ar@{>->}[dr]_{\widehat{\mu}_i} \ar@{>->}[rr]^{\widehat{\mu}_{i,j}}& & {\rm Func}\,(\mathbb{R}^{E_j}, \mathbb{R}) \ar@{>->}[dl]^{\widehat{\mu}_j}\\
 & \varinjlim_{E' \subseteq_f E}{\rm Func}\,(\mathbb{R}^{E'}, \mathbb{R}) &
}$$

which is the  colimit diagram of the directed system $({\rm Func}\,(\mathbb{R}^{E_i}, \mathbb{R}), \widehat{\mu}_{i,j})$.\\

Recall that, concretely we have:

$$\varinjlim_{E' \subseteq_f E}{\rm Func}\,(\mathbb{R}^{E'}, \mathbb{R}) = \dfrac{\bigcup_{E_i \subseteq_f E} {\rm Func}\,(\mathbb{R}^{E_i}, \mathbb{R})\times \{ E_i\}}{\thicksim}$$

where

$$(f_i,E_i) \thicksim (f_j,E_j) \iff (\exists E_k \subseteq_f E)(E_i \subseteq E_k)(E_j \subseteq E_k)(\widehat{\mu}_{i,k}(f_i) = \widehat{\mu}_{j,k}(f_j))$$

Given the cone \eqref{Icaro}, the universal property of the colimit yields a unique function
$$u: \varinjlim_{E' \subseteq_f E} {\rm Func}\,(\mathbb{R}^{E'},\mathbb{R}) \to {\rm Func}\,(\mathbb{R}^E, \mathbb{R})$$

such that the following diagram commutes:

$$\xymatrixcolsep{5pc}\xymatrix{
{\rm Func}\,(\mathbb{R}^{E_i}, \mathbb{R}) \ar[dr]^{\widehat{\mu}_i} \ar@/_2pc/[ddr]_{\iota_{E_i,E}} \ar@{>->}[rr]^{\widehat{\mu}_{i,j}}& & {\rm Func}\,(\mathbb{R}^{E_j}, \mathbb{R}) \ar[dl]_{\widehat{\mu}_j} \ar@/^2pc/[ddl]^{\iota_{E_j,E}}\\
  & \varinjlim_{E' \subseteq_f E} {\rm Func}\,(\mathbb{R}^{E'}, \mathbb{R}) \ar[d]^{\exists ! u} & \\
 & {\rm Func}\,(\mathbb{R}^E, \mathbb{R}) &
}$$

Now we claim that $u$ is an injective function.\\

Given any $[(f_i,E_i)], [(f_j,E_j)] \in \varinjlim_{E' \subseteq_f E} {\rm Func}\,(\mathbb{R}^{E'}, \mathbb{R})$ such that:

$$u([(f_i,E_i)]) = u([(f_j,E_j)]),$$

and since $\widehat{\mu}_i(f_i) = [(f_i,E_i)]$ and $\widehat{\mu}_j(f_j) = [(f_j, E_j)]$, this is equivalent to:

$$\iota_{E_i,E}(f_i) = \iota_{E_j,E}(f_j),$$

so:

$$(\forall g \in \mathbb{R}^{E})(\widetilde{f_i}(g) = \widetilde{f_j}(g))$$
i.e.,

\begin{equation}\label{europa}
(\forall g \in \mathbb{R}^{E})(f_i(g \circ \iota_i) = f_j(g \circ \iota_j))
\end{equation}

Now, given any finite $E_i$ and $E_j$, we can take $E_k \subseteq E$ such that $E_i \subseteq E_k$ and $E_j\subseteq E_k$, namely $E_k = E_i \cup E_j$. We are going to show  that $\widehat{\mu}_{i,k}(f_i) = \widehat{\mu}_{j,k}(f_j)$.\\

Given any $g: E_i \cup E_j \rightarrow \mathbb{R}$ consider:

$$\begin{array}{cccc}
\check{g}: & E & \rightarrow & \mathbb{R}\\
        & x & \mapsto & \begin{cases}
        g(x), \, \mbox{if}\,\, x \in E_k\\
        0 \,\, \mbox{otherwise}
        \end{cases}
\end{array}$$

and note that $\check{g}\circ \iota_i = g \circ \imath_{ik}$ and $\check{g}\circ \iota_j = g \circ \imath_{jk}$.\\

Thus we have:

$$\widehat{\mu}_{i,k}(f_i)(g) = f_i(g \circ \imath_{ik})= f_i(\check{g}\circ \iota_i) \stackrel{\eqref{europa}}{=} f_j(\check{g}\circ \iota_j) = f_j(g \circ \imath_{jk}) = \widehat{\mu}_{j,k}(f_j)(g)$$

Since $g$ is arbitrary, we have:

$$\widehat{\mu}_{i,k}(f_i) = \widehat{\mu}_{j,k}(f_j)$$

and

$$[(f_i,E_i)]=[(f_j,E_j)],$$

so $u$ is injective.\\

Therefore we can identify $\varinjlim_{E' \subseteq_f E}{\rm Func}\,(\mathbb{R}^{E'}, \mathbb{R})$ with a subset of ${\rm Func}\,(\mathbb{R}^{E}, \mathbb{R})$, namely:
$$u\left[ \varinjlim_{E' \subseteq_f E}{\rm Func}\,(\mathbb{R}^{E'}, \mathbb{R})\right].$$\\

In this sense, we say that $\varinjlim_{E' \subseteq_f E}{\rm Func}\,(\mathbb{R}^{E'}, \mathbb{R})$ consists of all functions $f: \mathbb{R}^{E} \rightarrow \mathbb{R}$ for which there are some finite $E' \subseteq E$ and some $f': \mathbb{R}^{E'}\rightarrow \mathbb{R}$ with $u([(f',E')])=f$. As an ``\textit{abus de langage}'', one says that $\varinjlim_{E' \subseteq_f E}{\rm Func}\,(\mathbb{R}^{E'}, \mathbb{R})$ is the set of all functions $f: \mathbb{R}^{E'} \to \mathbb{R}$ for some $E' \subseteq_f E$.\\

Now we proceed to describe the free $\mathcal{C}^{\infty}-$ring determined by an arbitrary set $E$.\\

Given two finite subsets of $E$, $E_i, E_j \subseteq E$, we set $\widehat{\mu}_{i,j}\upharpoonright_{\mathcal{C}^{\infty}(\mathbb{R}^{E_i})}: \mathcal{C}^{\infty}(\mathbb{R}^{E_i}) \rightarrowtail \mathcal{C}^{\infty}(\mathbb{R}^{E_j})$

$$\begin{array}{cccc}
\widehat{\mu}_{i,j}\upharpoonright_{\mathcal{C}^{\infty}(\mathbb{R}^{E_i})}: & \mathcal{C}^{\infty}(\mathbb{R}^{E_i})& \rightarrow & \mathcal{C}^{\infty}(\mathbb{R}^{E_j})\\
& (\mathbb{R}^{E_i} \stackrel{\widehat{g}}{\rightarrow} \mathbb{R}) & \mapsto & \begin{array}{cccc}
\widehat{g}\circ \mu_{i,j}: & \mathbb{R}^{E_j} & \rightarrow & \mathbb{R}\\
 & f & \mapsto & \widehat{g}(f\circ \imath_{ij}: \mathbb{R}^{E_i} \rightarrow \mathbb{R})
\end{array}
\end{array}$$

to be the restriction of the map $\widehat{\mu}_{i,j}$ to $\mathcal{C}^{\infty}(\mathbb{R}^{E_i})$, so the following rectangle commutes

$$\xymatrixcolsep{5pc}\xymatrix{
\mathcal{C}^{\infty}(\mathbb{R}^{E_i}) \ar[rr]^{\widehat{\mu}_{i,j}\upharpoonright_{\mathcal{C}^{\infty}(\mathbb{R}^{E_i}, \mathbb{R})}} \ar@{^{(}->}[d]_{\iota_i} & & \mathcal{C}^{\infty}(\mathbb{R}^{E_j}) \ar@{^{(}->}[d]^{\iota_j}\\
{\rm Func}\,(\mathbb{R}^{E_i}, \mathbb{R})  \ar@{>->}[rr]^{\widehat{\mu}_{i,j}}& & {\rm Func}\,(\mathbb{R}^{E_j}, \mathbb{R})
}$$

and consider the following colimit diagram in ${\rm \bf Set}$:

$$\xymatrixcolsep{5pc}\xymatrix{
\mathcal{C}^{\infty}\,(\mathbb{R}^{E_i}) \ar@{>->}[dr]_{\ell_i} \ar@{>->}[rr]^{\widehat{\mu}_{i,j}\upharpoonright_{\mathcal{C}^{\infty}\,(\mathbb{R}^{E_i})}}& & \mathcal{C}^{\infty}\,(\mathbb{R}^{E_j}) \ar@{>->}[dl]^{\ell_j}\\
 & \varinjlim_{E' \subseteq_f E}\mathcal{C}^{\infty}\,(\mathbb{R}^{E'}) &
}$$

Recall that:

$$\varinjlim_{E' \subseteq_f E}\mathcal{C}^{\infty}\,(\mathbb{R}^{E'}) = \dfrac{\bigcup_{E_i \subseteq_f E}\mathcal{C}^{\infty}(\mathbb{R}^{E_i}) \times \{ E_i\}}{\thicksim}$$

where

\begin{multline*}(f_i,E_i)\thicksim (f_j,E_j) \iff \\
\iff (\exists E_k \subseteq_f E)(E_i \subseteq E_k)(E_j \subseteq E_k)\left( \widehat{\mu}_{i,k}\upharpoonright_{\mathcal{C}^{\infty}(\mathbb{R}^{E_i})}(f_i)=\widehat{\mu}_{j,k}\upharpoonright_{\mathcal{C}^{\infty}(\mathbb{R}^{E_j})}(f_j)\right)\end{multline*}

so:

$$\begin{array}{cccc}
\ell_i: & \mathcal{C}^{\infty}(\mathbb{R}^{E_i}) & \rightarrow & \varinjlim_{E' \subseteq_f E}\mathcal{C}^{\infty}\,(\mathbb{R}^{E'})\\
 & (\mathbb{R}^{E_i} \stackrel{f_i}{\rightarrow} \mathbb{R}) & \mapsto & [(f_i, E_i)]
\end{array}$$

We now claim that $\widehat{\mu}_{i,j}\upharpoonright_{\mathcal{C}^{\infty}(\mathbb{R}^{E_i})}: (\mathcal{C}^{\infty}(\mathbb{R}^{E_i}), \Phi_{E_i}) \rightarrow (\mathcal{C}^{\infty}(\mathbb{R}^{E_j}), \Phi_{E_j})$ is a $\mathcal{C}^{\infty}-$homomorphism.\\

Given any $f \in \mathcal{C}^{\infty}(\mathbb{R}^{n}, \mathbb{R})$, we are going to show that the following diagram commutes:

$$\xymatrixcolsep{5pc}\xymatrix{
\mathcal{C}^{\infty}(\mathbb{R}^{E_i})^n \ar[d]_{\Phi_{E_i}(f)} \ar[r]^{{\widehat{\mu}_{i,j}\upharpoonright_{\mathcal{C}^{\infty}(\mathbb{R}^{E_i})}}^n} & \mathcal{C}^{\infty}(\mathbb{R}^{E_j})^n \ar[d]^{\Phi_{E_j}(f)}\\
\mathcal{C}^{\infty}(\mathbb{R}^{E_i}) \ar[r]_{{\widehat{\mu}_{i,j}\upharpoonright_{\mathcal{C}^{\infty}(\mathbb{R}^{E_i})}}} & \mathcal{C}^{\infty}(\mathbb{R}^{E_j})
}$$

Let $(\widehat{g}_1, \cdots, \widehat{g}_n) \in \mathcal{C}^{\infty}(\mathbb{R}^{E_i})^n$. On the one hand we have:

$$\Phi_{E_j}(f)\circ {\widehat{\mu}_{i,j}\upharpoonright_{\mathcal{C}^{\infty}(\mathbb{R}^{E_i})}}^n(\widehat{g}_1, \cdots, \widehat{g}_n) = \Phi_{E_j}(f)(\widehat{\mu}_{i,j}\upharpoonright_{\mathcal{C}^{\infty}(\mathbb{R}^{E_i})}(\widehat{g}_1), \cdots, \widehat{\mu}_{i,j}\upharpoonright_{\mathcal{C}^{\infty}(\mathbb{R}^{E_i})}(\widehat{g}_n))) = $$
$$= \Phi_{E_j}(f)(\widehat{g}_1 \circ \mu_{ij}, \cdots, \widehat{g}_n \circ \mu_{i,j})$$

and on the other hand we have:

$$\widehat{\mu}_{i,j}\upharpoonright_{\mathcal{C}^{\infty}(\mathbb{R}^{E_i})} \circ \Phi_{E_i}(f)(\widehat{g}_1, \cdots, \widehat{g}_n) = \widehat{\mu}_{i,j}\upharpoonright_{\mathcal{C}^{\infty}(\mathbb{R}^{E_i})}(\widehat{f \circ (g_1, \cdots, g_n)})=$$
$$= (\widehat{f \circ (g_1, \cdots, g_n)}) \circ \mu_{i,j}$$

Given any $h \in \mathbb{R}^{E_j}$, we have, on the one side:

$$\Phi_{E_j}(f)(\widehat{g}_1 \circ \mu_{ij}, \cdots, \widehat{g}_n \circ \mu_{i,j})(h) = \Phi_{E_j}(f)(\widehat{g}_1 (h \circ \imath_{ij}), \cdots, \widehat{g}_n (h \circ \imath_{i,j}))=$$
$$=\Phi_{E_j}(f)(\widehat{g}_1, \cdots, \widehat{g}_n)(h \circ \imath_{ij})= \widehat{f \circ (g_1, \cdots, g_n)}(h \circ \imath_{ij})$$

and on the other hand:

$$[\widehat{\mu}_{ij}\upharpoonright_{\mathcal{C}^{\infty}(\mathbb{R}^{E_i})} \circ \Phi_{E_i}(f)(\widehat{g}_1, \cdots, \widehat{g}_n)](h) = [\widehat{\mu}_{ij}\upharpoonright_{\mathcal{C}^{\infty}(\mathbb{R}^{E_i})} \circ (\widehat{f \circ (g_1, \cdots, g_n)})](h) =$$
$$= [\widehat{f \circ (g_1, \cdots, g_n)}\circ \mu_{ij}](h) = \widehat{f \circ (g_1, \cdots, g_n)}(h \circ \imath_{ij})$$

so $\widehat{\mu}_{ij}\upharpoonright_{\mathcal{C}^{\infty}(\mathbb{R}^{E_i})}$ is a $\mathcal{C}^{\infty}-$homomorphism.\\

We have, thus, the directed system $((\mathcal{C}^{\infty}(\mathbb{R}^{E_i}),\Phi_{E_i}), \widehat{\mu}_{ij}\upharpoonright_{\mathcal{C}^{\infty}(\mathbb{R}^{E_i})})$ of $\mathcal{C}^{\infty}-$rings, and the following commutative diagrams (for any $E_i,E_j \subseteq_f E$) in $\mathcal{C}^{\infty}{\rm \bf Rng}$:

$$\xymatrixcolsep{5pc}\xymatrix{
(\mathcal{C}^{\infty}\,(\mathbb{R}^{E_i}),\Phi_{E_i}) \ar@{>->}[dr]_{\ell_i} \ar@{>->}[rr]^{\widehat{\mu}_{i,j}\upharpoonright_{\mathcal{C}^{\infty}\,(\mathbb{R}^{E_i})}}& & (\mathcal{C}^{\infty}\,(\mathbb{R}^{E_j}),\Phi_{E_j}) \ar@{>->}[dl]^{\ell_j}\\
 & \varinjlim_{E' \subseteq_f E}(\mathcal{C}^{\infty}\,(\mathbb{R}^{E'}), \Phi_{E'}) &
}$$

Note that each $\ell_i: (\mathcal{C}^{\infty}\,(\mathbb{R}^{E_i}),\Phi_{E_i}) \rightarrow \varinjlim_{E' \subseteq_f E}(\mathcal{C}^{\infty}\,(\mathbb{R}^{E'}), \Phi_{E'})$ is a $\mathcal{C}^{\infty}-$homomorphism.\\

Let us examine this colimit more closely. We denote:

$$\mathcal{C}^{\infty}(\mathbb{R}^{E}) = \varinjlim_{E' \subseteq_f E}\mathcal{C}^{\infty}\,(\mathbb{R}^{E'})$$

and describe the $\mathcal{C}^{\infty}-$structure on $\mathcal{C}^{\infty}(\mathbb{R}^{E})$,

$$\begin{array}{cccc}
\Phi_E: & \bigcup_{n \in \mathbb{N}}\mathcal{C}^{\infty}(\mathbb{R}^n, \mathbb{R})& \rightarrow & \bigcup_{n \in \mathbb{N}} {\rm Func}\,(\mathcal{C}^{\infty}(\mathbb{R}^{E})^n, \mathcal{C}^{\infty}(\mathbb{R}^{E}))\\
 & (\mathbb{R}^n \stackrel{f}{\rightarrow} \mathbb{R}) & \mapsto & \Phi_E(f): \mathcal{C}^{\infty}(\mathbb{R}^{E})^n \rightarrow \mathcal{C}^{\infty}(\mathbb{R}^{E})
\end{array}$$

where:

$$\begin{array}{cccc}
\Phi_E(f):& \mathcal{C}^{\infty}(\mathbb{R}^{E})^n & \rightarrow & \mathcal{C}^{\infty}(\mathbb{R}^{E})\\
         & ([(\widehat{f}_1,E_1)], \cdots, [(\widehat{f}_n,E_n)]) & \mapsto & [(f\circ (\widehat{\mu}_{1k}\upharpoonright_{\mathcal{C}^{\infty}(\mathbb{R}^{E_1})}(\widehat{f}_1), \cdots, \widehat{\mu}_{nk}\upharpoonright_{\mathcal{C}^{\infty}(\mathbb{R}^{E_n})}(\widehat{f}_n)), E_k)]
\end{array}$$

and $E_k = \cup_{i=1}^{n}E_i$. We have thus:

$$f \circ (\widehat{\mu}_{1k}\upharpoonright_{\mathcal{C}^{\infty}(\mathbb{R}^{E_1})}(\widehat{f}_1), \cdots, \widehat{\mu}_{nk}\upharpoonright_{\mathcal{C}^{\infty}(\mathbb{R}^{E_n})}(\widehat{f}_n)): \mathbb{R}^{E_k} \rightarrow \mathbb{R},$$

which belongs to $\mathcal{C}^{\infty}(\mathbb{R}^{E_k})$, since $\widehat{\mu}_{1k}\upharpoonright_{\mathcal{C}^{\infty}(\mathbb{R}^{E_1})}(\widehat{f}_1), \cdots, \widehat{\mu}_{nk}\upharpoonright_{\mathcal{C}^{\infty}(\mathbb{R}^{E_n})}(\widehat{f}_n) \in \mathcal{C}^{\infty}(\mathbb{R}^{E_k})$ and $f \in \mathcal{C}^{\infty}(\mathbb{R}^n)$.\\

Thus we have the $\mathcal{C}^{\infty}-$ring:

$$(\mathcal{C}^{\infty}(\mathbb{R}^{E}),\Phi_E)$$

Let $E$ be any set, and write

$$\xymatrixcolsep{5pc}\xymatrix{
 & \varinjlim_{E' \subseteq_f E} E' & \\
E' \ar[ur]^{\iota_{E'}} \ar[rr]_{\iota^{E'}_{E''}}& & E'' \ar[ul]_{\iota_{E''}}
}$$

For every finite subset $E' \subseteq_f E$, we have the free $\mathcal{C}^{\infty}-$ring:

$$\jmath_{E'}: E' \rightarrow U(\mathcal{C}^{\infty}(\mathbb{R}^{E'}), \Phi_{E'}).$$

so we can form the commutative cone:

$$\xymatrixcolsep{5pc}\xymatrix{
 & \mathcal{C}^{\infty}(\mathbb{R}^{E}) & \\
E' \ar[ur]^{\ell_{E'}\circ \jmath_{E'}} \ar[rr]_{\iota^{E''}_{E'}}& & E'' \ar[ul]_{\ell_{E''}\circ \jmath_{E''}}
}$$

The universal property of $\varinjlim_{E' \subseteq_f E} E'$ yields a unique function  $\jmath_E: \varinjlim_{E' \subseteq_f E}E' \rightarrow \mathcal{C}^{\infty}(\mathbb{R}^{E})$ such that the following prism commutes in ${\rm \bf Set}$:

$$\xymatrixcolsep{8pc}\xymatrix @!0 @R=4pc @C=6pc {
    E' \ar[rr]^{\iota_{E'}} \ar[rd]_{\iota^{E''}_{E'}} \ar[ddd]_{\jmath_{E'}} && \varinjlim_{E' \subseteq_f E}E' \ar@{-->}[ddd]^{\exists ! \jmath_E} \\
    & E'' \ar[ru]_{\iota_{E''}} \ar[ddd]^{\jmath_{E''}}&\\
    & \\
    \mathcal{C}^{\infty}(\mathbb{R}^{E'}) \ar[rr]^(.25){\ell_{E'}} |!{[ur];[dr]}\hole \ar[rd]_{\widehat{\mu}_{E'E''}\upharpoonright_{\mathcal{C}^{\infty}(\mathbb{R}^{E'})}} && \mathcal{C}^{\infty}(\mathbb{R}^{E}) \\
    & \mathcal{C}^{\infty}(\mathbb{R}^{E''}) \ar[ru]_{\ell_{E''}} & }$$

In ${\rm \bf Set}$ we have:

$$\varinjlim_{E' \subseteq_f E}E' \cong \bigcup_{E' \subseteq_f E} E' = E$$

so we have the function $\jmath_E: E \rightarrow \mathcal{C}^{\infty}(\mathbb{R}^{E})$.

\begin{remark}\label{Benicio}Given a set $E$, for every finite subset $E' \subseteq_f E$, with $\sharp E' = n$, there is a bijection $\omega_n: \{ 1, \cdots, n\} \rightarrow E'$ and a corresponding $\mathcal{C}^{\infty}-$isomorphism:

$$\widehat{\omega_n}: (\mathcal{C}^{\infty}(\mathbb{R}^{E'}), \Phi_{E'}) \rightarrow (\mathcal{C}^{\infty}(\mathbb{R}^n), \Omega).$$

Also, given $E'' \subseteq_f E$ and a bijection $\omega_{m}: \{ 1, \cdots, m\} \to E''$, with $E' \subseteq E''$ (so  with $n \leq m$) we have the composite:

$$ \widehat{\mu}_{nm} := \widehat{\omega_m}^{-1}\circ \widehat{\mu}_{E'E''}\upharpoonright_{\mathcal{C}^{\infty}(\mathbb{R}^{E'})} \circ \widehat{\omega_n}: \mathcal{C}^{\infty}(\mathbb{R}^n) \to \mathcal{C}^{\infty}(\mathbb{R}^m) $$

so the following rectangle commutes:

$$\xymatrixcolsep{5pc}\xymatrix{
\mathcal{C}^{\infty}(\mathbb{R}^{E'}) \ar[r]^{\widehat{\mu}_{E'E''}\upharpoonright_{\mathcal{C}^{\infty}(\mathbb{R}^{E'})}} & \mathcal{C}^{\infty}(\mathbb{R}^{E''})\\
\mathcal{C}^{\infty}(\mathbb{R}^n) \ar[u]_{\widehat{\omega_n}} \ar[r]_{\widehat{\mu_{nm}}} & \mathcal{C}^{\infty}(\mathbb{R}^m) \ar[u]^{\widehat{\omega_m}}
}$$



\end{remark}

Considering $E_i \subseteq E_j$ and the inclusion maps $\iota_i : \mathcal{C}^{\infty}(\mathbb{R}^{E_i}) \hookrightarrow {\rm Func}\,(\mathbb{R}^{E_i})$,  we have the following commutative diagram:

$$\xymatrixcolsep{5pc}\xymatrix{
\mathcal{C}^{\infty}(\mathbb{R}^{E_i}) \ar[rr]^{\widehat{\mu}_{i,j}\upharpoonright_{\mathcal{C}^{\infty}(\mathbb{R}^{E_i})}} \ar@{^{(}->}[d]_{\iota_i} & & \mathcal{C}^{\infty}(\mathbb{R}^{E_j}) \ar@{^{(}->}[d]^{\iota_j}\\
{\rm Func}\,(\mathbb{R}^{E_i}, \mathbb{R}) \ar[dr]^{\widehat{\mu}_i}  \ar@{>->}[rr]^{\widehat{\mu}_{i,j}}& & {\rm Func}\,(\mathbb{R}^{E_j}, \mathbb{R}) \ar[dl]_{\widehat{\mu}_j} \\
  & \varinjlim_{E' \subseteq_f E} {\rm Func}\,(\mathbb{R}^{E'}, \mathbb{R}) & \\
}$$

By the universal property of $\varinjlim_{E_i \subseteq_f E} \mathcal{C}^{\infty}(\mathbb{R}^{E_i},\mathbb{R})$, in the category ${\rm \bf Set}$, given the cone:
$$(\widehat{\mu}_i \circ \iota_i : \mathcal{C}^{\infty}(\mathbb{R}^{E_i}, \mathbb{R}) \to \varinjlim_{E' \subseteq_f E} {\rm Func}\,(\mathbb{R}^{E'}, \mathbb{R})),$$

there is a unique function $v: \varinjlim_{E' \subseteq_f E} \mathcal{C}^{\infty}(\mathbb{R}^{E'}, \mathbb{R}) \to \varinjlim_{E' \subseteq_f E}{\rm Func}\,(\mathbb{R}^{E'}, \mathbb{R})$ such that the following prism commutes:

$$ \xymatrixrowsep{5pc}\xymatrixcolsep{5pc}\xymatrix @!0 @R=4pc @C=6pc {
    \mathcal{C}^{\infty}(\mathbb{R}^{E_i}, \mathbb{R}) \ar@{>->}[rr]^{\widehat{\mu}_{i,j}\upharpoonright_{\mathcal{C}^{\infty}(\mathbb{R}^{E_i}, \mathbb{R})}} \ar[rd]^{\ell_i} \ar@{^{(}->}[dd]_{\iota_i} && \mathcal{C}^{\infty}(\mathbb{R}^{E_j}, \mathbb{R}) \ar@{^{(}->}[dd]^{\iota_j} \ar[ld]_{\ell_j} \\
    & \varinjlim_{E' \subseteq_f E} \mathcal{C}^{\infty}(\mathbb{R}^{E'}, \mathbb{R}) \ar@{>->}[dd]^(.25){v} \\
    {\rm Func}\,(\mathbb{R}^{E_i}, \mathbb{R}) \ar@{>->}[rr]|!{[ur];[dr]}\hole \ar[rd]^{\widehat{\mu}_i} && {\rm Func}\,(\mathbb{R}^{E_j}, \mathbb{R}) \ar[ld]_{\widehat{\mu}_j} \\
    & \varinjlim_{E' \subseteq_f E}{\rm Func}\,(\mathbb{R}^{E'}, \mathbb{R})}$$

where $\ell_i: \mathcal{C}^{\infty}(\mathbb{R}^{E_i}, \mathbb{R}) \to \varinjlim_{E' \subseteq E} \mathcal{C}^{\infty}(\mathbb{R}^{E'}, \mathbb{R})$ are the canonic colimit arrows.\\

We claim that $v = \iota: \varinjlim_{E' \subseteq_f E} \mathcal{C}^{\infty}(\mathbb{R}^{E'}, \mathbb{R}) \hookrightarrow \varinjlim_{E' \subseteq_f E}{\rm Func}\,(\mathbb{R}^{E'}, \mathbb{R})$ (the inclusion).\\

Given $f \in \varinjlim_{E'\subseteq_f E}\mathcal{C}^{\infty}(\mathbb{R}^{E'}, \mathbb{R})$, there is some $E_k \subseteq_f E$ and some $f_k \in \mathcal{C}^{\infty}(\mathbb{R}^{E_k}, \mathbb{R})$ such that:
$$f= [(f_k, E_k)] = \ell_k(f_k)$$
We have:

$$v(f) = v(\ell_k(f_k)) = \widehat{\mu}_k \circ \iota_k(f) = \widehat{\mu}_k(f) = [(f_k,E_k)] =f$$

So we have $\varinjlim_{E' \subseteq_f E} \mathcal{C}^{\infty}(\mathbb{R}^{E'}, \mathbb{R}) \subseteq \varinjlim_{E' \subseteq_f E}{\rm Func}\,(\mathbb{R}^{E'}, \mathbb{R}) \rightarrowtail {\rm Func}\,(\mathbb{R}^{E}, \mathbb{R})$.\\

Hence, we can regard $\varinjlim_{E' \subseteq_f E} \mathcal{C}^{\infty}(\mathbb{R}^{E'}, \mathbb{R})$ as (an actual) subset of $$\varinjlim_{E' \subseteq_f E}{\rm Func}\,(\mathbb{R}^{E'}, \mathbb{R}),$$ and consider, as in \cite{rings2}, $(\mathcal{C}^{\infty}(\mathbb{R}^{E}),\Phi_E)$ as the ring of functions $\mathbb{R}^{E} \to \mathbb{R}$ which smoothly depend on finitely many variables only.\\

The results given in the previous section assure the existence of some constructions within the category of $\mathcal{C}^{\infty}-$rings, such as quotients, products, coproducts and so on.\\

As we shall see, the category of $\mathcal{C}^{\infty}-$rings holds a strong relation with rings of the form $\mathcal{C}^{\infty}(\mathbb{R}^n)$.\\

In this section we are going to describe concretely such constructions in $\mathcal{C}^{\infty}{\rm \bf Rng}$.\\

Our definition of $\mathcal{C}^{\infty}-$ring yields a forgetful functor:

$$\begin{array}{cccc}
U: & \mathcal{C}^{\infty}{\rm \bf Rng} & \rightarrow & {\rm \bf Set}\\
   & (A,\Phi) & \mapsto & A\\
   & ((A,\Phi) \stackrel{\varphi}{\rightarrow} (B,\Psi)) & \mapsto & (A \stackrel{U(\varphi)}{\rightarrow} B)
\end{array}$$


We are going to show that this functor has a left adjoint, the ``\index{free $\mathcal{C}^{\infty}-$ring}free $\mathcal{C}^{\infty}-$ring'', that we shall denote by $L: {\rm \bf Set} \rightarrow \mathcal{C}^{\infty}{\rm \bf Rng}$. Before we do it, we need the following:\\

\begin{remark}\label{Nikos}
Given any $m \in \mathbb{N}$, we note that the set $\mathcal{C}^{\infty}(\mathbb{R}^n)$ may be endowed with a $\mathcal{C}^{\infty}-$structure:

$$\begin{array}{cccc}
    \Omega: & \bigcup_{n \in \mathbb{N}} \mathcal{C}^{\infty}(\mathbb{R}^n, \mathbb{R}) & \rightarrow & \bigcup_{n \in \mathbb{N}} {\rm Func}\,(\mathcal{C}^{\infty}(\mathbb{R}^m)^n, \mathcal{C}^{\infty}(\mathbb{R}^m))\\
    & (\mathbb{R}^n \stackrel{f}{\rightarrow} \mathbb{R}) & \mapsto & \begin{array}{cccc}
    \Omega(f) = f \circ - : & \mathcal{C}^{\infty}(\mathbb{R}^m)^n & \rightarrow & \mathcal{C}^{\infty}(\mathbb{R}^m)\\
      & (h_1, \cdots, h_n) & \mapsto & f \circ (h_1, \cdots, h_n)
    \end{array}
\end{array}$$

so it is easy to see that it can be made into a $\mathcal{C}^{\infty}-$ring $(\mathcal{C}^{\infty}(\mathbb{R}^m), \Omega)$. From now, when dealing with this``canonical'' $\mathcal{C}^{\infty}-$structure,  we shall omit the symbol $\Omega$, writting $\mathcal{C}^{\infty}(\mathbb{R}^m)$ instead of $(\mathcal{C}^{\infty}(\mathbb{R}^m), \Omega)$.\\
\end{remark}

In fact, we are going to show that for finite sets $X$ with $\sharp X = m$, for instance, the $\mathcal{C}^{\infty}-$ring given in the previous remark is (up to isomorphism) the free $\mathcal{C}^{\infty}-$ring on $m$ generators.\\

Indeed, given any finite set $X$ with $\sharp X = m \in \mathbb{N}$, we define:

$$L(X):= \mathcal{C}^{\infty}(\mathbb{R}^m)$$

together with the canonical $\mathcal{C}^{\infty}-$structure $\Omega$ given in \textbf{Remark \ref{Nikos}}.\\

We have the following:\\

\begin{proposition}\label{Kakeu}Let $U: \mathcal{C}^{\infty}{\rm \bf Rng} \rightarrow {\rm \bf Set}$, $(A,\Phi) \mapsto A$ , be the forgetful functor. The pair $(\jmath_n, (\mathcal{C}^{\infty}(\R^n),\Omega))$, where:
$$\begin{array}{cccc}
 \jmath_n: & \{1, \cdots, n \} & \rightarrow & U(\mathcal{C}^{\infty}(\mathbb{R}^n),\Omega)\\
           & i & \mapsto & \pi_i: \mathbb{R}^n \rightarrow \mathbb{R}
\end{array},$$

is the free $\mathcal{C}^{\infty}$-ring with $n$ generators, which are the projections:

$$\begin{array}{cccc}
    \pi_i: & \mathbb{R}^n & \rightarrow & \mathbb{R} \\
     & (x_1, \cdots, x_i, \cdots, x_n) & \mapsto & x_i
  \end{array}$$
\end{proposition}
\begin{proof}
(Cf. \textbf{Proposition 1.1} of \cite{MoerdijkReyess}.)\\

Given any $\mathcal{C}^{\infty}-$ring $(A, \Phi)$ and any function $\alpha : \{ 1,2, \cdots, n \} \rightarrow U(A, \Phi)$, we are going to show that there is a unique $\mathcal{C}^{\infty}-$homomorphism $\widetilde{\alpha}:(\mathcal{C}^{\infty}(\mathbb{R}^n), \Omega) \to (A,\Phi)$ such that the following diagram commutes:
$$\xymatrixcolsep{5pc}\xymatrix{
\{1,2, \cdots, n \} \ar[r]^{\jmath_n} \ar[dr]_{\alpha} & U(\mathcal{C}^{\infty}(\mathbb{R}^n),\Omega) \ar@{.>}[d]^{ U(\widetilde{\alpha})}\\
  & U(A,\Phi)
}$$

Given $\alpha(1), \alpha(2), \cdots, \alpha(n) \in A$, define

$$\begin{array}{cccc}
    \widetilde{\alpha}: & \mathcal{C}^{\infty}(\mathbb{R}^n) & \rightarrow & A \\
     & f & \mapsto & \Phi(f)(\alpha(1), \alpha(2), \cdots, \alpha(n))
  \end{array}$$

Note that such a function satisfies $(\forall i \in \{ 1,2, \cdots, n\})(\widetilde{\alpha}(\jmath_n(i))= \alpha(i))$, since for every $i \in \{1,2, \cdots, n\}$:
\begin{multline*}\widetilde{\alpha}(\jmath_n(i)) = \widetilde{\alpha}(\pi_i) = \Phi(\pi_i)(\alpha(1), \cdots, \alpha(i), \cdots, \alpha(n)) = \\
=\pi_i(\alpha(1), \cdots, \alpha(i), \cdots, \alpha(n)) = \alpha(i).\end{multline*}

Next thing we show is that $\widetilde{\alpha}$ is a $\mathcal{C}^{\infty}-$homomorphism.\\

Let $f: \mathbb{R}^m \to \mathbb{R}$ be any smooth function. We claim that the following diagram commutes:

$$\xymatrixcolsep{5pc}\xymatrix{
\mathcal{C}^{\infty}(\mathbb{R}^n)^m \ar[r]^{\widetilde{\alpha}^m} \ar[d]_{\Omega(f)} & A^m \ar[d]^{\Phi(f)}\\
\mathcal{C}^{\infty}(\mathbb{R}^n) \ar[r]^{\widetilde{\alpha}} & A
}$$

Let $(\varphi_1, \cdots, \varphi_m) \in \mathcal{C}^{\infty}(\mathbb{R}^n)^m$. On the one hand we have:

\begin{multline*}
  \Phi(f) \circ \widetilde{\alpha}^m (\varphi_1, \cdots, \varphi_m) = \Phi(f) \circ (\widetilde{\alpha}(\varphi_1), \cdots, \widetilde{\alpha}(\varphi_m)) = \\
  \Phi(f)(\Phi(\varphi_1)(\alpha(1), \alpha(2), \cdots, \alpha(n)), \cdots, \Phi(\varphi_m)(\alpha(1), \alpha(2), \cdots, \alpha(n)))
\end{multline*}

On the other hand we have:

\begin{multline*}
  \widetilde{\alpha} \circ \mathcal{C}^{\infty}(\mathbb{R}^n)(f)(\varphi_1, \cdots, \varphi_m) = \widetilde{\alpha}(f \circ (\varphi_1, \cdots, \varphi_m)) =\\
  =\Phi(f \circ (\varphi_1, \cdots, \varphi_m))(\alpha(1), \alpha(2), \cdots, \alpha(n))
\end{multline*}

Since $(A, \Phi)$ is a $\mathcal{C}^{\infty}-$ring, we have also:

$$\Phi(f \circ (\varphi_1, \cdots, \varphi_m)) = \Phi(f) \circ (\Phi(\varphi_1), \cdots, \Phi(\varphi_m)),$$

hence:
$$\Phi(f) \circ \widetilde{\alpha}^m (\varphi_1, \cdots, \varphi_m) = \widetilde{\alpha} \circ \Omega(f)(\varphi_1, \cdots, \varphi_m),$$

and $\widetilde{\alpha}$ is indeed a $\mathcal{C}^{\infty}-$homomorphism.\\

For the uniqueness of $\widetilde{\alpha}$, suppose $\Psi :(\mathcal{C}^{\infty}(\mathbb{R}^n),\Omega) \to (A,\Phi)$ is a $\mathcal{C}^{\infty}-$homo\-morphism such that $(\forall i \in \{1,2, \cdots,n \})(\Psi(\jmath_n(i)) = \alpha(i))$. Since $\Psi$ is a $\mathcal{C}^{\infty}-$homo\-morphism, in particular the following diagram commutes:

$$\xymatrixcolsep{5pc}\xymatrix{
\mathcal{C}^{\infty}(\mathbb{R}^n)^n \ar[r]^{\Psi^n} \ar[d]_{\Omega(f)} & A^n \ar[d]^{\Phi(f)}\\
\mathcal{C}^{\infty}(\mathbb{R}^n) \ar[r]^{\Psi} & A
}$$

Note that for any $f \in \mathcal{C}^{\infty}(\mathbb{R}^n)$, $f \circ (\pi_1, \cdots, \pi_n)  =f$, and since the diagram above commutes, we have:

$$\Phi(f)(\Psi(\pi_1), \cdots, \Psi(\pi_n)) = \Psi (f \circ (\pi_1, \cdots, \pi_n)) = \Psi(f)$$

so

\begin{multline*}\Psi(f) = \Phi(f)(\Psi(\pi_1), \cdots, \Psi(\pi_n)) = \\
=\Phi(f)(\Psi(\jmath_n(1)), \cdots, \Psi(\jmath_n(n))) = \Phi(f)(\alpha(1), \cdots, \alpha(n)) = \widetilde{\alpha}(f).\end{multline*}

Thus it is proved that $(\mathcal{C}^{\infty}(\mathbb{R}^n), \Omega)$ is (up to isomorphism) the free $\mathcal{C}^{\infty}-$ring on $n$ generators.

\end{proof}

Given a finite set $X$ with $\sharp X = n$, consider a bijection $\omega: \{ 1, \cdots, n\} \rightarrow X$ and denote, for any $i \in \{ 1, \cdots, n\}$, $x_i = \omega(i)$.\\

We define:

$$\mathbb{R}^{X} = \{ f: X \to \mathbb{R} | f \,\, \mbox{is}\,\, \mbox{a}\,\,\mbox{function} \}.$$

Due to the bijection $\omega$, any smooth function $g \in \mathcal{C}^{\infty}(\mathbb{R}^n, \mathbb{R})$ can be interpreted as:

$$\begin{array}{cccc}
\widehat{g}: & \mathbb{R}^{X} & \rightarrow & \mathbb{R}\\
             & (X \stackrel{f}{\rightarrow} \mathbb{R}) & \mapsto & g(f(x_1), \cdots, f(x_n))
\end{array}$$

so we define:

$$\mathcal{C}^{\infty}(\mathbb{R}^{X}):= \{ \widehat{g}: \mathbb{R}^{X} \rightarrow \mathbb{R} | g \in \mathcal{C}^{\infty}(\mathbb{R}^n,\mathbb{R})\}$$

Note that $\mathcal{C}^{\infty}\,(\mathbb{R}^{X})$ has a natural $\mathcal{C}^{\infty}-$structure given by:

$$\begin{array}{cccc}
\Phi_X: & \bigcup_{n \in \mathbb{N}}\mathcal{C}^{\infty}\,(\mathbb{R}^n, \mathbb{R}) & \rightarrow & \bigcup_{n \in \mathbb{N}}{\rm Func}\,(\mathcal{C}^{\infty}(\mathbb{R}^{X})^n, \mathcal{C}^{\infty}(\mathbb{R}^X))\\
    & (\mathbb{R}^n \stackrel{f}{\rightarrow} \mathbb{R}) & \mapsto & \Phi_X(f): \mathcal{C}^{\infty}(\mathbb{R}^{X})^n \rightarrow \mathcal{C}^{\infty}(\mathbb{R}^{X})
\end{array}$$

where:

$$\begin{array}{cccc}
\Phi_X(f): & \mathcal{C}^{\infty}(\mathbb{R}^{X})^n & \rightarrow & \mathcal{C}^{\infty}(\mathbb{R}^{X})\\
          & (\widehat{g_1}, \cdots, \widehat{g_n}) & \mapsto & \widehat{f \circ (g_1, \cdots, g_n)}: \mathbb{R}^{X} \rightarrow \mathbb{R}
\end{array}$$

which is well-defined since the map $g \mapsto \widehat{g}$ is a bijection for a fixed $\omega$.\\

Thus we have the following $\mathcal{C}^{\infty}-$ring:

$$(\mathcal{C}^{\infty}(\mathbb{R}^{X}), \Phi_X)$$

Note that given any $x \in X$, the evaluation map:

$$\begin{array}{cccc}
          {\rm ev}_x: & \mathbb{R}^{X} & \rightarrow & \mathbb{R}\\
              & f & \mapsto & f(x)
          \end{array}$$

belongs to $\mathcal{C}^{\infty}(\mathbb{R}^{X})$. In fact, given $x \in X = \{ x_1, \cdots, x_n\}$, there is $i \in \{ 1, \cdots, n\}$ such that $x = x_i$.\\

\textbf{Claim:} ${\rm ev}_{x_i} = \widehat{\pi_i}$, where $\pi_i: \mathbb{R}^n \rightarrow \mathbb{R}$ is the (smooth) projection on the $i-$th coordinate.\\

We have, in fact, for any $f \in \mathbb{R}^{X}$, $ {\rm ev}_{x_i}(f) = f(x_i) = \pi_i(f(x_1), \cdots, f(x_i), \cdots, f(x_n)) = \widehat{\pi_i}(f)$, so ${\rm ev}_x \in \mathcal{C}^{\infty}(\mathbb{R}^{X})$.\\

We have, thus, the following function:

$$\begin{array}{cccc}
\jmath_X: & X & \rightarrow & U(\mathcal{C}^{\infty}(\mathbb{R}^{X}),\Phi_X)\\
          & x & \mapsto & \begin{array}{cccc}
          {\rm ev}_x: & \mathbb{R}^{X} & \rightarrow & \mathbb{R}\\
              & f & \mapsto & f(x)
          \end{array}
\end{array}$$

\begin{proposition}\label{Kakau}Let $X$ be a finite set with $\sharp X = n$ and consider:

$$\mathbb{R}^{X} = \{ f: X \to \mathbb{R} | f \,\, \mbox{is}\,\, \mbox{a}\,\,\mbox{function} \}$$

together with some bijection $\omega: \{1, \cdots, n\} \rightarrow X$, $\omega(i)=x_i \in X$.\\

Under those circumstances, the following function is a $\mathcal{C}^{\infty}-$isomorphism:

$$\begin{array}{cccc}
\widehat{\omega}: & (\mathcal{C}^{\infty}(\mathbb{R}^{X}), \Phi_X) & \rightarrow & (\mathcal{C}^{\infty}(\mathbb{R}^n),\Omega)\\
  & \begin{array}{cccc}
  \widehat{g}:& \mathbb{R}^{X} &\to & \mathbb{R}\\
  & f & \mapsto & \widehat{g}(f) \end{array} & \mapsto & g: \mathbb{R}^n \to \mathbb{R}
\end{array}$$
\end{proposition}
\begin{proof}
First we show that $\widehat{\omega}$ is a $\mathcal{C}^{\infty}-$homomorphism.\\

Let $m \in \mathbb{N}$ and $h \in \mathcal{C}^{\infty}(\mathbb{R}^m, \mathbb{R})$. We claim that the following diagram commutes:

$$\xymatrixcolsep{5pc}\xymatrix{
\mathcal{C}^{\infty}(\mathbb{R}^{X})^m \ar[r]^{\widehat{\omega}^m} \ar[d]_{\Phi_X(h)}& \mathcal{C}^{\infty}(\mathbb{R}^n)^m \ar[d]^{\Omega(h)}\\
\mathcal{C}^{\infty}(\mathbb{R}^{X}) \ar[r]_{\widehat{\omega}} & \mathcal{C}^{\infty}(\mathbb{R}^n)
}$$

Given $(\widehat{g_1}, \cdots, \widehat{g_m}) \in \mathcal{C}^{\infty}(\mathbb{R}^{X})^m$, we have:

$$\widehat{\omega}(\Phi_X(h)(\widehat{g_1}, \cdots, \widehat{g_m})) = \widehat{\omega}(\widehat{h \circ (g_1, \cdots,g_m)}) = h \circ (g_1, \cdots, g_m)$$

and

$$\Omega(h)(\widehat{\omega}(\widehat{g_1}), \cdots, \widehat{\omega}(\widehat{g_m})) = \Omega(h)\circ (g_1, \cdots, g_m) = h \circ (g_1, \cdots, g_m)$$

so $\widehat{\omega}$ is a $\mathcal{C}^{\infty}-$homomorphism.\\

We claim that:

$$\begin{array}{cccc}
\widehat{\delta}: & (\mathcal{C}^{\infty}(\mathbb{R}^m),\Omega) & \mapsto & (\mathcal{C}^{\infty}(\mathbb{R}^{X}),\Phi_X)\\
     & (\mathbb{R}^m \stackrel{g}{\rightarrow} \mathbb{R}) & \mapsto & \begin{array}{cccc}
     \widehat{g}: & \mathbb{R}^{X} & \rightarrow & \mathbb{R}\\
      & f & \mapsto & g(f(x_1), \cdots, f(x_n))
     \end{array}
\end{array}$$

is a $\mathcal{C}^{\infty}-$homomorphism. In fact,  for every $n \in \mathbb{N}$, $h \in \mathcal{C}^{\infty}(\mathbb{R}^n, \mathbb{R})$ the following diagram commutes:

$$\xymatrixcolsep{5pc}\xymatrix{
\mathcal{C}^{\infty}(\mathbb{R}^m)^n \ar[r]^{\widehat{\delta}^n} \ar[d]_{\Omega(h)} & \mathcal{C}^{\infty}(\mathbb{R}^{X})^n \ar[d]^{\Phi_X(h)}\\
\mathcal{C}^{\infty}(\mathbb{R}^m) \ar[r]_{\widehat{\delta}} & \mathcal{C}^{\infty}(\mathbb{R}^{X})
}$$

since given $(g_1, \cdots, g_n) \in \mathcal{C}^{\infty}(\mathbb{R}^m)^n$ we have:

$$\Phi_X(h)(\widehat{\delta}(g_1), \cdots, \widehat{\delta}(g_m)) = \Phi_X(h)\circ (\widehat{g_1}, \cdots, \widehat{g_n}) = \widehat{h \circ (g_1, \cdots, g_n)}$$

and

$$\widehat{\delta}(\Omega(h)(g_1, \cdots, g_n)) = \widehat{\delta}(h \circ (g_1, \cdots, g_n)) = \widehat{h \circ (g_1, \cdots, g_n)},$$

so $\widehat{\delta}$ is a $\mathcal{C}^{\infty}-$homomorphism.\\

Given $g \in \mathcal{C}^{\infty}(\mathbb{R}^m)$, we have:

$$\widehat{\omega}(\widehat{\delta}(g)) = \widehat{\omega}(\widehat{g}) = g$$

so

$$\widehat{\omega} \circ \widehat{\delta} = {\rm id}_{\mathcal{C}^{\infty}(\mathbb{R}^m)}.$$

Also, given $\widehat{g} \in \mathcal{C}^{\infty}(\mathbb{R}^{X})$, we have:

$$\widehat{\delta}(\widehat{\omega}(\widehat{g})) = \widehat{\delta}(g) = \widehat{g},$$

so

$$\widehat{\delta} \circ \widehat{\omega} = {\rm id}_{\mathcal{C}^{\infty}(\mathbb{R}^{X})}$$

\end{proof}

As an immediate consequence of \textbf{Propositions  \ref{Kakau}} and \textbf{\ref{Kakeu}}, we have:

\begin{corollary}For any finite set $X$, $(\jmath_X, (\mathcal{C}^{\infty}(\mathbb{R}^{X}), \Phi_X))$ is the free $\mathcal{C}^{\infty}-$ring defined by $X$.
\end{corollary}
\begin{proof}
Let $\omega: \{1, \cdots, n\} \rightarrow X$ be a bijection, $\widehat{\omega}: (\mathcal{C}^{\infty}(\mathbb{R}^{X}), \Phi_X) \rightarrow (\mathcal{C}^{\infty}(\mathbb{R}^n), \Omega)$ be the $\mathcal{C}^{\infty}-$isomorphism induced by $\omega$, as given in \textbf{Proposition \ref{Kakau}}, $(A,\Phi)$ be any $\mathcal{C}^{\infty}-$ring and let $f: X \rightarrow U(A,\Phi)$ be any function. \\

Note that the following diagram commutes:

$$\xymatrixcolsep{5pc}\xymatrix{
\{1, \cdots, n \} \ar[r]^{\jmath_n} \ar[d]_{\omega} & U(\mathcal{C}^{\infty}(\mathbb{R}^n), \Omega) \ar[d]^{U(\widehat{\omega}^{-1})}\\
X \ar[r]^{\jmath_X} & U(\mathcal{C}^{\infty}(\mathbb{R}^{X}),\Phi_X)\\
},$$

since given any $i \in \{ 1, \cdots, n\}$, $\widehat{\omega}^{-1}(\jmath_n)(i) = \widehat{\omega}^{-1}(\pi_i) = \widehat{\pi_i} = {\rm ev}_{x_i} = \jmath_X(x_i) = \jmath_X \circ \omega(i)$.

We have the following diagram:

$$\xymatrixcolsep{5pc}\xymatrix{
\{1, \cdots, n \} \ar[r]^{\jmath_n} \ar[d]_{\omega} & U(\mathcal{C}^{\infty}(\mathbb{R}^n), \Omega) \ar[d]^{U(\widehat{\omega}^{-1})}\\
X \ar[r]^{\jmath_X} \ar[dr]_{f} & U(\mathcal{C}^{\infty}(\mathbb{R}^{X}),\Phi_X)\\
 & U(A,\Phi)
}$$

Since $(\jmath_n,(\mathcal{C}^{\infty}(\mathbb{R}^n), \Omega))$ is free, given the function $f \circ \omega: \{ 1, \cdots, n\} \rightarrow U(A,\Phi)$, there is a unique $\mathcal{C}^{\infty}-$homomorphism $\widetilde{f\circ \omega}: (\mathcal{C}^{\infty}(\mathbb{R}^n),\Omega) \rightarrow (A, \Phi)$ such that the following diagram commutes:\\

$$\xymatrixcolsep{5pc}\xymatrix{
\{ 1, \cdots, n\} \ar[r]^{\jmath_n} \ar[dr]_{f \circ \omega} & U(\mathcal{C}^{\infty}(\mathbb{R}^n),\Omega) \ar[d]^{U(\widetilde{f \circ \omega})}\\
 & U(A,\Phi)
}$$

We have, thus, the following commutative diagram:

$$\xymatrixcolsep{5pc}\xymatrix{
\{1, \cdots, n \} \ar[r]^{\jmath_n} \ar[d]_{\omega} & U(\mathcal{C}^{\infty}(\mathbb{R}^n), \Omega) \ar@/^5pc/[dd]^{U(\widetilde{f \circ \omega})} \ar[d]^{U(\widehat{\omega}^{-1})}\\
X \ar[r]^{\jmath_X} \ar[dr]_{f} & U(\mathcal{C}^{\infty}(\mathbb{R}^{X}),\Phi_X)\\
 & U(A,\Phi)
}$$

It now suffices to take the $\mathcal{C}^{\infty}-$homomorphism $\widetilde{f}:= (\widetilde{f \circ \omega})\circ \widehat{\omega}: (\mathcal{C}^{\infty}(\mathbb{R}^{X}),\Phi_X) \rightarrow (A,\Phi)$, and the following diagram commutes:

$$\xymatrixcolsep{5pc}\xymatrix{
X \ar[r]^{\jmath_X} \ar[dr]_{f} & U(\mathcal{C}^{\infty}(\mathbb{R}^X),\Phi_X) \ar[d]^{U(\widetilde{f})}\\
 & U(A,\Phi)
}$$

By construction such a $\widetilde{f}$ is unique, so $(\jmath_X, (\mathcal{C}^{\infty}(\mathbb{R}^{X}), \Phi_X))$ is the free $\mathcal{C}^{\infty}-$ring determined by $X$.
\end{proof}

Now we turn to the definition of the free $\mathcal{C}^{\infty}-$ring determined by an arbitrary set.\\

Let $E$ be any set, and consider:

$$(\mathcal{C}^{\infty}(\mathbb{R}^{E}),\Phi_E) = \varinjlim_{E' \subseteq_f E} (\mathcal{C}^{\infty}(\mathbb{R}^{E'}),\Phi_{E'}),$$

where ``$\subseteq_f$'' stands for ``is a finite subset of'', together with the unique arrow $\jmath_E: E = \bigcup_{E' \subseteq_f E} E' \rightarrow \mathcal{C}^{\infty}(\mathbb{R}^{E})$ such that for every $E' \subseteq_f E$ the following diagram commutes:

$$\xymatrixcolsep{5pc}\xymatrix{
E \ar[r]^{\jmath_E} & \mathcal{C}^{\infty}(\mathbb{R}^{E})\\
E' \ar@{^{(}->}[u] \ar[r]_{\jmath_{E'}} & \mathcal{C}^{\infty}(\mathbb{R}^{E'}) \ar[u]^{\ell_{E'}}}$$

\begin{proposition}Let $E$ be any set. The pair $(\jmath_E, (\mathcal{C}^{\infty}(\mathbb{R}^{E}),\Phi_E))$ is the free $\mathcal{C}^{\infty}-$ring determined by $E$.
\end{proposition}
\begin{proof}
Let $(A,\Phi)$ be any $\mathcal{C}^{\infty}-$ring and let $f: E \to A$ be any function. We decompose the set $E$ as

$$E = \bigcup_{E' \subseteq_f E}E' = \varinjlim_{E' \subseteq_f E}E$$

and take the following cone:

$$\xymatrixcolsep{5pc}\xymatrix{
 & A & \\
E' \ar[ur]^{f \circ \iota_{E'}} \ar[rr]_{ \iota^{E'}_{E''}}& & E'' \ar[ul]_{f \circ\iota_{E''}}
}$$

Since for every finite $E' \subseteq E$, $(\jmath_{E'}, \mathcal{C}^{\infty}(\mathbb{R}^{E'}))$ is the free $\mathcal{C}^{\infty}-$ring determined by $E'$, given the function $f \circ \iota_{E'}: E' \to A$, there is a unique $\mathcal{C}^{\infty}-$homomorphism $\widehat{f_{E'}}: (\mathcal{C}^{\infty}(\mathbb{R}^{E'}), \Phi_{E'}) \rightarrow (A,\Phi)$ such that the following diagram commutes:

$$\xymatrixcolsep{5pc}\xymatrix{
E' \ar[r]^{\jmath_{E'}} \ar[dr]_{f \circ \iota_{E'}} & \mathcal{C}^{\infty}(\mathbb{R}^{E'}) \ar[d]^{U(\widehat{f}_{E'})}\\
 & U(A,\Phi)
}$$

Note that the following diagram commutes:

$$\xymatrixcolsep{5pc}\xymatrix{
(\mathcal{C}^{\infty}(\mathbb{R}^{E'}), \Phi_{E'}) \ar[rr]^{\widehat{\mu}_{E'E''}\upharpoonright_{\mathcal{C}^{\infty}(\mathbb{R}^{E'})}} \ar[dr]_{\widehat{f}_{E'}} & & (\mathcal{C}^{\infty}(\mathbb{R}^{E''}), \Phi_{E''}) \ar[dl]^{\widehat{f}_{E''}}\\
 & (A,\Phi) &
}$$

In fact, the following rectangle commutes:

$$\xymatrixcolsep{8pc}\xymatrix{
E' \ar[r]^{\iota_{E'}^{E''}} \ar[d]_{\jmath_{E'}} & E'' \ar[d]^{\jmath_{E''}}\\
\mathcal{C}^{\infty}(\mathbb{R}^{E'}) \ar[r]_{U(\widehat{\mu_{E'E''}}\upharpoonright_{\mathcal{C}^{\infty}(\mathbb{R}^{E'})})} & \mathcal{C}^{\infty}(\mathbb{R}^{E''})
}$$

so:

\begin{equation}\label{words}
\jmath_{E''}\circ \iota^{E''}_{E'} = U(\widehat{\mu}_{E'E''}\upharpoonright_{\mathcal{C}^{\infty}(\mathbb{R}^{E'})})\circ \jmath_{E'}
\end{equation}

and $\widehat{f}_{E'}: (\mathcal{C}^{\infty}(\mathbb{R}^{E'}),\Phi_{E'}) \to (A,\Phi)$ is the unique $\mathcal{C}^{\infty}-$homomorphism such that:

$$\xymatrixcolsep{5pc}\xymatrix{
E' \ar[r]^{\jmath_{E'}} \ar[dr]_{f \circ \iota_{E'}} & \mathcal{C}^{\infty}(\mathbb{R}^{E'}) \ar[d]^{U(\widehat{f}_{E'})}\\
 & U(A,\Phi)
}$$

commutes, i.e.,

\begin{equation}\label{words1}
 U(\widehat{f}_{E'})\circ \jmath_{E'} = f \circ \iota_{E'}
\end{equation}

We are going to show that $\widehat{f}_{E''}\circ \widehat{\mu}_{E'E''}\upharpoonright_{\mathcal{C}^{\infty}(\mathbb{R}^{E'})}:(\mathcal{C}^{\infty}(\mathbb{R}^{E'}),\Phi_{E'}) \to (A,\Phi)$ is such that:

$$\xymatrixcolsep{5pc}\xymatrix{
E' \ar[r]^{\jmath_{E'}} \ar[dr]_{f \circ \iota_{E'}} & \mathcal{C}^{\infty}(\mathbb{R}^{E'}) \ar[d]^{U(\widehat{f}_{E''}\circ \widehat{\mu}_{E'E''}\upharpoonright_{\mathcal{C}^{\infty}(\mathbb{R}^{E'})})}\\
 & U(A,\Phi)
}$$

commutes. By uniqueness it will follow that:

$$\widehat{f}_{E''}\circ \widehat{\mu}_{E'E''}\upharpoonright_{\mathcal{C}^{\infty}(\mathbb{R}^{E'})} = \widehat{f}_{E'}.$$

Composing \eqref{words} with $U(\widehat{f}_{E''})$ by the left yields:

$$U(\widehat{f}_{E''}) \circ \jmath_{E''}\circ \iota^{E''}_{E'} = U(\widehat{f}_{E''}) \circ U(\widehat{\mu}_{E'E''}\upharpoonright_{\mathcal{C}^{\infty}(\mathbb{R}^{E'})})\circ \jmath_{E'}$$

i.e.,

$$U(\widehat{f}_{E''}) \circ \jmath_{E''}\circ \iota^{E''}_{E'} = U(\widehat{f}_{E''} \circ \widehat{\mu}_{E'E''}\upharpoonright_{\mathcal{C}^{\infty}(\mathbb{R}^{E'})})\circ \jmath_{E'}$$

Since $U(\widehat{f}_{E''})\circ \jmath_{E''} = f \circ \iota_{E''}$, we have:

$$(f \circ \iota_{E''})\circ \iota^{E''}_{E'} = U(\widehat{f}_{E''} \circ \widehat{\mu}_{E'E''}\upharpoonright_{\mathcal{C}^{\infty}(\mathbb{R}^{E'})})\circ \jmath_{E'}$$

We know that $\iota_{E''}\circ \iota^{E''}_{E'} = \iota_{E'}$, so we get:

$$f \circ \iota_{E'}= U(\widehat{f}_{E''} \circ \widehat{\mu}_{E'E''}\upharpoonright_{\mathcal{C}^{\infty}(\mathbb{R}^{E'})})\circ \jmath_{E'},$$

and the diagram commutes.\\

By the universal property of the colimit $(\mathcal{C}^{\infty}(\mathbb{R}^{E}),\Phi_{E})$, there is a unique $\mathcal{C}^{\infty}-$homomorphism $\widehat{f}: (\mathcal{C}^{\infty}(\mathbb{R}^{E}),\Phi_E) \to (A,\Phi)$ such that:

$$\xymatrixcolsep{5pc}\xymatrix{
(\mathcal{C}^{\infty}(\mathbb{R}^{E'}),\Phi_{E'}) \ar[rr]^{\widehat{\mu}_{E'E''}\upharpoonright_{\mathcal{C}^{\infty}(\mathbb{R}^{E'})}} \ar[dr]_{\ell_{E''}} \ar@/_2pc/[ddr]_{\widehat{f}_{E'}} & & (\mathcal{C}^{\infty}(\mathbb{R}^{E''}),\Phi_{E''}) \ar[dl]^{\ell_{E''}} \ar@/^2pc/[ddl]^{\widehat{f}_{E''}} \\
  & (\mathcal{C}^{\infty}(\mathbb{R}^{E}),\Phi_{E}) \ar[d]^{\widehat{f}} & \\
  & (A,\Phi)&
}$$

Now we need only to show that $U(\widehat{f})\circ \jmath_E = f$.\\

Given any $x \in E$, there is some finite $E' \subseteq_f E$ such that $x \in E'$.\\

We have the following commutative diagram:

$$\xymatrixcolsep{5pc}\xymatrix{
E' \ar[r]^{\jmath_{E'}} & U(\mathcal{C}^{\infty}(\mathbb{R}^{E'}),\Phi_{E'}) \ar[r]^{\ell_{E'}} \ar[dr]_{U(\widehat{f}_{E'})} & U(\mathcal{C}^{\infty}(\mathbb{R}^{E}),\Phi_{E}) \ar[d]^{U(\widehat{f})}\\
 & & U(A,\Phi)
}$$

so:

$$U(\widehat{f})\circ \ell_{E'}\circ \jmath_{E'} = U(\widehat{f}_{E'})\circ \jmath_{E'}$$

Since $\ell_{E'}\circ \jmath_{E'} = \jmath_E \circ \iota_{E'}$ and $U(\widehat{f}_{E'})\circ \jmath_{E'} = f \circ \iota_{E'}$, we get:

$$ U(\widehat{f})\circ \jmath_E \circ \iota_{E'}  = f \circ \iota_{E'}$$

thus:

$$U(\widehat{f})\circ \jmath_{E}(x) = U(\widehat{f})\circ \jmath_{E} \circ \iota_{E'}(x) = f(\iota_{E'}(x)) = f(x)$$

and the diagram commutes.\\

The uniqueness of $\widehat{f}$ comes from its construction.
\end{proof}

In the following proposition, we present a description of a left adjoint to the forgetful functor $U: \mathcal{C}^{\infty}{\rm \bf Rng} \rightarrow {\rm \bf Set}$.\\

\begin{proposition}\label{Julia}The functions:
$$\begin{array}{cccc}
L_0: & {\rm Obj}\,({\rm \bf Set}) & \rightarrow & {\rm Obj}\,(\mathcal{C}^{\infty}{\rm \bf Rng})\\
     & X & \mapsto & (\mathcal{C}^{\infty}\,(\mathbb{R}^X), \Phi_X)
\end{array}$$

and

$$\begin{array}{cccc}
L_1: & {\rm Mor}\,({\rm \bf Set}) & \rightarrow & {\rm Mor}\,(\mathcal{C}^{\infty}{\rm \bf Rng})\\
     & (X \stackrel{f}{\rightarrow} Y) & \mapsto &   (\mathcal{C}^{\infty}(\mathbb{R}^X), \Phi_X) \stackrel{\widetilde{f}}{\rightarrow} (\mathcal{C}^{\infty}(\mathbb{R}^Y), \Phi_Y)
     \end{array}$$

where $\widetilde{f}: L_0(X) \to L_0(Y)$ is the unique $\mathcal{C}^{\infty}-$homomorphism given by the universal property of the free $\mathcal{C}^{\infty}-$ring $\jmath_X: X \rightarrow \mathcal{C}^{\infty}(\mathbb{R}^{X})$: given the function $\jmath_Y \circ f: X \rightarrow U(\mathcal{C}^{\infty}(\mathbb{R}^{Y}))$, there is a unique $\mathcal{C}^{\infty}-$homomorphism such that the following diagram commutes:\\

$$\xymatrixcolsep{5pc}\xymatrix{
X \ar[r]^{\jmath_X} \ar[dr]_{\jmath_Y \circ f} & U(\mathcal{C}^{\infty}(\mathbb{R}^{X})) \ar@{-->}[d]^{U(\widetilde{f})}\\
 & U(\mathcal{C}^{\infty}(\mathbb{R}^{Y}))
}$$

define a functor $L: {\rm \bf Set} \rightarrow \mathcal{C}^{\infty}{\rm \bf Rng}$ which is left adjoint to the forgetful functor $U: \mathcal{C}^{\infty}{\rm \bf Rng} \to {\rm \bf Set}$.
\end{proposition}
\begin{proof}
We claim that $L_1({\rm id}_X)={\rm id}_{(\mathcal{C}^{\infty}(\mathbb{R}^{X}),\Phi_X)}$.\\

On the one hand, $\widetilde{{\rm id}_X}$ is the only $\mathcal{C}^{\infty}-$homomorphism such that:

$$\xymatrixcolsep{5pc}\xymatrix{
X \ar[r]^{\jmath_X} \ar[dr]_{\jmath_X \circ {\rm id}_X} & U(\mathcal{C}^{\infty}(\mathbb{R}^{X})) \ar@{-->}[d]^{U(\widetilde{{\rm id}_X})}\\
 & U(\mathcal{C}^{\infty}(\mathbb{R}^{X}))
}$$

commutes.\\

On the other hand, since ${\rm id}_{(\mathcal{C}^{\infty}(\mathbb{R}^{X}),\Phi_X)}$ is a $\mathcal{C}^{\infty}-$homomorphism which makes the following diagram to commute:

$$\xymatrixcolsep{5pc}\xymatrix{
X \ar[r]^{\jmath_X} \ar[dr]_{\jmath_X } & U(\mathcal{C}^{\infty}(\mathbb{R}^{X})) \ar@{-->}[d]^{U({\rm id}_{(\mathcal{C}^{\infty}(\mathbb{R}^{X}),\Phi_X)})}\\
 & U(\mathcal{C}^{\infty}(\mathbb{R}^{X}))
}$$

it follows that:

$$L_1({\rm id}_X) = \widetilde{{\rm id}_X}= {\rm id}_{(\mathcal{C}^{\infty}(\mathbb{R}^{X}),\Phi_X)}$$

Now, let $X,Y,Z$ be sets and $f: X \to Y$ and $g: Y \to Z$ be functions. We claim that:

$$L_1(g \circ f) = L_1(g) \circ L_1(f)$$

By definition, $L_1(f)=\widetilde{f}: (\mathcal{C}^{\infty}(\mathbb{R}^{X}),\Phi_X) \rightarrow (\mathcal{C}^{\infty}(\mathbb{R}^{Y}),\Phi_Y)$ is the only $\mathcal{C}^{\infty}-$homomorphism such that:

$$\xymatrixcolsep{5pc}\xymatrix{
X \ar[r]^{\jmath_X} \ar[dr]_{\jmath_Y \circ f} & U(\mathcal{C}^{\infty}(\mathbb{R}^{X})) \ar@{-->}[d]^{U(\widetilde{f})}\\
 & U(\mathcal{C}^{\infty}(\mathbb{R}^{Y}))
}$$

commutes, that is, such that $U(\widetilde{f})\circ \jmath_X = \jmath_Y \circ f$. Also, by definition, $L_1(g)=\widetilde{g}: (\mathcal{C}^{\infty}(\mathbb{R}^{Y}),\Phi_Y) \rightarrow (\mathcal{C}^{\infty}(\mathbb{R}^{Z}),\Phi_Z)$ is the only $\mathcal{C}^{\infty}-$homomorphism such that:

$$\xymatrixcolsep{5pc}\xymatrix{
Y \ar[r]^{\jmath_Y} \ar[dr]_{\jmath_Z \circ g} & U(\mathcal{C}^{\infty}(\mathbb{R}^{Y})) \ar@{-->}[d]^{U(\widetilde{g})}\\
 & U(\mathcal{C}^{\infty}(\mathbb{R}^{Z}))
}$$

commutes, that is, such that $U(\widetilde{g})\circ \jmath_Y = \jmath_Z \circ g$.\\

Finally, $\widetilde{g \circ f}: (\mathcal{C}^{\infty}(\mathbb{R}^{X}),\Phi_X) \rightarrow (\mathcal{C}^{\infty}(\mathbb{R}^{Z}),\Phi_Z)$ is the unique $\mathcal{C}^{\infty}-$homo\-morphism such that:

$$\xymatrixcolsep{5pc}\xymatrix{
X \ar[r]^{\jmath_X} \ar[dr]_{\jmath_Z \circ (g\circ f)} & U(\mathcal{C}^{\infty}(\mathbb{R}^{X})) \ar@{-->}[d]^{U(\widetilde{g\circ f})}\\
 & U(\mathcal{C}^{\infty}(\mathbb{R}^{Z}))
}$$

commutes, that is, such that $U(\widetilde{g \circ f})\circ \jmath_X = \jmath_Z \circ (g \circ f)$.\\

We are going to show that $\widetilde{g}\circ \widetilde{f}: (\mathcal{C}^{\infty}(\mathbb{R}^{X}),\Phi_X) \rightarrow (\mathcal{C}^{\infty}(\mathbb{R}^{Z}),\Phi_Z)$ also has this property.\\

In fact,

\begin{multline*}
    U(\widetilde{f})\circ \jmath_X = \jmath_Y \circ f \Rightarrow U(\widetilde{g})\circ (U(\widetilde{f})\circ \jmath_X) = (U(\widetilde{g})\circ \jmath_Y) \circ f \Rightarrow \\
    \Rightarrow (U(\widetilde{g})\circ U(\widetilde{f}))\circ \jmath_X = (\jmath_Z \circ g) \circ f
\end{multline*}

and thus

$$U(\widetilde{g}\circ \widetilde{f})\circ \jmath_X = \jmath_Z \circ (g \circ f).$$

By uniqueness it follows that:

$$L_1(g \circ f) = L_1(g) \circ L_1(f),$$

so $L = (L_0,L_1)$ defines a functor.\\

We claim that:

$$\begin{array}{cccc}
\phi_{X,(A,\Phi)}: & {\rm Hom}_{\rm \bf Set}\,(X, A) & \rightarrow & {\rm Hom}_{\mathcal{C}^{\infty}{\rm \bf Rng}}\,((L(X),\Phi_X),(A,\Phi))\\
                   & ( X \stackrel{\varphi}{\rightarrow} A) & \mapsto &  (L(X),\Phi_X) \stackrel{\widetilde{\varphi}}{\rightarrow} (A,\Phi)

\end{array}$$

where $\widetilde{\varphi}$ is the unique $\mathcal{C}^{\infty}-$homomorphism such that $U(\widetilde{\varphi})\circ \jmath_X = \varphi$

is a natural bijection, so $L \dashv U$.\\

$\phi_{X, (A,\Phi)}$ is surjective. Given any $\mathcal{C}^{\infty}-$homomorphism $\widetilde{\varphi}: (L(X), \Phi_X) \rightarrow (A,\Phi)$, taking $\varphi:= U(\widetilde{\varphi}) \circ \jmath_X : X \rightarrow A$ we have $L(\varphi) = \widetilde{\varphi}$.\\

Given $\varphi, \psi: X \to A$ such that $\phi_{X, (A,\Phi)}(\varphi) = \widetilde{\varphi} = \widetilde{\psi} = \phi_{X, (A,\Phi)}(\psi)$, we have:

$$U(\widetilde{\varphi}) = U(\widetilde{\psi})$$

and
$$\varphi = U(\widetilde{\varphi}) \circ \jmath_X = U(\widetilde{\psi}) \circ \jmath_X = \psi$$

so $\phi_{X, (A,\Phi)}$ is a bijection.

\end{proof}

Let $(A,\Phi)$ be a $\mathcal{C}^{\infty}-$ring and let $X \subseteq A$. Given $\iota^{A}_{X}: X \hookrightarrow A$, there is a unique $\mathcal{C}^{\infty}-$homomorphism $\widetilde{\iota^{A}_{X}}: (L(X),\Phi_X) \to (A,\Phi)$ such that  the following diagram commutes:

\begin{equation}\label{carrasco}\xymatrixcolsep{5pc}\xymatrix{
X \ar[r]^{\eta_X} \ar[dr]_{\iota^{A}_{X}} & U(L(X)) \ar[d]^{U(\widetilde{\iota^{A}_{X}})}\\
 & U(A,\Phi)
}
\end{equation}

We claim that if $\langle X\rangle = A$, then $\widetilde{\iota^{A}_{X}}$ is surjective.\\

Indeed,

$$X = \iota^{A}_{X}[X] = {\rm im}(\iota^{A}_{X})$$

and since $U(\widetilde{\iota^{A}_{X}})\circ \eta_X = \iota^{A}_{X}$, we have:

$${\rm im}(\iota^{A}_{X}) = {\rm im}(U(\widetilde{\iota^{A}_{X}})\circ \eta_X).$$

Since the diagram \eqref{carrasco} commutes, on the other hand,

$${\rm im}(U(\widetilde{\iota^{A}_{X}})\circ \eta_X) \subseteq {\rm im}\,(U(\widetilde{\iota^{A}_{X}}))$$

thus

$$X \subseteq {\rm im}\,(U(\widetilde{\iota^{A}_{X}}))$$

and

$$\langle X \rangle \subseteq \langle {\rm im}\,(U(\widetilde{\iota^{A}_{X}})) \rangle = \widetilde{\iota^{A}_{X}}[A]$$

Since $X$ generates $A$ and $\langle {\rm im}\,(U(\widetilde{\iota^{A}_{X}}))\rangle =   {\rm im}\,(\widetilde{\iota^{A}_{X}})$ is a $\mathcal{C}^{\infty}-$subring of $(A,\Phi)$, it follows that::

$$\langle X \rangle = A \subseteq {\rm im}(\widetilde{\iota^{A}_{X}}) \subseteq A$$

so ${\rm im}\,(\widetilde{\iota^{A}_{X}}) = A$ and $\iota^{A}_{X}$ is surjective.\\

In particular, taking $X=A$ yields $\iota^{A}_{A} = {\rm id}_{A}$, and since $\varepsilon_A = \phi_{A,(A,\Phi)}({\rm id}_A) = \widetilde{{\rm id}_{(A,\Phi)}}$, we have:

$$\xymatrixcolsep{5pc}\xymatrix{
A \ar[r]^{\eta_A} \ar[dr]_{{\rm id}_A} & U(L(A)) \ar[d]^{U(\varepsilon_A)}\\
 & U(A,\Phi)
}$$

so ${\rm im}\,(\varepsilon_A) = (A,\Phi)$, and $\varepsilon_A$ is surjective. Now, given any $\mathcal{C}^{\infty}-$ring $(A,\Phi)$ we have the surjective morphism:

$$\varepsilon_A: L(U(A,\Phi)) \twoheadrightarrow (A,\Phi).$$

We have seen that since $\varepsilon_A$  is a $\mathcal{C}^{\infty}-$homomorphism,  $\ker(\varepsilon_A)$ is a $\mathcal{C}^{\infty}-$congru\-ence. By the \textbf{Fundamental Theorem of the $\mathcal{C}^{\infty}-$Isomorphism} we have:\\

$$(A,\Phi) \cong \dfrac{L(A)}{\ker(\varepsilon_A)}$$

\section{Other Constructions}\label{oc}

In this section we describe further results and constructions involving $\mathcal{C}^{\infty}-$rings.\\

\subsection{Ring-Theoretic Ideals and $\mathcal{C}^{\infty}-$Congruences}\label{coquette}

Our next goal is to classify the congruences of any $\mathcal{C}^{\infty}-$ring. We shall see that they are classified by their ring-theoretic ideals.\\

We begin with the finitely generated case.\\

An interesting result, which is a consequence of \textbf{Hadamard's Lemma}, \textbf{Theorem \ref{hadamard}}, is the description of the ideals of finitely generated $\mathcal{C}^{\infty}-$rings in terms of $\mathcal{C}^{\infty}-$congruences.\\

For the reader's benefit, we state and prove the following result:

\begin{theorem}{\textbf{\emph{(Hadamard's Lemma)}}}\label{hadamard}For every smooth function $f \in \mathcal{C}^{\infty}(\mathbb{R}^n)$  there are smooth functions $g_1, \cdots, g_n \in \mathcal{C}^{\infty}(\mathbb{R}^{2n})$ such that 
$$\forall (x_1, \cdots, x_n), (y_1, \cdots, y_n) \in \mathbb{R}^n$$
$$f(x_1, \cdots, x_n) - f(y_1, \cdots, y_n)=  \sum_{i=1}^{n} (x_i-y_i)\cdot g_i(x_1, \cdots, x_n, y_1, \cdots, y_n)$$
\end{theorem}
\begin{proof}
Holding $(v_1, \cdots, v_n) = \vec{v} \in \mathbb{R}^n$ fixed, we define:
$$h(t)=f(\vec{y} + t\cdot (\vec{x} - \vec{y})).$$
We have:
$$f(\vec{x})-f(\vec{y}) = \displaystyle\int_{0}^{1}h'(t)dt = \displaystyle\int_{0}^{1} \sum_{i=1}^{n} \dfrac{\partial f}{\partial v_i}(\vec{y} + t\cdot (\vec{x} - \vec{y})))\cdot (x_i -y_i) dt,$$
where the second equality uses the chain rule.\\

The result follows by putting:
$$g_i(\vec{x}, \vec{y}) = \displaystyle\int_{0}^{1}\dfrac{\partial f}{\partial v_i}(\vec{y} + t\cdot (\vec{x} - \vec{y}))dt$$
\end{proof}

\begin{proposition}\label{Monica}Given a finitely generated $\mathcal{C}^{\infty}-$ring $(A,\Phi)$, considering the forgetful functor given in \textbf{Remark \ref{Rock}}, $\widetilde{U}: \mathcal{C}^{\infty}{\rm \bf Rng} \rightarrow {\rm \bf CRing}$, we have that if $I$ is a subset of $A$ that is an ideal (in the ordinary ring-theoretic sense) in   $\widetilde{U}(A,\Phi)$, then $\widehat{I} = \{ (a,b) \in A \times A | a - b \in I\}$ is a $\mathcal{C}^{\infty}-$congruence in $A$.
\end{proposition}
\begin{proof}
(Cf. p. 18 of \cite{MSIA})\\

Let $n \in \mathbb{N}$ and $f \in \mathcal{C}^{\infty}(\mathbb{R}^n,\mathbb{R})$.\\

We have to show that if for every $i \in \{ 1, \cdots, n\}$ we have $(a_i,b_i) \in \widehat{I}$, that is, $a_i - b_i \in I$,  then:

$$ \Phi^{(2)}(f)((a_1,b_1), \cdots, (a_n,b_n)) = (\Phi(f)(a_1, \cdots, a_n),\Phi(f)(b_1, \cdots, b_n)) \in \widehat{I},$$

that is, $\Phi(f)(a_1, \cdots, a_n) - \Phi(f)(b_1, \cdots, b_n) \in I$.\\

Suppose that $(a_1,b_1), \cdots, (a_n,b_n) \in A \times A$ are such that $a_i - b_i \in I$ for every $i \in \{1, \cdots,n \}$. By \textbf{Hadamard's Lemma}, there are smooth functions $g_1, \cdots, g_n : \mathbb{R}^n \times \mathbb{R} \to \mathbb{R}$ such that for all $x = (x_1, \cdots, x_n)$ and $y=(y_1, \cdots, y_n) \in \mathbb{R}^n$,

$$f(x)-f(y) = \sum_{i=1}^{n}(x_i - y_i)\cdot g_i(x,y).$$

$\Phi$ preserves this equation, \textit{i.e.},

\begin{multline*}\Phi(f)(a_1, \cdots, a_n) - \Phi(f)(b_1, \cdots, b_n) =\\
=\sum_{i=1}^{n}\underbrace{(a_i - b_i)}_{\in I}\cdot \Phi(g_i)(a_1, \cdots, a_n, b_1, \cdots, b_n) \in I,\end{multline*}

and the result follows.
\end{proof}

The following result tells us that the ideals of the $\mathcal{C}^{\infty}-$rings classify its congruences.\\

\begin{proposition}\label{Cemil}Given any finitely generated $\mathcal{C}^{\infty}-$ring $(A, \Phi)$, let ${\rm Cong}\,(A,\Phi)$ denote the set of all the $\mathcal{C}^{\infty}-$congruences in $A$ and let $\mathfrak{I}(A,\Phi)$ denote the set of all ideals of $A$. The following map is a bijection:

$$\begin{array}{cccc}
\psi_A: & {\rm Cong}\,(A,\Phi) & \rightarrow & \mathfrak{I}(A, \Phi)\\
    & R & \mapsto & \{ x \in A | (x,0) \in R \}
\end{array}$$
whose inverse is given by:
$$\begin{array}{cccc}
\varphi_A: & \mathfrak{I}(A, \Phi) & \rightarrow & {\rm Cong}\,(A, \Phi)\\
    & I & \mapsto & \{ (x,y) \in A\times A | x - y \in I \}
\end{array}$$
\end{proposition}
\begin{proof}
\textbf{Proposition \ref{Monica}} assures that $\varphi_A$ maps ideals of $(A,\Phi)$ to $\mathcal{C}^{\infty}-$congruences.\\

Also, given a $\mathcal{C}^{\infty}-$congruence $R$ in $A$, we have that $\psi_A(R)$ is an ideal.\\

\textbf{Claim:} $\{ x \in A | (x,0) \in R\}$ is closed under $\Phi(+)$ and $\Phi(-)$, since for any $x,y \in A$ such that $(x,0),(y,0) \in R$, we have 
$$\Phi^{(2)}(+)((x,0),(y,0))= (\Phi(+)(x,y),\Phi(+)(0,0)) = (x+y,0)\in R,$$ for $R$ is a $\mathcal{C}^{\infty}-$congruence and $(R, {\Phi^{(2)}}')$ is a $\mathcal{C}^{\infty}-$subring of $(A\times A, \Phi^{(2)})$.\\

Also, given $x \in A$ such that $(x,0) \in R$, since $(R, {\Phi^{(2)}}')$ is a $\mathcal{C}^{\infty}-$subring of $(A \times A, \Phi^{(2)})$, we have ${\Phi^{(2)}}'(-)(x,0) = (-x,0) \in R$.\\

Finally, given $x \in A$ such that $(x,0)\in R$ and any $y \in A$, we have both $(x,0), (y,y)\in R$, and since $R$ is a $\mathcal{C}^{\infty}-$congruence, and thus $(R, {\Phi^{(2)}}')$ is a $\mathcal{C}^{\infty}-$subring of $(A\times A,\Phi^{(2)})$, we have:

$${\Phi^{(2)}}'(\cdot)((x,0),(y,y)) = (x \cdot y, y \cdot 0) = (x \cdot y, 0) \in R$$

so $x\cdot y \in \{ x \in A | (x,0) \in R \}$. Hence $\{ x \in A | (x,0) \in R\}$ is an ideal in $(A,\Phi)$.\\

Now we are going to show that $\psi_A$ and $\varphi_A$ are inverse functions.\\

Given $R \in {\rm Cong}\,(A,\Phi)$, we have:

$$\varphi_A(\psi_A(R)) = \varphi_A(\{ x \in A | (x,0) \in R\}) = \{ (x,y) \in A \times A | (x-y,0) \in R \} = R$$

so $\varphi_A \circ \psi_A = {\rm id}_{{\rm Cong}\,(A,\Phi)}$.\\

Given $I \in \mathfrak{I}\,(A)$, we have:

$$\psi_A(\varphi_A(I)) = \psi_A(\{ (x,y) \in A \times A | x - y \in I\}) = \{ z \in A | z - 0 \in I \} = I,$$

so $\psi_A \circ \varphi_A = {\rm id}_{\mathfrak{I}\,(A,\Phi)}$.
\end{proof}

For any $\mathcal{C}^{\infty}-$ring $(A,\Phi)$ we have the function:

 $$\begin{array}{cccc}
     \psi_A: & {\rm Cong}\,(A,\Phi) & \rightarrow & \mathfrak{I}(A,\Phi) \\
      & R & \mapsto & \{ a \in A | (a,\Phi(0)) \in R \}
   \end{array}$$

and whenever $(A,\Phi)$ is a finitely generated $\mathcal{C}^{\infty}-$ring, we have seen in \textbf{Proposition \ref{Cemil}} that $\psi_A$ has $\varphi_A$ as inverse - which is a consequence of \textbf{Hadammard's Lemma}.\\

In order to show that $\psi_A$ is a bijection for any $\mathcal{C}^{\infty}-$ring $(A,\Phi)$, first we decompose it as a directed colimit of its finitely generated $\mathcal{C}^{\infty}-$subrings (cf. \textbf{Theorem \ref{LimMono}}):

$$(A,\Phi) \cong \varinjlim_{(A_i,\Phi_i) \subseteq_{{\rm f.g.}} (A,\Phi)}(A_i,\Phi\upharpoonright_{A_i})$$

and then we use \textbf{Proposition \ref{Cemil}} to obtain a bijection:

$$\widetilde{\varphi}: \varprojlim_{(A_i,\Phi_i) \subseteq_{{\rm f.g.}} (A,\Phi)}\mathfrak{I}(A_i,\Phi_i) \rightarrow \varprojlim_{(A_i,\Phi_i) \subseteq_{{\rm f.g.}} (A,\Phi)}{\rm Cong}\,(A_i,\Phi_i)$$

such that for every $(A_i,\Phi_i) \subseteq_{{\rm f.g.}} (A,\Phi)$ the following diagram commutes:

$$\xymatrixcolsep{5pc}\xymatrix{
\varprojlim_{(A_i,\Phi_i) \subseteq_{{\rm f.g.}} (A,\Phi)} \mathfrak{I}(A_i,\Phi_i) \ar[d] \ar[r]^{\widetilde{\varphi}} & \varprojlim_{(A_i,\Phi_i) \subseteq_{{\rm f.g.}} (A,\Phi)}{\rm Cong}\,(A_i,\Phi_i)\ar[d]\\
\mathfrak{I}(A_i,\Phi_i) \ar[r]_{\varphi_{A_i}} & {\rm Cong}(A_i,\Phi_i)}$$

where the vertical downward arrows are the canonical arrows of the projective limit.\\

Finally we show that there is a bijective correspondence, $\alpha$, between $$\varprojlim_{(A_i,\Phi_i) \subseteq_{{\rm f.g.}} (A,\Phi)} \mathfrak{I}(A_i,\Phi_i)$$ and $\mathfrak{I}\left( \varinjlim_{(A_i,\Phi_i) \subseteq_{{\rm f.g.}} (A,\Phi)}(A_i,\Phi_i)\right)$, and a bijective correspondence, $\beta$, between $\varprojlim_{(A_i,\Phi_i) \subseteq_{{\rm f.g.}} (A,\Phi)} {\rm Cong}\,(A_i,\Phi_i)$ and ${\rm Cong}\left( \varinjlim_{(A_i,\Phi_i) \subseteq_{{\rm f.g.}} (A,\Phi)}(A_i,\Phi_i)\right)$. In fact, the bijections $\alpha, \beta, \widetilde{\varphi}$ and $\psi$ are complete lattices isomorphisms.\\

By composing these bijections we prove that the congruences of $(\mathcal{C}^{\infty}(\mathbb{R}^{E}),\Phi_E)$ are classified by the ring-theoretic ideals of $(\mathcal{C}^{\infty}(\mathbb{R}^{E}),\Phi_E)$.\\

\begin{theorem}\label{LimMono} Let $(A, \Phi)$ be any $\mathcal{C}^{\infty}-$ring. There is a directed system of $\mathcal{C}^{\infty}-$rings, $((A_i, \Phi_i), \alpha_{ij})$, where each $(A_i, \Phi_i)$ is a finitely generated $\mathcal{C}^{\infty}-$ring and each $\alpha_{ij}: (A_i, \Phi_i) \to (A_j, \Phi_j)$ is a monomorphism such that:
$$(A,\Phi) \cong \varinjlim (A_i,\Phi_i)$$
\end{theorem}
\begin{proof}
Let $\mathcal{P}_f$ be the set of all finite subsets of $A$, on which we consider the partial order defined by inclusion, \textit{i.e.}, $(\forall i \in I)(\forall j \in I)(i \preceq j \leftrightarrow i \subseteq j)$.\\

Let $(A_i, \Phi_i) := \langle i \rangle$ denote the intersection of  all $\mathcal{C}^{\infty}-$subrings $(B,\Phi_B)$ of $(A, \Phi)$ which contain $i$, \textit{i.e.},
$$(A_i, \Phi_i) = (\langle i \rangle, \Phi_i) = \bigcap_{i \subset B} (B, \Phi') $$
where $\Phi_i = \Phi'$ is the $\mathcal{C}^{\infty}-$structure described in \textbf{Proposition \ref{Zimbabue}}

Note that $I = \{ \langle i \rangle = (A_i, \Phi_i)| i \in \mathcal{P}_f \}$, partially ordered by the inclusion, is a directed set, since it is not empty and given any $\langle i \rangle$ and $\langle j \rangle$ we can take, for instance, $k = i \cup j \in \mathcal{P}_f$ such that $\langle i \rangle \subseteq \langle k \rangle$ and $\langle j \rangle \subseteq \langle k \rangle$.\\

By \textbf{Proposition \ref{xiquexique}}, and using its notation,

$$\bigcup_{i \in \mathcal{P}_f} (\langle i \rangle ,\Phi')$$

is a $\mathcal{C}^{\infty}-$subring of $(A,\Phi)$. By construction, each $(A_i, \Phi_i)$ is a finitely generated $\mathcal{C}^{\infty}-$subring of $(A,\Phi)$.\\

We clearly have the following set-theoretic equality:
$$A = U(A, \Phi) = U \left( \bigcup_{i \in \mathcal{P}_f} (\langle i \rangle , \Phi')\right) = \bigcup_{i \in \mathcal{P}_f} \langle i \rangle,$$
hence the equality $\Phi' = \Phi$, so:

$$(A,\Phi) =  \bigcup_{i \in \mathcal{P}_f} (\langle i \rangle, \Phi')$$

For each $i \in I$ we define $A_i = \langle i \rangle$ and take $D(i) := (A_i, \Phi_i)$. Whenever $i \subseteq j$, we have the inclusion  $D(i \subseteq j):=\alpha_{ij}: (\langle i \rangle, \Phi_i) \hookrightarrow (\langle j \rangle, \Phi_j)$ which makes the following triangle commute:
$$\xymatrix{
(\langle i \rangle, \Phi_i) \ar[dr]_{\alpha_i}\ar[rr]^{\alpha_{ij}} & & (\langle j \rangle, \Phi_j )\ar[dl]^{\alpha_j} \\
    & \left(\cup_{i \in \mathcal{P}_f} A_i, \Phi' \right)&
}$$
where $\alpha_i : (\langle i \rangle, \Phi_i) \to (A, \Phi)$ and $\alpha_j : (\langle j \rangle, \Phi_j) \to (A, \Phi)$ are the $\mathcal{C}^{\infty}-$homo\-morphisms of inclusion. \\

Consider the following diagram, where $\mathcal{P}_f$ is viewed as a category:
$$\begin{array}{cccc}
    D: & \mathcal{P}_f & \to & \mathcal{C}^{\infty}{\rm Rng} \\
     & i \rightarrow j & \mapsto & (\langle i\rangle, \Phi_i) \stackrel{\alpha_{ij}}{\rightarrow} (\langle j \rangle, \Phi_j)
  \end{array}$$

We claim that:
$$\varinjlim_{i \in \mathcal{P}_f} D(i) \cong (A, \Phi)$$

Let $(B,\Psi)$ be any $\mathcal{C}^{\infty}-$ring and $h_i: (\langle i \rangle, \Phi_i) \to (B, \Psi)$ and $h_j: (\langle j \rangle, \Phi_j) \to (B, \Psi)$ be any two $\mathcal{C}^{\infty}-$homomorphisms. We are going to show that there exists a unique $\mathcal{C}^{\infty}-$homomorphism $h: (\cup_{i \in \mathcal{P}_f}A_i, \Phi) \to (B, \Psi)$ such that the following diagram commutes:
$$
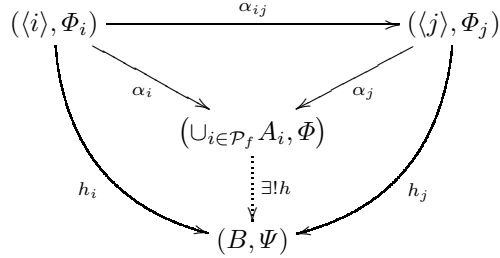
\xymatrix{
(\langle i \rangle, \Phi_i) \ar@/_2pc/[ddr]_{h_i}\ar[dr]_{\alpha_i}\ar[rr]^{\alpha_{ij}} & & (\langle j \rangle, \Phi_j) \ar[dl]^{\alpha_j} \ar@/^2pc/[ddl]^{h_j} \\
    & \left(\cup_{i \in \mathcal{P}_f} A_i, \Phi \right) \ar@{.>}[d]^{\exists ! h}&  \\
    &   (B, \Psi) &
}$$

We now define the function which underlies $h$ in the following way: given $a \in \cup_{i \in \mathcal{P}_f} A_i$ there is some $i_0 \in \mathcal{P}_f$ such that $a \in i_0$. For any $j_0 \in \mathcal{P}_f$ such that $a \in j_0$, since $I$ is directed, there is some $k_0 \in \mathcal{P}_f$ such that $k_0 \supseteq i_0, j_0$, so we have $a \in k_0$ and the following commutative diagram:
$$
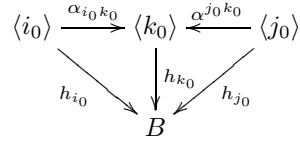
\xymatrix{
\langle i_0 \rangle \ar[r]^{\alpha_{i_0k_0}} \ar[dr]_{h_{i_0}} & \langle k_0 \rangle \ar[d]^{h_{k_0}} &  \langle j_0 \rangle \ar[l]_{\alpha^{j_0k_0}} \ar[dl]^{h_{j_0}}\\
  & B &
}$$
so:
$$h_{i_0}(a) = h_{k_0}(\alpha_{i_0k_0}(a)) = h_{k_0}(a) = h_{k_0}(\alpha_{j_0k_0}(a)) = h_{j_0}(a).$$

Hence, given $a \in \cup_{i \in \mathcal{P}_f} A_i$, choose any $i_0 \in \mathcal{P}_f$ such that $a \in i_0$ and define $h(a)  = h_{i_0}(a)$ (up to here $h$ is merely a function, not a $\mathcal{C}^{\infty}-$homomorphism).\\

Now we prove that $h$ is a $\mathcal{C}^{\infty}-$homomorphism.\\

Let $f \in \mathcal{C}^{\infty}(\R^n, \mathbb{R})$ be any $n-$ary function symbol. We must show that the following diagram commutes:
$$\xymatrixcolsep{5pc}
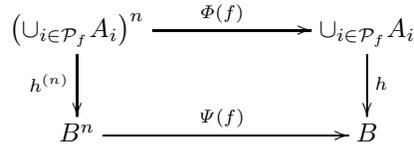
\xymatrix{\left( \cup_{i \in \mathcal{P}_f} A_i\right)^n \ar[d]_{h^{(n)}}\ar[r]^{\Phi(f)} & \cup_{i \in \mathcal{P}_f} A_i \ar[d]^{h}\\
B^n \ar[r]^{\Psi(f)} & B}$$
that is,
$$(\forall a_1)(\forall a_2)\cdots(\forall a_n)(\Psi(f)(h(a_1), h(a_2), \cdots, h(a_n))) = h(\Phi(f)(a_1, a_2, \cdots,a_n)).$$

Let $i_1, i_2, \cdots, i_n \in \mathcal{P}_f$ be such that $a_1 \in \langle i_1\rangle, a_2 \in \langle i_2 \rangle, \cdots, a_n \in \langle i_n \rangle$, and take $j \in \mathcal{P}_f$ such that $j \supseteq i_1, \cdots, i_n$, so
$$a_1, a_2, \cdots, a_n \in \langle j \rangle$$

For any $\ell \in \{ 1,2, \cdots, n \}$, $h_{i}(a_{\ell}) = h_j(a_{\ell})$, so $h(a_1) = h_j(a_1), h(a_2) = h_j(a_2), \cdots, h(a_n) = h_j(a_n)$.\\

We have:
\begin{multline*}
  \Psi(f)(h(a_1), h(a_2), \cdots, h(a_n)) = \Psi(f)(h_j(a_1), h_j(a_2), \cdots, h_j(a_n)) \stackrel{{\rm (1)}}{=} \\
  = h_j(\Phi_j(f) (a_1,a_2, \cdots, a_n)) \stackrel{{\rm (2)}}{=} h_j(\Phi(f)(a_1,a_2, \cdots, a_n)) \stackrel{{\rm (3)}}{=} h(\Phi_i(f)(a_1,a_2,\cdots,a_n))
\end{multline*}

where (1) is due to the fact that $h_j : (\langle i \rangle, \Phi_i) \to (B,\Psi)$ is a $\mathcal{C}^{\infty}-$homomorphism, (2) is due to the fact that $(\langle j \rangle, \Phi_j)$ is a $\mathcal{C}^{\infty}-$subring of $(\cup_{i \in \mathcal{P}_f}A_i, \Phi)$  and (3) occurs by the very definition of $h$. The uniqueness is granted by the property that $h$ must satisfy as a function.\\

Under those circumstances,
$$\left( \bigcup_{i \in \mathcal{P}_f}\langle i \rangle, \Phi \right) \cong \varinjlim_{i \in I} (A_i, \Phi_i)$$

so

$$(A,\Phi) = \left( \bigcup_{i \in \mathcal{P}_f}\langle i \rangle, \Phi \right) \cong \varinjlim_{i \in \mathcal{P}_f} (A_i, \Phi_i)$$
\end{proof}

\begin{remark}\label{Orangotango}Let $((A_i,\Phi_i)_{i \in I}, \alpha_{ij}: (A_i,\Phi_i) \to (A_j,\Phi_j))$ be an inductive directed system of $\mathcal{C}^{\infty}-$rings, and let $(J_{\ell})_{\ell \in I} \in \varprojlim_{\ell \in I} \mathfrak{I}(A_{\ell},\Phi_{\ell})$.

For every $i \in I$, we have the maps:

$$\begin{array}{cccc}
    {\alpha_i}^{*}: & \mathfrak{I}\left( \varinjlim_{i \in I} (A_i, \Phi_i)\right) & \rightarrow & \mathfrak{I}(A_i,\Phi_i) \\
     & J & \mapsto & \alpha_i^{\dashv}[J]
  \end{array}$$

$$\begin{array}{cccc}
    \widehat{\alpha}: & \mathfrak{I}\left( \varinjlim_{i \in I} (A_i, \Phi_i)\right) & \rightarrow & \mathfrak{I}(A_i,\Phi_i) \\
     & J & \mapsto & (\alpha_i^{\dashv}[J])_{i \in I}
  \end{array}$$

and the following limit diagram:

$$\xymatrixcolsep{5pc}\xymatrix{
 & \varprojlim_{i \in I} \mathfrak{I}(A_i,\Phi_i) \ar[dl]_{\alpha_i^{*}} \ar[dr]^{\alpha_j^{*}} & \\
\mathfrak{I}(A_i, \Phi_i) & & \mathfrak{I}(A_j,\Phi_j) \ar[ll]^{\alpha_{ij}^{*}}
},$$

where:

$$\begin{array}{cccc}
    {\alpha_{ij}}^{*}: & \mathfrak{I}(A_j,\Phi_j) & \rightarrow & \mathfrak{I}(A_i,\Phi_i) \\
     & J_j & \mapsto & {\alpha_{ij}}^{\dashv}[J_j]
  \end{array}$$

We note, first, that $\alpha$  maps ideals of $\varinjlim_{i \in I} (A_i,\Phi_i)$ to an element of $\varprojlim_{i \in I} \mathfrak{I}(A_i,\Phi_i)$.

In fact, given any ideal $J \in \mathfrak{I}\left( \varinjlim_{i \in I} (A_i, \Phi_i)\right)$, since for every $i \in I$, ${\alpha_i}$ is a $\mathcal{C}^{\infty}-$homomorphism, it follows that for every $i \in I$, $\alpha_i^{*}(J) = {\alpha_i}^{\dashv}[J]$ is an ideal of $(A_i, \Phi_i)$, so $\alpha(J) = (\alpha_i^{*}(J))_{i \in I} \in \displaystyle\prod_{i \in I} \mathfrak{I}(A_i, \Phi_i)$. Moreover, the family $(\alpha_i^{*}(J))_{i \in I}$ is compatible, since:

$$(\forall i \in I)(\forall j \in I)(i \preceq j)(\alpha_j \circ \alpha_{ij} = \alpha_i \Rightarrow \alpha_i^{*} = {\alpha_{ij}}^{*}\circ \alpha_j^{*})$$

and

$$(\forall i \in I)(\forall j \in I)(i \preceq j)({\alpha_i}^{*}(J)= ({\alpha_{ij}}^{*}\circ \alpha_j^{*})(J) = {\alpha_{ij}}^{*}(\alpha_j^{*}(J))$$

so

$$\alpha(J) = (\alpha_i^{*}(J))_{i \in I} \in \varprojlim_{i \in I} \mathfrak{I}(A_i, \Phi_i).$$
\end{remark}

\begin{proposition}\label{paulinho}Let $(I, \preceq)$ be a directed partially ordered set and
$$\{ (A_i, \Phi_i), \alpha_{ij}: (A_i, \Phi_i) \to (A_j,\Phi_j)\}_{i,j \in I}$$
be a directed inductive system of $\mathcal{C}^{\infty}-$rings and $\mathcal{C}^{\infty}-$homomorphisms. For every $i \in I$, we have the map:

$$\begin{array}{cccc}
    {\alpha_i}^{*}: & \mathfrak{I}\left( \varinjlim_{i \in I} (A_i, \Phi_i)\right) & \rightarrow & \mathfrak{I}(A_i,\Phi_i) \\
     & J & \mapsto & \alpha_i^{\dashv}[J]
  \end{array}$$

By the universal property of $\varprojlim_{i \in I} \mathfrak{I}(A_i,\Phi_i)$, there is a unique $\alpha: \mathfrak{I}\left( \varinjlim_{i \in I} (A_i,\Phi_i)\right) \to \varprojlim_{i \in I} \mathfrak{I}(A_i, \Phi_i)$ such that for every $i \in I$ the following diagram commutes:

$$\xymatrixcolsep{5pc}\xymatrix{
\mathfrak{I}\left( \varinjlim_{i \in I} (A_i,\Phi_i)\right) \ar[d]_{\exists ! \alpha} \ar@/^2pc/[rdd]^{\alpha_i^{*}} & \\
\varprojlim_{i \in I} \mathfrak{I}(A_i,\Phi_i) \ar[dr]_{\pi_i\upharpoonright_{\varprojlim_{i \in I} \mathfrak{I}(A_i,\Phi_i)}} & \\
 & \mathfrak{I}(A_i,\Phi_i)}$$

that is, such that ${\alpha_i}^{*} = \pi_i\upharpoonright_{\varprojlim_{i \in I} \mathfrak{I}(A_i,\Phi_i)} \circ \alpha$.

We have, thus:

$$\begin{array}{cccc}
    \alpha: & \mathfrak{I}\left( \varinjlim_{i \in I} (A_i, \Phi_i)\right) & \rightarrow & \varprojlim_{i \in I} \mathfrak{I}(A_i,\Phi_i) \\
     & J & \mapsto & (\alpha_i^{*}(J))_{i \in I}
  \end{array}$$

For any $(J_i)_{i \in I} \in \varprojlim_{i \in I} \mathfrak{I}(A_i,\Phi_i)$, $\bigcup_{i \in I}\alpha_i[J_i]$ is an ideal of $\varinjlim_{i \in I} (A_i,\Phi_i)$ and the map:

$$\begin{array}{cccc}
    \alpha': & \varprojlim_{i \in I} \mathfrak{I}(A_i,\Phi_i) & \rightarrow & \mathfrak{I}\left( \varinjlim_{i \in I} (A_i, \Phi_i)\right) \\
     & (J_i)_{i \in I} & \mapsto & \bigcup_{i \in I}\alpha_i[J_i]
  \end{array}$$

is an inverse for $\alpha$, so $\alpha$ is a bijection.
\end{proposition}
\begin{proof}
In ${\rm \bf Set}$, we have:

$$U\left(\varinjlim_{i \in I} (A_i,\Phi_i)\right) = \dfrac{\bigcup_{i \in I} A_i \times \{ i\}}{\thicksim}$$

where:

$$(a_i,i)\thicksim (a_j,j) \iff (\exists k \in I)(i,j \preceq k)(\alpha_{ik}(a_i)=\alpha_{jk}(a_j))$$

and

$$U \left(\varprojlim_{i \in I} \mathfrak{I}(A_i,\Phi_i)\right) = \left\{ (J_{\ell})_{\ell \in I} \in \displaystyle\prod_{i \in I} \mathfrak{I}(A_i,\Phi_i) | (\forall i,j \in I)(i \preceq j)({\alpha_{ij}}^{*}(J_j) = J_i) \right\}$$

Given the following colimit diagram:

$$\xymatrixcolsep{5pc}\xymatrix{
 & \varinjlim_{i \in I} (A_i,\Phi_i) & \\
(A_i, \Phi_i) \ar[ur]^{\alpha_i} \ar[rr]_{\alpha_{ij}}& & (A_j,\Phi_j) \ar[ul]_{\alpha_j}
}$$

with:

$$\begin{array}{cccc}
    \alpha_i: & (A_i,\Phi_i) & \rightarrow & \varinjlim_{i \in I} (A_i,\Phi_i) \\
     & a_i & \mapsto & [(a_i,i)]
  \end{array}$$

consider the following limit diagram:

$$\xymatrixcolsep{5pc}\xymatrix{
 & \varprojlim_{i \in I} \mathfrak{I}(A_i,\Phi_i) \ar[dl]_{\alpha_i^{*}} \ar[dr]^{\alpha_j^{*}} & \\
\mathfrak{I}(A_i, \Phi_i) & & \mathfrak{I}(A_j,\Phi_j) \ar[ll]^{\alpha_{ij}^{*}}
},$$

where:

$$\begin{array}{cccc}
    {\alpha_{ij}}^{*}: & \mathfrak{I}(A_j,\Phi_j) & \rightarrow & \mathfrak{I}(A_i,\Phi_i) \\
     & J_j & \mapsto & {\alpha_{ij}}^{\dashv}[J_j]
  \end{array}$$










Given any $(J_i)_{i \in I}$, we are going to show that $\alpha'((J_i)_{i \in I}) \in \mathfrak{I}\left( \varinjlim_{i \in I} (A_i,\Phi_i)\right)$.\\

Since $\varinjlim_{i \in I}(A_i, \Phi_i)$ is directed, given $\overline{x}, \overline{y} \in \bigcup_{i \in I} \alpha_i[J_i]$, there are $i,j \in I$, $x_i \in J_i$ and $y_j \in J_j$ such that $\alpha_i(x_i) =[(x_i,i)] = \overline{x}$ and $\alpha_j(y_j) = [(y_j,j)] = \overline{y}$. Also, since the colimit is directed, there is $k \in I$ with $i,j \leq k$ such that:

$$\alpha_{ik}(x_i) \in \alpha_{ik}[J_i], \alpha_{jk}(y_j) \in \alpha_{jk}[J_j],$$

so

$$\alpha_k(\alpha_{ik}(x_i)), \alpha_k(\alpha_{jk}(y_j)) \in \alpha_k[J_k]$$

Since $\alpha_k[J_k]$ is an ideal in $\alpha_k[A_k]$, it follows that:

$$\alpha_k(\alpha_{ik}(x_i)) - \alpha_k(\alpha_{jk}(y_j)) \in \alpha_k[J_k]$$

and

$$\overline{x} - \overline{y} = \alpha_k(\alpha_{ik}(x_i))- \alpha_k(\alpha_{jk}(y_j)) \in \alpha_k[J_k] \subseteq \bigcup_{i \in I}\alpha_i[J_i].$$

Given $\overline{x} \in \varinjlim_{i \in I} \alpha_i[J_i]$ there are some $i \in I$, $x_i \in J_i$ such that $\overline{x} = \alpha_i(x_i) = [(x_i,i)]$, and given  $\overline{a} \in \varinjlim_{i \in I} (A_i,\Phi_i)$, there are some $j \in I$ and some $a_j \in A_j$ such that $\overline{a} = \alpha_j(a_j) = [(a_j,j)]$.\\

Since the colimit is directed, there is some $k \in I$ with $i,j \leq k$ such that:

$$\alpha_{ik}(x_i) \in \alpha_{ik}[J_i], \alpha_{jk}(a_j) \in \alpha_{jk}[A_j]$$

so

$$\alpha_i(x_i) = \alpha_k(\alpha_{ik}(x_i)) \in \alpha_i[\alpha_{ik}[J_i]] = \alpha_k[J_k]$$
and
$$\alpha_j(a_j) = \alpha_k(\alpha_{jk}(a_j)) \in \alpha_j[\alpha_{jk}[J_j]] = \alpha_k[J_k]$$

Since $\alpha_k[J_k]$ is an ideal of $\alpha_k[A_k]$, it follows that:

$$\alpha_{k}(\alpha_{ik}(x_i)) \cdot \alpha_{k}(\alpha_{jk}(a_j)) \in \alpha_k[J_k].$$

Thus,

\begin{multline*}
\overline{x}\cdot \overline{a} = \alpha_i(x_i)\cdot \alpha_j(a_j) =  \alpha_{k}(\alpha_{ik}(x_i)) \cdot \alpha_{k}(\alpha_{jk}(a_j))=\\
=\alpha_k(\alpha_{ik}(x_i))\cdot \alpha_k(\alpha_{jk}(a_j)) \in \alpha_k[J_k] \subseteq \bigcup_{i \in I} \alpha_i[J_i]
\end{multline*}

so $\alpha'$ maps elements of $\varprojlim_{i \in I} \mathfrak{I}(A_i,\Phi_i)$ to ideals of $\varinjlim_{i \in I} (A_i, \Phi_i)$.\\

\textbf{Claim:} $\alpha \circ \alpha' = {\rm id}_{\varprojlim_{i \in I} \mathfrak{I}(A_i,\Phi_i)}$.\\

Given $(J_{\ell})_{\ell \in I} \in \varprojlim_{i \in I} \mathfrak{I}(A_i,\Phi_i)$, we have:

$$\alpha'((J_{\ell})_{\ell \in I}) := \bigcup_{\ell \in I} \alpha_{\ell}[J_{\ell}],$$

so

$$\alpha(\alpha'((J_{\ell})_{\ell \in I})) = \alpha\left( \bigcup_{\ell \in I} \alpha_{\ell}[J_{\ell}]\right):= \left( {\alpha_{i}}^{\dashv}\left[ \bigcup_{\ell \in I} \alpha_{\ell}[J_{\ell}]\right]\right)_{i \in I}.$$

We are going to show that for every $i \in I$,

$${\alpha_i}^{*}(\alpha'((J_{\ell})_{\ell \in I})) = J_i$$

Given $i \in I$, on the one hand we have:

$$J_i \subseteq {\alpha_i}^{\dashv}[\alpha_i[J_i]] \subseteq {\alpha_i}^{\dashv}\left[ \bigcup_{\ell \in I} \alpha_{\ell}[J_{\ell}]\right] = {\alpha_i}^{*}(\alpha'((J_{\ell})_{\ell \in I})) = \bigcup_{\ell \in I} {\alpha_i}^{\dashv}[\alpha_{\ell}[J_{\ell}]]$$

On the other hand, given $a_i \in {\alpha_i}^{*}(\alpha'((J_{\ell})_{\ell \in I})) = {\alpha_i}^{\dashv}\left[ \bigcup_{\ell \in I}\alpha_{\ell}[J_{\ell}]\right]$, that is,  $a_i \in A_i$ such that $\alpha_i(a_i) \in \bigcup_{\ell \in I}{\alpha_{\ell}}[J_{\ell}]$, there is some $j \in I$ and some $b_j \in J_j$ such that:

$$\alpha_i(a_i) = \alpha_j(b_j).$$

Since the system is directed, there is some $k \in I$, $i,j \preceq k$  such that $\alpha_{ik}(a_i) = \alpha_{jk}(b_j)$.\\

By compatibility we have ${\alpha_{jk}}^{\dashv}[J_k]={\alpha_{jk}}^{*}(J_k) = J_j$, so in particular $\alpha_{jk}^{\dashv}[J_k] \supseteq J_j$ and $ \alpha_{jk}(b_j) \in\alpha_{jk}[J_j] \subseteq \alpha_{jk}[{\alpha_{jk}}^{\dashv}[J_k]] \subseteq J_k$.\\

Thus, taking $c_k = \alpha_{jk}(b_j)$, we have:

$$c_k = \alpha_{jk}(b_j) \in \alpha_{jk}^{\dashv}[J_j] \subseteq J_k,$$

that is,

$$c_k \in J_k.$$

For any $k' \geq k$ we have:

$$\alpha_{ik'}(a_i) = \alpha_{kk'}(c_k) \in J_{k'}$$

so

$$a_i \in {\alpha_{ik'}}^{\dashv}[J_{k'}] = J_i.$$

Hence:

$$(\forall i \in I)\left(\bigcup_{\ell \in I}{\alpha_i}^{\dashv}[\alpha_{\ell}[J_{\ell}]] = J_i\right)$$

and

$$\alpha(\alpha'((J_i)_{i \in I})) = (J_i)_{i \in I}.$$

Since $(J_{i})_{i \in I}$ is arbitrary, we have:

$$\alpha \circ \alpha' = {\rm id}_{\varprojlim_{i \in I}\mathfrak{I}(A_i,\Phi_i)}.$$

On the other hand, given $J \in \mathfrak{I}\left( \varinjlim_{i \in I} (A_i,\Phi_i)\right)$, we have $\alpha(J) = ({\alpha_i}^{*}(J))_{i \in I}$ and:

$$\alpha'(\alpha(J)) = \alpha'(({\alpha_i}^{*}(J))_{i \in I}) = \bigcup_{i \in I} \alpha_i[{\alpha_i}^{*}(J)].$$

On the one hand, since for every $i \in I$, $\alpha_i[{\alpha_i}^{*}(J)] = \alpha_i[{\alpha_i}^{\dashv}[J]] \subseteq J$, we have $\alpha'(\alpha(J)) \subseteq J$.\\

Recall that given $J \in \mathfrak{I}\left( \varinjlim_{i \in I}(A_i,\Phi_i)\right)$, we have:

$$J \subseteq \varinjlim_{i \in I} (A_i, \Phi_i) = \bigcup_{i \in I} \alpha_i[A_i],$$

so given $\overline{x} \in J \subseteq \bigcup_{i \in I} \alpha_i[A_i]$, there are  some $i_0 \in I$ and some $x_{i_0} \in {\alpha_{i_0}}^{*}(J) = {\alpha_{i_0}}^{\dashv}[J]$ such that:

$$\overline{x} = [(x_{i_0},i_0)]$$

and since $x_{i_0} \in \alpha_{i_0}^{*}(J)$, we have $[(x_{i_0},i_0)] \in \alpha_{i_0}[{\alpha_{i_0}}^{*}(J)] \subseteq \bigcup_{i \in I} \alpha_i[\alpha_i^{*}(J)]$, so

$$\overline{x} \in \bigcup_{i \in I} \alpha_i[\alpha_i^{*}(J)] = \alpha'(({\alpha_i}^{*}(J))_{i \in I}) = \alpha'(\alpha(J))$$

and:

$$\alpha'(\alpha(J)) = J.$$

Since $J$ is an arbitrary ideal of $\varprojlim_{i \in I}\mathfrak{I}(A_i,\Phi_i)$, it follows that:

$$\alpha' \circ \alpha = {\rm id}_{\mathfrak{I}\left(\varinjlim_{i \in I}(A_i,\Phi_i)\right)}.$$

so $\alpha$ is a bijection whose inverse is $\alpha'$.
\end{proof}

The proof of the following proposition is similar to the proof we just made, so we are going to omit it.\\

\begin{proposition}\label{gogo}Let $(I, \preceq)$ be a directed partially ordered set and $\{ (A_i, \Phi_i), \alpha_{ij}: (A_i, \Phi_i) \to (A_j,\Phi_j)\}_{i,j \in I}$ be a directed inductive system of $\mathcal{C}^{\infty}-$rings and\\ $\mathcal{C}^{\infty}-$homomorphisms. The following function is a bijection:

$$\begin{array}{cccc}
    \beta: & {\rm Cong}\left( \varinjlim_{i \in I} (A_i, \Phi_i)\right) & \rightarrow & \varprojlim_{i \in I} {\rm Cong}(A_i,\Phi_i) \\
     & R & \mapsto & ((\alpha_i \times \alpha_i)^{\dashv}(R))_{i \in I}
  \end{array}$$
whose inverse is given by:
$$\begin{array}{cccc}
    \beta': & \varprojlim_{i \in I} {\rm Cong}(A_i,\Phi_i) & \rightarrow & {\rm Cong}\left( \varinjlim_{i \in I} (A_i, \Phi_i)\right)  \\
     & (R_i)_{i \in I} & \mapsto & \varinjlim_{i \in I} R_i
  \end{array}$$

\end{proposition}

The following result extends \textbf{Proposition \ref{Cemil}} in the sense that it shows us, with details, that the congruences of any free $\mathcal{C}^{\infty}-$ring are classified by their ring-theoretic ideals (in the finitely generated case it follows from Hadamard's lemma, and this case is used here).

\begin{lemma}\label{Asso}The congruences of $(\mathcal{C}^{\infty}(\mathbb{R}^{E}),\Phi_E)$, the free $\mathcal{C}^{\infty}-$ring determined by the set $E$, are classified by their ring-theoretic ideals.
\end{lemma}
\begin{proof}
By \textbf{Proposition \ref{Monica}}, for any finite $E$ the result holds.\\

Suppose $E$ is any set (not necessarily finite). We have, by definition,

$$(\mathcal{C}^{\infty}(\mathbb{R}^{E}), \Phi_E) \cong \varinjlim_{E' \subseteq_f E} (\mathcal{C}^{\infty}(\mathbb{R}^{E'}, \mathbb{R}), \Phi_{E'})$$

where $\Phi_{E}$ is the $\mathcal{C}^{\infty}-$structure of the colimit.\\

The family of bijections $\{ \psi_{E'}: {\rm Cong}(\mathcal{C}^{\infty}(\mathbb{R}^{E'}),\Phi_{E'}) \rightarrow \mathfrak{I}(\mathcal{C}^{\infty}(\mathbb{R}^{E'}),\Phi_{E'}) \}$ always exist, so it yields a coherent bijection:

$$\widetilde{\psi}:  \varprojlim_{E' \subseteq_f E}{\rm Cong}(\mathcal{C}^{\infty}(\mathbb{R}^{E'}),\Phi_{E'}) \rightarrow \varprojlim_{E' \subseteq_f E} \mathfrak{I}(\mathcal{C}^{\infty}(\mathbb{R}^{E'}),\Phi_{E'}).$$

We have the following commutative diagram:

$$\xymatrixcolsep{5pc}\xymatrix{
{\rm Cong}\,(\mathcal{C}^{\infty}(\mathbb{R}^{E'}), \Phi_{E'}) \ar[r]^{\psi_{E'}} & \mathfrak{I}(\mathcal{C}^{\infty}(\mathbb{R}^{E'}), \Phi_{E'})\\
\varprojlim_{E' \subseteq_f E} {\rm Cong}(\mathcal{C}^{\infty}(\mathbb{R}^{E'}),\Phi_{E'}) \ar[r]^{\widetilde{\psi}} \ar[u] & \varprojlim_{E' \subseteq_f E} \mathfrak{I}(\mathcal{C}^{\infty}(\mathbb{R}^{E'}),\Phi_{E'}) \ar[u] \\
{\rm Cong}(\mathcal{C}^{\infty}(\mathbb{R}^{E}),\Phi_E) \ar[u]^{\alpha} \ar[r]_{\psi_E} & \mathfrak{I}(\mathcal{C}^{\infty}(\mathbb{R}^{E}),\Phi_{E})\ar[u]^{\beta}} $$

where the two upper  vertical arrows are the canonical arrows of the projective limit,
$\alpha$ is given in \textbf{Proposition \ref{paulinho}} and $\beta$ is given in \textbf{Proposition \ref{gogo}}.\\

Since $\alpha, \beta$ and $\widetilde{\psi}$ are bijections, it follows that $\psi_E: {\rm Cong}(\mathcal{C}^{\infty}(\mathbb{R}^{E}),\Phi_E) \rightarrow \mathfrak{I}(\mathcal{C}^{\infty}(\mathbb{R}^{E}),\Phi_{E}) $ is a bijection.
\end{proof}

The following lemma is a well-known result of Universal Algebra applied to $\mathcal{C}^{\infty}-$rings:\\

\begin{lemma}\label{Aracy}Let $(A,\Phi)$ be a $\mathcal{C}^{\infty}-$ring and let $R \in {\rm Cong}\,(A,\Phi)$. Given the quotient $\mathcal{C}^{\infty}-$homomorphism:

$$\begin{array}{cccc}
    q_R: & (A,\Phi) & \rightarrow & \left( \dfrac{A}{R}, \overline{\Phi}\right) \\
     & x & \mapsto & x + R
  \end{array}$$

we have the bijection:

$$\begin{array}{cccc}
    (q_R)_{*}: & \{ S \in {\rm Cong}\,(A,\Phi) | R \subseteq S\} & \rightarrow & {\rm Cong}\left( \dfrac{A}{R}, \overline{\Phi}\right) \\
     & S & \mapsto & \{ (q_R(s),q_R(t))\in \frac{A}{R}\times \frac{A}{R} | (s,t) \in S\}
  \end{array}$$

whose inverse is given by:

$$\begin{array}{cccc}
    (q_R)^*: & {\rm Cong}\left( \dfrac{A}{R}, \overline{\Phi}\right) & \rightarrow &
    \{ S \in {\rm Cong}(A,\Phi) | R \subseteq S\} \\
     & S' & \mapsto & (q_R \times q_R)^{\dashv}[S']
  \end{array}$$
\end{lemma}

As a consequence of the above lemma, we have:\\

\begin{lemma}\label{lambor}Let $(A,\Phi)$ be a $\mathcal{C}^{\infty}-$ring and $R \in {\rm Cong}\,(A,\Phi)$, and suppose that $\psi_A: {\rm Cong}\,(A,\Phi) \rightarrow \mathfrak{I}(A,\Phi)$ is a bijection with an inverse, $\varphi_A: \mathfrak{I}(A,\Phi) \rightarrow {\rm Cong}\,(A,\Phi)$. Under those circumstances, the quotient $\mathcal{C}^{\infty}-$homomorphism:

$$q_R: (A,\Phi) \rightarrow \left( \dfrac{A}{R}, \overline{\Phi}\right)$$

induces a pair of inverse bijections:

$$\begin{array}{cccc}
    (q_R)_{+}: & \mathfrak{I}\left( \dfrac{A}{R}, \overline{\Phi}\right) & \rightarrow & \{ I' \in \mathfrak{I}(A,\Phi) | \psi_A(R) \subseteq I'\} \\
     & J & \mapsto & q_R[J]
  \end{array}$$

and

$$\begin{array}{cccc}
    (p_R)^{-}: & \{ I' \in \mathfrak{I}(A,\Phi) | \psi_A(R) \subseteq I'\} & \rightarrow & \mathfrak{I}\left( \dfrac{A}{R}, \overline{\Phi}\right) \\
     & J' & \mapsto & (q_R)^{\dashv}[J']
  \end{array}$$
\end{lemma}
\begin{proof}
The proof is a direct consequence of \textbf{Lemma \ref{Aracy}} and a diagram chase.
\end{proof}

\begin{proposition}Let $(A,\Phi)$ and $(B,\Psi)$ be two $\mathcal{C}^{\infty}-$rings and let $h: (A,\Phi) \rightarrow (B, \Psi)$ be a surjective $\mathcal{C}^{\infty}-$homomorphism. The following functions are bijections:

$$\begin{array}{cccc}
    h^{*}: & {\rm Cong}\,(B,\Psi) & \rightarrow & \{S \in {\rm Cong}\,(A,\Phi) \mid \ker(h) \subseteq S\} \\
     & R & \mapsto & (h \times h)^{\dashv}[R]
  \end{array}$$

$$\begin{array}{cccc}
    {h}^{-}: & \mathfrak{I}(A,\Phi) & \rightarrow & \{ I' \in \mathfrak{I}(B,\Psi) \mid \varphi_A(\ker(h)) \subseteq I'\} \\
     & J & \mapsto & h^{\dashv}[J]
  \end{array}$$
\end{proposition}
\begin{proof}
By the \textbf{Theorem of the $\mathcal{C}^{\infty}-$Isomorphism}, we have the $\mathcal{C}^{\infty}-$isomorphism:

$$\begin{array}{cccc}
    \overline{h}: & (\dfrac{A}{\ker(h)}, \overline{\Phi}) & \rightarrow & (B,\Psi) \\
     & a + \ker(h) & \mapsto & h(a)
  \end{array}$$

so

$$\begin{array}{cccc}
    \alpha: & {\rm Cong}\left( \dfrac{A}{\ker(h)}, \overline{\Phi}\right) & \rightarrow & {\rm Cong}(B,\Psi) \\
     & S' & \mapsto & \{ (\overline{h}(s), \overline{h}(s')) \mid (s,s') \in S'\}
  \end{array}$$

is a bijection whose inverse is given by:

$$\begin{array}{cccc}
    \beta: & {\rm Cong}\,(B,\Psi) & \rightarrow & {\rm Cong}\left( \dfrac{A}{\ker(h)}, \overline{\Phi}\right) \\
     & R & \mapsto & \{ (\overline{h}^{-1}(r), \overline{h}^{-1}(r')) | (r,r') \in R\} = \{ (a + \ker(h), a'+\ker(h)) | (h(a),h(a')) \in R \}
  \end{array}$$

Taking $R = \ker(h)$ in the \textbf{Lemma \ref{Aracy}} yields a bijection $(q_{\ker(h)})^{*}:  {\rm Cong}\left(\dfrac{A}{\ker(h)}\right) \to \{ S \in {\rm Cong}\,(A) \mid \ker(h) \subseteq S\}$.\\

Now it suffices to prove that $h^{*}= (q_{\ker(h)})^{*}\circ \beta$.\\

Since $\overline{h}$ is an isomorphism, note that $\beta = \overline{h}^{*}$, and as $h = \overline{h} \circ q_{\ker(h)}$, we get $h^{*} = q_{\ker(h)}^{*}\circ \overline{h}^{*}$, as we claimed.\\

The proof for $h^{-}$ is analogous, since $h^{-} = (q_{\ker(h)})^{-}\circ \overline{h}^{-}$.

\end{proof}

The following proposition gives us a description of $\mathcal{C}^{\infty}-$rings via generators and relations.\\

\begin{proposition}\label{Meirelles}Let $(A,\Phi)$ be any $\mathcal{C}^{\infty}-$ring. The $\mathcal{C}^{\infty}-$congruences of $(A,\Phi)$ are classified by the ring-theoretic ideals of $(A,\Phi)$.
\end{proposition}
\begin{proof}
Given any $\mathcal{C}^{\infty}-$ring $(A,\Phi)$, we have the following isomorphism:

$$\overline{\varepsilon}_A: \dfrac{\mathcal{C}^{\infty}(\mathbb{R}^{U(A)})}{\ker(\varepsilon_A)} \stackrel{\cong}{\rightarrow} A$$

so we can write:

$$A \cong \dfrac{L(A)}{\ker(\varepsilon_A)} = \dfrac{\mathcal{C}^{\infty}(\mathbb{R}^{A})}{\ker(\varepsilon_A)}.$$

By \textbf{Lemma \ref{Asso}}, we have a bijection:

$$\psi_A: {\rm Cong}\,(\mathcal{C}^{\infty}(\mathbb{R}^{A}),\Phi_A) \rightarrow \mathfrak{I}(\mathcal{C}^{\infty}(\mathbb{R}^{A}),\Phi_A)$$

Note that given any $R \in {\rm Cong}\,(\mathcal{C}^{\infty}(\mathbb{R}^{A}),\Phi_A)$ such that $\ker(\varepsilon_A) \subseteq R$, we have $\psi_A(\ker(\varepsilon_A)) = \{ g \in \mathcal{C}^{\infty}(\mathbb{R}^{A})| (g,0) \in \ker(\varepsilon_A)\} \subseteq \{ g \in \mathcal{C}^{\infty}(\mathbb{R}^{A})| (g,0) \in R\} = \psi_A(R)$, so:

$$\psi_A[\{ R \in {\rm Cong}\,(\mathcal{C}^{\infty}(\mathbb{R}^{A}),\Phi_A) \mid \ker(\varepsilon_A) \subseteq R\}] \subseteq \{ I' \in \mathfrak{I}(\mathcal{C}^{\infty}(\mathbb{R}^{A}),\Phi_A) \mid \psi_A(\ker(\varepsilon_A)) \subseteq I'\}.$$

Also, let $I' \in \mathfrak{I}(\mathcal{C}^{\infty}(\mathbb{R}^{A}),\Phi_A)$ be such that $\psi_A(\ker(\varepsilon_A))=\{ g \in \mathcal{C}^{\infty}(\mathbb{R}^{A}) \mid (g,\Phi_A(0)) \in \ker(\varepsilon_A)\} \subseteq I'$. Given $(g_1,g_2) \in \ker(\varepsilon_A)$, we have $(g_1 - g_2,\Phi_A(0)) \in \ker(\varepsilon_A)$, and since $\{ g \in \mathcal{C}^{\infty}(\mathbb{R}^{A}) \mid (g,\Phi_A(0)) \in \ker(\varepsilon_A)\} \subseteq I'$, $g_1 - g_2 \in I'$, so $\ker(\varepsilon_A) \subseteq \varphi_A(I')$ and :

$$\varphi_A[\{ I' \in \mathfrak{I}(\mathcal{C}^{\infty}(\mathbb{R}^{A}),\Phi_A) \mid \psi_A(\ker(\varepsilon_A)) \subseteq I'\}] \subseteq \{ R \in {\rm Cong}\,(\mathcal{C}^{\infty}(\mathbb{R}^{A}),\Phi_A) \mid \ker(\varepsilon_A) \subseteq R\}.$$

We have, thus the functions:

$$\begin{array}{cccc}
     \psi_A': & \{ R \in {\rm Cong}\,(\mathcal{C}^{\infty}(\mathbb{R}^{A}),\Phi_A) \mid \ker(\varepsilon_A) \subseteq R\} & \rightarrow & \{ I' \in \mathfrak{I}(\mathcal{C}^{\infty}(\mathbb{R}^{A}), \Phi_A) \mid \psi_A[R] \subseteq I'\} \\
     & R & \mapsto & \varphi_A(R)=\{ g \in \mathcal{C}^{\infty}(\mathbb{R}^{A})| (g,\phi_A(0)) \in R \}
  \end{array}$$

that is, $\psi_A' = \psi_A\upharpoonright_{\{ R \in {\rm Cong}\,(\mathcal{C}^{\infty}(\mathbb{R}^{A}),\Phi_A) \mid \ker(\varepsilon_A) \subseteq R\}}$, and:

$$\begin{array}{cccc}
     \varphi_A': & \{ I' \in \mathfrak{I}(\mathcal{C}^{\infty}(\mathbb{R}^{A}), \Phi_A) \mid \psi_A[R] \subseteq I'\} & \rightarrow &  \{ R \in {\rm Cong}\,(\mathcal{C}^{\infty}(\mathbb{R}^{A}),\Phi_A) \mid \ker(\varepsilon_A) \subseteq R\}\\
     & I' & \mapsto & \varphi_A(I')=\{ (g_1,g_2) \in \mathcal{C}^{\infty}(\mathbb{R}^{A})\times \mathcal{C}^{\infty}(\mathbb{R}^{A}) \mid g_1 - g_2 \in I'\}
  \end{array}$$

where  $$\varphi_A' = \varphi_A\upharpoonright_{\{ I' \in \mathfrak{I}(\mathcal{C}^{\infty}(\mathbb{R}^{A}), \Phi_A) \mid \varphi_A[R] \subseteq I'\}}.$$

Since $\varphi_A'$ and $\psi_A'$ are inverse bijections, we have the following commutative diagram:

$$\xymatrixcolsep{5pc}\xymatrix{
{\rm Cong}\,(\mathcal{C}^{\infty}(\mathbb{R}^{A}),\Phi_A) \ar[r]^{\psi_A} & \mathfrak{I}(\mathcal{C}^{\infty}(\mathbb{R}^{A}),\Phi)\\
\{ R \in {\rm Cong}(\mathcal{C}^{\infty}(\mathbb{R}^{A}),\Phi_A)| \ker(\varepsilon_A) \subseteq R\} \ar@{^{(}->}[u]^{\imath} \ar[r]^{\psi_A'} & \{ I' \in \mathfrak{I}(\mathcal{C}^{\infty}(\mathbb{R}^{A}),\Phi_A) \mid \psi_A(\ker(\varepsilon_A))\subseteq I'\} \ar@{^{(}->}[u]^{\jmath}}$$

where $\imath$ and $\jmath$ are the ordinary inclusions.\\

We also have the bijection given in the \textbf{Lemma \ref{Aracy}}:

$$(q_{\ker(\varepsilon_A)})^{*}: {\rm Cong}\left( \dfrac{\mathcal{C}^{\infty}(\mathbb{R}^{A})}{\ker(\varepsilon_A)}, \overline{\Phi}\right) \rightarrow \{ R \in {\rm Cong}\,(A,\Phi) | \ker(\varepsilon_A) \subseteq R\}$$

and the bijection given in lemmas \textbf{Lemma \ref{lambor}} and \textbf{\ref{Asso}}:

$$(q_{\ker(\varepsilon_A)})^{-}: \mathfrak{I}\left( \dfrac{\mathcal{C}^{\infty}(\mathbb{R}^{A})}{\ker(\varepsilon_A)}, \overline{\Phi}\right) \rightarrow \{ I' \in \mathfrak{I}(\mathcal{C}^{\infty}(\mathbb{R}^{A}), \Phi) | \varphi_A(\ker(\varepsilon_A)) \subseteq I'\}.$$

Hence, we have the following bijection:

$$\psi_{\frac{\mathcal{C}^{\infty}(\mathbb{R}^{A})}{\ker(\varepsilon_A)}}: {\rm Cong}\left( \dfrac{\mathcal{C}^{\infty}(\mathbb{R}^{A})}{\ker(\varepsilon_A)}, \overline{\Phi}\right) \rightarrow \mathfrak{I}\left( \dfrac{\mathcal{C}^{\infty}(\mathbb{R}^{A})}{\ker(\varepsilon_A)}, \overline{\Phi}\right),$$

since the following diagram commutes:

$$\xymatrixcolsep{5pc}\xymatrix{
{\rm Cong}\,(\mathcal{C}^{\infty}(\mathbb{R}^{U(A)}),\Phi) \ar[r]^{\psi_{U(A)}} & \mathfrak{I}(\mathcal{C}^{\infty}(\mathbb{R}^{U(A)}),\Phi)\\
\{ R \in {\rm Cong}(\mathcal{C}^{\infty}(\mathbb{R}^{A}),\Phi_A)| \ker(\varepsilon_A) \subseteq R\} \ar@{^{(}->}[u] \ar[r]^{\psi_A'} & \{ I' \in \mathfrak{I}(\mathcal{C}^{\infty}(\mathbb{R}^{A}),\Phi_A) \mid \varphi_A(\ker(\varepsilon_A))\subseteq I'\}  \ar@{^{(}->}[u] \\
{\rm Cong}\left( \dfrac{\mathcal{C}^{\infty}(\mathbb{R}^{A})}{\ker(\varepsilon_A)}, \overline{\Phi}\right) \ar[u]_{(q_{\ker(\varepsilon_A)})^{*}} \ar[r]_{\psi_{\dfrac{\mathcal{C}^{\infty}(\mathbb{R}^{U(A)})}{\ker(\varepsilon_A)}}} \ar@/^9pc/[uu] & \mathfrak{I}\left( \dfrac{\mathcal{C}^{\infty}(\mathbb{R}^{A})}{\ker(\varepsilon_A)}, \overline{\Phi}\right) \ar[u]^{(q_{\ker(\varepsilon_A)})^{-}} \ar@/_9pc/[uu] }$$

Since $\overline{\varepsilon}_A: \dfrac{\mathcal{C}^{\infty}(\mathbb{R}^{U(A)})}{\ker(\varepsilon_A)} \stackrel{\cong}{\rightarrow} A$ is an isomorphism, we have the bijections:

$${\rm Cong}(A,\Phi) \stackrel{(\overline{\varepsilon}_A)^{*}}{\rightarrow} {\rm Cong}\left( \dfrac{\mathcal{C}^{\infty}(\mathbb{R}^{A})}{\ker(\varepsilon_A)}, \overline{\Phi}\right)$$

and

$$\mathfrak{I}(A,\Phi) \stackrel{(\overline{\varepsilon}_A)^{-}}{\rightarrow} \mathfrak{I}\left( \dfrac{\mathcal{C}^{\infty}(\mathbb{R}^{A})}{\ker(\varepsilon_A)}, \overline{\Phi}\right)$$

so we have the bijection:

$$\psi_{(A,\Phi)}: {\rm Cong}\,(A,\Phi) \to \mathfrak{I}(A,\Phi),$$

since the following diagram commutes:

$$\xymatrixcolsep{5pc}\xymatrix{
{\rm Cong}\left( \dfrac{\mathcal{C}^{\infty}(\mathbb{R}^{A})}{\ker(\varepsilon_A)}, \overline{\Phi}\right) \ar[r]^{\psi_{\frac{\mathcal{C}^{\infty}(\mathbb{R}^{U(A)})}{\ker(\varepsilon_A)}}} & \mathfrak{I}\left( \dfrac{\mathcal{C}^{\infty}(\mathbb{R}^{A})}{\ker(\varepsilon_A)}, \overline{\Phi}\right) \ar[d]^{(\varepsilon_A)^{-}}\\
{\rm Cong}\,(A,\Phi) \ar[u]^{\theta} \ar[r]_{\psi_{(A,\Phi)}} & \mathfrak{I}(A,\Phi)
}$$
\end{proof}

Given any $\mathcal{C}^{\infty}-$ring $(A,\Phi)$, there is a ring-theoretical ideal $I = \psi_A(\ker(\varepsilon_A))$ such that:

$$(A,\Phi) \cong \left( \dfrac{\mathcal{C}^{\infty}(\mathbb{R}^{A})}{I}, \overline{\Phi}\right),$$

that is, every $\mathcal{C}^{\infty}-$ring is the quotient of a free $\mathcal{C}^{\infty}-$ring by some of its ring-theoretic ideals. We say that any $\mathcal{C}^{\infty}-$ring is given by generators and relations.\\

\begin{remark}Let $(A,\Phi)$ be a $\mathcal{C}^{\infty}-$ring. The set ${\rm Cong}\,(A,\Phi)$ is partially ordered by inclusion. Also, given $\{ R_i | i \in I \} \subseteq {\rm Cong}\,(A,\Phi)$, we have:
$$\bigcap_{i \in I}R_i \in {\rm Cong}\,(A,\Phi),$$
so we can define:

$$\begin{array}{cccc}
    \bigwedge : & \mathcal{P}({\rm Cong}(A,\Phi)) & \rightarrow & {\rm Cong} \\
     & \{ R_i \mid i \in I\} & \mapsto & \bigcap_{i \in I}\{ R_i \mid i \in I\}
  \end{array}.$$

Also, given $\{ R_i \mid i \in I\} \subseteq {\rm Cong}\,(A,\Phi)$, we define:

$$\begin{array}{cccc}
    \bigvee : & \mathcal{P}({\rm Cong}(A,\Phi)) & \rightarrow & {\rm Cong} \\
     & \{ R_i \mid i \in I\} & \mapsto & \bigcap \{ R \in {\rm Cong}(A,\Phi) \mid \bigcup_{i \in I}\{ R_i \mid i \in I\} \subseteq R\}
  \end{array},$$

so $({\rm Cong}\,(A,\Phi), \bigwedge, \bigvee)$ is a complete lattice.\\

Note that $\mathfrak{I}(A,\Phi)$, partially ordered by inclusion, also has a structure of complete lattice, since it the set of ring-theoretic ideals of $(A, \Phi(+), \Phi(\cdot), \Phi(-), \Phi(0), \Phi(1))$.\\

We have constructed, in \textbf{Proposition \ref{Meirelles}}, a bijection:

$$\varphi_{(A,\Phi)}: {\rm Cong}\,(A,\Phi) \rightarrow \mathfrak{I}(A,\Phi).$$

$$\begin{array}{cccc}
    \varphi_{(A,\Phi)}: & {\rm Cong}\,(A,\Phi) & \rightarrow & \mathfrak{I}(A,\Phi) \\
     & R & \mapsto & \{ g \in A | (g,0) \in R\}
  \end{array}$$

We claim that $\varphi_{(A,\Phi)}: {\rm Cong}(A,\Phi) \to \mathfrak{I}(A,\Phi)$ is an isomorphism of lattices.\\

Given $\{ R_i | i \in I \} \subseteq {\rm Cong}\,(A,\Phi)$, it is easy to see that:

$$\psi_{(A,\Phi)}\left(\bigwedge \{ R_i | i \in I\} \right) = \bigwedge \psi_{(A,\Phi)}(R_i),$$

so $\psi_{(A,\Phi)}$ is a homomorphism of lattices. Also, given $R,S \in {\rm Cong}\,(A,\Phi)$ such that $R \subseteq S$, we have:

$$\psi_{(A,\Phi)}(R) \subseteq \psi_{(A,\Phi)}(S).$$

Given $I' \supseteq \psi_{(A,\Phi)}(R)$, since $\psi_A'$ is surjective, there is some $S \in {\rm Cong}\,(A,\Phi)$ with $\psi_A'(S) = I'$. Also, since $\psi_A'$ is injective, such an $S$ is unique.\\

Now,

$$\psi_{(A,\Phi)}(R) \subseteq \psi_{(A,\Phi)}(S),$$

so

$$R = \psi_{(A,\Phi)}^{\dashv}[\psi_{(A,\Phi)}(R)] \subseteq {\psi_{(A,\Phi)}}^{\dashv}[\psi_{(A,\Phi)}(S)] = S$$

Since $\psi_{(A, \Phi)}$ is bijective, it follows that $\psi_{(A,\Phi)}$ is an isomorphism of complete lattices. Moreover, both lattices are \textbf{algebraic lattices}, whose compact elements are the finitely generated congruences or ideals.
\end{remark}

The following result relates the ideals of a product of $\mathcal{C}^{\infty}-$rings with the ideals of its factors.\\

\begin{proposition}\label{Gabi}Let $A$ and $B$ be two $\mathcal{C}^{\infty}-$rings and let $\mathfrak{I}(A)$ be the set of all ideals of $A$ and $\mathfrak{I}(B)$ be the set of all ideals of $B$. Every ideal of the product $A \times B$ has the form $\mathfrak{a}\times \mathfrak{b}$, where $\mathfrak{a}$ is an ideal of $A$ and $\mathfrak{b}$ is an ideal of $B$, so we have the following bijection:

$$\begin{array}{cccc}
    \Phi: & \mathfrak{I}(A) \times \mathfrak{I}(B) & \rightarrow & \mathfrak{I}(A \times B) \\
     & (\mathfrak{a},\mathfrak{b}) & \mapsto & \mathfrak{a} \times \mathfrak{b}
  \end{array}$$
\end{proposition}
\begin{proof}
Let $p_1: A\times B \to A$ and $p_2: A \times B \to B$ be the canonical projections of the product:
$$\xymatrixcolsep{5pc}
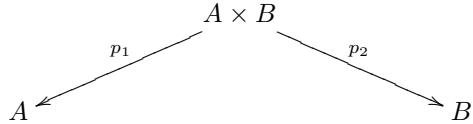
\xymatrix{
 & A \times B \ar[dl]_{p_1} \ar[dr]^{p_2} & \\
A & & B
}$$
and let $I$ be any ideal of $A \times B$. We have:
$$p_1[I] = \{ a \in A | (\exists b \in B)((a,b) \in I)\} = \{ a \in A | (a,0) \in I\}$$
$$p_2[I] = \{ b \in B | (\exists a \in A)((a,b) \in I)\} = \{ b \in B | (0,b) \in I\}$$

We claim that $p_1[I] \times p_2[I] = I$.\\

Given $(a,b) \in I$, then $p_1(a,b) = a \in p_1[I]$ and $p_2(a,b) = b \in p_2[I]$, so $(a,b) \in p_1[I]\times p_2[I]$. Conversely, given $(a,b) \in p_1[I]\times p_2[I]$, $a \in p_1[I]$ and $b \in p_2[I]$, so $(a,0) \in I$ and $(0,b) \in I$, so $(a,b) = (a,0)+(0,b) \in I$.\\

Since for every $\mathfrak{a}$, ideal of $A$ and for every $\mathfrak{b}$ ideal of $B$, $\mathfrak{a} \times \mathfrak{b}$ is an ideal of $A \times B$, the following map is a bijection:

$$\begin{array}{cccc}
    \Phi: & \mathfrak{I}(A) \times \mathfrak{I}(B) & \rightarrow & \mathfrak{I}(A \times B) \\
     & (\mathfrak{a},\mathfrak{b}) & \mapsto & \mathfrak{a} \times \mathfrak{b}
  \end{array}$$

\end{proof}

In the following proposition we are going to describe how to calculate any limit and any directed colimit, making use of the forgetful functor $U: \mathcal{C}^{\infty}{\rm \bf Rng} \rightarrow {\rm \bf Set}$.\\

\begin{proposition}
In the category $\mathcal{C}^{\infty}{\rm \bf Rng}$ of the $\mathcal{C}^{\infty}-$rings, all the limits and all filtered colimits exist and are created by the forgetful functor $U: \mathcal{C}^{\infty}{\rm \bf Rng} \rightarrow {\rm \bf Set}$.
\end{proposition}
\begin{proof}Cf. p. 7 of \cite{MSIA}
\end{proof}

By a general argument, it can be shown that the category $\mathcal{C}^{\infty}{\rm \bf Rng}$ has all small colimits. In particular, coequalizers of pairs of $\mathcal{C}^{\infty}-$homomorphisms, $f,g: (A,\Phi) \to (B,\Psi)$, are given by  quotients:

$$\xymatrix{(A,\Phi) \ar@<1ex>[r]^{f} \ar@<-1ex>[r]_{g} & (B, \Psi) \ar[r]^{q_I} & \left(\dfrac{B}{I}, \overline{\Psi} \right)}$$

where $I = \langle \{(f(a),g(a)) | a \in A\}\rangle$.\\

In order to describe all small colimits, it is enough  to construct coproducts, and since $\mathcal{C}^{\infty}{\rm \bf Rng}$ has filtered colimits, it suffices to construct only finite coproducts. Also, since $\mathbb{R} \cong \mathcal{C}^{\infty}(\{ *\})$ is the initial $\mathcal{C}^{\infty}-$ring, it is enough, by induction, to describe binary coproducts in $\mathcal{C}^{\infty}{\rm \bf Rng}$.\\

\subsection{The $\mathcal{C}^{\infty}-$Coproduct}\label{kop}

\hspace{0.5cm}In this subsection we describe the coproduct in the category $\mathcal{C}^{\infty}{\rm \bf Rng}$, which E. Dubuc calls ``the $\mathcal{C}^{\infty}-$tensor product'', and we call ``the $\mathcal{C}^{\infty}-$coproduct''. First we give its categorial definition, then we define the binary $\mathcal{C}^{\infty}-$coproduct of free, finitely generated and finally the binary $\mathcal{C}^{\infty}-$coproduct of arbitrary $\mathcal{C}^{\infty}-$rings. Then we give a description of the $\mathcal{C}^{\infty}-$coproduct of an arbitrary family of arbitrary $\mathcal{C}^{\infty}-$rings.\\

\begin{definition}\label{cop}Let $(A, \Phi)$ and $(B, \Psi)$  be two  $\mathcal{C}^{\infty}-$rings. We will denote the underlying set of the \index{coproduct}\textbf{coproduct} of $(A, \Phi)$ and $(B, \Psi)$ by $A \otimes_{\infty} B$, and its corresponding canonical arrows by $\iota_A$ and $\iota_B$:
$$\xymatrix{
A \ar[dr]^{\iota_A}& \\
         & A \otimes_{\infty} B \\
B \ar[ur]_{\iota_B} &
}$$
\end{definition}

In order to describe concretely the coproduct in $\mathcal{C}^{\infty}{\rm \bf Rng}$, first we compute the coproduct of two free $\mathcal{C}^{\infty}-$rings with $m$ and $n$ generators.\\

Since $m=\{0, \cdots, m-1 \}$, $n = \{ 0, \cdots, n-1\}$, $m \sqcup n \cong m+n$ and the functor $L: {\rm \bf Set} \rightarrow \mathcal{C}^{\infty}{\rm \bf Rng}$ preserves coproducts (since it is a left adjoint functor), we have:

$$\mathcal{C}^{\infty}(\R^m)\otimes_{\infty}\mathcal{C}^{\infty}(\R^n) \cong \mathcal{C}^{\infty}(\R^{m}\times \R^n) \cong \mathcal{C}^{\infty}(\R^{m+n})$$

Now, given ideals $I \subset \mathcal{C}^{\infty}(\R^m)$ and $J \subset \mathcal{C}^{\infty}(\R^n)$, then:

$$\dfrac{\mathcal{C}^{\infty}(\R^m)}{I}\otimes_{\infty} \dfrac{\mathcal{C}^{\infty}(\R^n)}{J} \cong \dfrac{\mathcal{C}^{\infty}(\R^m \times \R^n)}{(I, J)},$$

where $(I, J) = \langle f \circ \pi_1, g \circ \pi_2 | (f \in I) \& (g \in J)\rangle$, where $\pi_1: \R^m \times \R^n \to \R^m$ and $\pi_2: \R^m \times \R^n \to \R^n$ are the projections on the first and the second coordinates.\\

Now, given any two $\mathcal{C}^{\infty}-$rings, $(A,\Phi)$ and $(B,\Psi)$, we describe concretely their coproduct.\\

First we write $(A,\Phi)$ and $(B,\Psi)$ as colimits of their finitely generated $\mathcal{C}^{\infty}-$sub\-rings:

$$(A,\Phi) \cong \varinjlim_{i \in I} (A_i,\Phi_i)$$
and
$$(B,\Psi) \cong \varinjlim_{j \in I} (B_j,\Psi_j)$$

Then, observing that colimits commute with coproducts, we have:

$$A\otimes_{\infty} B \cong \varinjlim_{\substack{i \in I\\ j \in J}} A_i \otimes_{\infty} B_j$$

Now let $\{ (A_i,\Phi_i) | i \in I\}$ be any set of $\mathcal{C}^{\infty}-$rings. As mentioned in \textbf{Remark \ref{Catarina}} of \textbf{Subsection \ref{DColim}} , such a family has a coproduct in $\mathcal{C}^{\infty}{\rm \bf Rng}$. This coproduct is given by the colimit:

$$\bigotimes_{ \substack{ \infty \\ i \in I}}A_i = \varinjlim_{I' \subseteq_{\rm fin} I} \bigotimes_{\substack{\infty \\ i \in I'}} A_i$$

\subsection{Addition of Variables: The $\mathcal{C}^{\infty}-$Ring of Polynomials}\label{pol}

As an application of the construction given above, we can describe the process of ``adding a set $S$ of variables to a $\mathcal{C}^{\infty}-$ring $(A,\Phi)$''. The construction is given as follows:\\

Let $(A,\Phi)$ be any $\mathcal{C}^{\infty}-$ring and let $S$ be any set. Consider $L(S)= \mathcal{C}^{\infty}(\mathbb{R}^S)$,  the free $\mathcal{C}^{\infty}-$ring on the set $S$ of generators, together with its canonical map, $\jmath_S: S \to \mathcal{C}^{\infty}(\mathbb{R}^S)$. If we denote by:

$$\xymatrixcolsep{5pc}\xymatrix{
A \ar[dr]^{\iota_A} & \\
   & A \otimes_{\infty} \mathcal{C}^{\infty}(\mathbb{R}^S)\\
\mathcal{C}^{\infty}(\mathbb{R}^S) \ar[ur]_{\iota_{\mathcal{C}^{\infty}(\mathbb{R}^S)}}
}$$

the coproduct of $A$ and $\mathcal{C}^{\infty}(\mathbb{R}^S)$, define:
$$x_s := \iota_{\mathcal{C}^{\infty}(\mathbb{R}^S)}(\jmath_S(s)).$$

We thus define:

$$A\{ x_s | s \in S \} := A \otimes_{\infty} \mathcal{C}^{\infty}(\mathbb{R}^S).$$

We have a natural bijection:

$$\dfrac{\mathcal{C}^{\infty}(\mathbb{R}^S) \rightarrow A\otimes_{\infty}\mathcal{C}^{\infty}(\mathbb{R}^S)}{S \rightarrow U(A\otimes_{\infty}\mathcal{C}^{\infty}(\mathbb{R}^S))}$$

\end{document}